\documentclass[a4paper, 10pt]{amsart} \usepackage[utf8]{inputenc}
\usepackage{amsmath, amsthm, amssymb, fancyhdr, enumerate, scalerel, tikz,  color, accents,
  fancybox, rotating, comment, colortbl, stmaryrd, hyperref,  cleveref, setspace, tikz-3dplot} \usepackage[margin=3cm]{geometry}

\title{Infinitesimal \texorpdfstring{$\mathfrak{\lowercase{sl}}_2$}{sl(2)}-symmetries on the equivariant skein lasagna module}

\author{You Qi}
\address{Y.~Q.: Department of Mathematics, University of Virginia,
  Charlottesville, VA 22904, USA}
\email{\href{mailto:yq2dw@virginia.edu}{yq2dw@virginia.edu}}
\author{Louis-Hadrien Robert}
 \address{L.-H.~R.: Université Clermont Auvergne, LMBP, Campus des Cézeaux, 3 place Vasarely, TSA 60026, CS 60026, 63178 Aubière Cedex, France}
 \email{\href{mailto:louis-hadrien.robert@uca.fr}{louis-hadrien.robert@uca.fr}}
 \author{Joshua Sussan}
 \address{J.~S.: Mathematics Program, 
 The Graduate Center, CUNY, New York, NY 10016, USA {\newline} Department of Mathematics, CUNY Medgar Evers, Brooklyn, NY,
   11225, USA} 
  \email{\href{mailto:jsussan@mec.cuny.ed}{jsussan@mec.cuny.edu}}
 \author{Emmanuel Wagner}
 \address{E.~W.: Univ Paris Cit\'e, IMJ-PRG, Univ Paris Sorbonne, UMR 7586 CNRS,
   F-75013, Paris, France} 
 \email{\href{mailto:emmanuel.wagner@imj-prg.fr}{emmanuel.wagner@imj-prg.fr}}
 \author{Paul Wedrich}
 \address{P.~W.: Universität Hamburg, Fachbereich Mathematik, Bundesstraße 55, 20146 Hamburg, Germany} 
 \email{\href{mailto:paul.wedrich@uni-hamburg.de}{paul.wedrich@uni-hamburg.de}}

 \hypersetup{
    pdftoolbar=true,        
    pdfmenubar=true,        
    pdffitwindow=false,     
    pdfstartview={FitH},    
    pdftitle={\shorttitle},    
    pdfauthor={\shortauthors},     
    pdfsubject={\shorttitle},   
    pdfcreator={\shortauthors},   
    pdfproducer={\shortauthors}, 
    pdfkeywords={}, 
    pdfnewwindow=true,      
    colorlinks=true,       
    linkcolor=darkgray,          
    citecolor=teal,        
    filecolor=magenta,      
    urlcolor=violet,          
    linkbordercolor=red,
    citebordercolor=teal,
    urlbordercolor=violet,  
    linktocpage=true
    }
\date{\today}

\usetikzlibrary{arrows, decorations.markings, decorations, patterns,
positioning, decorations.pathreplacing, decorations.pathmorphing,
intersections, fit,   spath3,
  hobby}
\tikzset{->-/.style={decoration={markings, mark=at position .5 with {\arrow{>}}},postaction={decorate}}}
\tikzset{-<-/.style={decoration={markings, mark=at position .5 with {\arrow{<}}},postaction={decorate}}}

\newcounter{res}[section]
\numberwithin{res}{section}
\newtheorem{thm}[res]{Theorem}

\newtheorem{lem}[res]{Lemma}
\newtheorem{lem-dfn}[res]{Lemma-Definition}

\newtheorem{prop}[res]{Proposition}
\newtheorem{cor}[res]{Corollary}

\theoremstyle{definition}
\newtheorem{notation}[res]{Notation}
\newtheorem{dfn}[res]{Definition}
\newtheorem{rmk}[res]{Remark}

\newtheorem{cjc}[res]{Conjecture}

\def\co{\colon\thinspace}

\newcommand{\NB}[1]{\ensuremath{\vcenter{\hbox{#1}}}}

\newcommand{\NN}{\ensuremath{\mathbb{N}}}
\newcommand{\ZZ}{\ensuremath{\mathbb{Z}}}
\newcommand{\CC}{\ensuremath{\mathbb{C}}}
\newcommand{\CP}{\ensuremath{\mathbb{CP}}}

\newcommand{\RR}{\ensuremath{\mathbb{R}}}

\newcommand{\sphere}{\ensuremath{\mathbb{S}}}
\newcommand{\ball}{\ensuremath{\mathbb{B}}}

 \newcommand
{\Id}{\operatorname{Id}} \newcommand{\id}{\mathrm{Id}}

\newcommand{\Hom}{\ensuremath{\mathrm{Hom}}}
\newcommand{\End}{\ensuremath{\mathrm{End}}}
\newcommand{\HOM}{\ensuremath{\mathrm{HOM}}}

\newcommand{\gll}{\ensuremath{\mathfrak{gl}}}
\newcommand{\sll}{\ensuremath{\mathfrak{sl}}}
\newcommand{\dTL}{\mathop{\mathrm{dTL}}\nolimits}
\newcommand{\MdTL}{\mathop{\mathrm{MdTL}}\nolimits}
\newcommand{\colim}{\mathop{\mathrm{Colim}}\nolimits}

\renewcommand{\deg}[2][{}]{\ensuremath{\mathrm{deg}_{#1}(#2)}}

\newcommand{\Links}{\ensuremath{\mathsf{Links}}}

\newcommand{\dLinks}{\ensuremath{\mathsf{dLinks}}}

\newcommand{\qbinom}[2]{\ensuremath
\begin{bmatrix}
  #1 \\
  #2
\end{bmatrix}}

\newcommand{\mHH}{\mathrm{HH}}

 \newcommand{\Sym}{\ensuremath{\mathrm{Sym}}}

\newcommand{\R}{\ensuremath{R}}

\newcommand{\gFm}{\ensuremath{\mathsf{gFoam}}}
\newcommand{\ch}{\ensuremath{\mathrm{Ch}}}
\newcommand{\MGH}{\ensuremath{\mathrm{MGH}}}
\newcommand{\MGS}{\ensuremath{\mathrm{MGS}}}
\newcommand{\MGO}{\ensuremath{\mathrm{MG1}}}
\newcommand{\MGTW}{\ensuremath{\mathrm{MG2}}}
\newcommand{\MGTH}{\ensuremath{\mathrm{MG3}}}
\newcommand{\cab}{\ensuremath{\mathrm{cab}}}

\newcommand{\scalars}{\ensuremath{\Bbbk}}
\newcommand{\de}{\ensuremath{\mathbf{e}}}
\newcommand{\df}{\ensuremath{\mathbf{f}}}
\renewcommand{\dh}{\ensuremath{\mathbf{h}}}
\newcommand{\Le}{\ensuremath{\mathsf{e}}}
\newcommand{\Lf}{\ensuremath{\mathsf{f}}}
\newcommand{\Lh}{\ensuremath{\mathsf{h}}}

\newcommand{\dotnewtoni}[1][i]{\ensuremath{\textcolor[rgb]{0,0,0.6}{\spadesuit_{#1}}}}
\newcommand{\wdotnewtoni}[1][i]{\ensuremath{\textcolor[rgb]{0,0,0.6}{\widehat{\spadesuit}_{#1}}}}

\newcommand{\bracketN}[1]{\left\langle #1 \right\rangle_{\myN}}

\newcommand{\KN}{\ensuremath{\Bbbk_\myN}}

\newcommand{\mymovie}[3][]{
  \NB{
    \begin{tikzpicture}[#1]
\begin{scope}
  \draw[gray, thick] (0, -0.05) -- +(0, 3.1);
  \draw[gray, thick] (4, -0.05) -- +(0, 3.1);
  \draw[gray, thick] (8, -0.05) -- +(0, 3.1);
  \draw[gray, line width=1mm] (0,0) -- +(8,0);
  \draw[white, densely dotted, line width=0.6mm] (0,0) -- +(8,0);
  \draw[gray, line width=1mm] (0,3) -- +(8,0);
  \draw[white, densely dotted, line width=0.6mm] (0,3) -- +(8,0);
  \node (Frame1) at (2, 1.5) {#2};
  \node (Frame1) at (6, 1.5) {#3};
\end{scope}
    \end{tikzpicture}
    }
}

\newcommand{\tone}{{t_1}}
\newcommand{\ttwo}{{t_2}}

\newcommand{\tqftfunc}[1][]{\ensuremath{\mathcal{F}_\myN^{#1}}}
\newcommand{\statespaceN}[2][]{\ensuremath{\tqftfunc[#1]\left(#2\right)}}

\newcommand{\RN}{\ensuremath{\mathbb{Z}_\myN}}

\newcommand{\myN}{\ensuremath{N}}
\newcommand{\web}{\ensuremath{\Gamma}}
\newcommand{\foam}{\ensuremath{F}}
\newcommand{\degN}[1]{\ensuremath{\mathrm{deg}_\myN\left(#1\right)}}
\newcommand{\facet}{\ensuremath{f}}
\newcommand{\surface}{\ensuremath{\Sigma}}

\newcommand{\foamcat}[1][]{\ensuremath{\mathsf{Foam}_{#1}}}

\newcommand{\vectweb}[1]{\ensuremath{V\left(#1\right)}}

\def\mc{\mathcal}

\def\lra{{\longrightarrow}}
\def\dmod{{\mathrm{\mbox{\textrm{-}}mod}}}  
\newcommand{\KR}{\mathrm{KR}}

\newcommand{\Com}{\ensuremath{\mathrm{Com}}}

\def\gmod{{\mathrm{\mbox{-}gmod}}}

\def\Mod{\ensuremath{\mathrm{mod}}}

\newcommand{\gdot}{\ensuremath{\NB{\tikz[thin, green!50!black]{\draw (0,0) circle (0.5mm);}}}}

\newcommand{\gsoliddot}{\ensuremath{\NB{\tikz[thin, green!50!black]{\filldraw[draw= green!50!black, fill = green!70!black] (0,0) circle (0.5mm);}}}}

\newenvironment{nalign}{
    \begin{equation}
    \begin{aligned}
}{
    \end{aligned}
    \end{equation}
    \ignorespacesafterend
}

 \newcommand{\imagesfolder}{.}

\begin{document}
\begin{abstract}
  We construct an \texorpdfstring{$\mathfrak{sl}_2$}{sl(2)}-action on the equivariant
  skein lasagna module.
\end{abstract}

\maketitle

\setcounter{tocdepth}{1}
\tableofcontents

\section{Introduction}
\label{sec:intro}
The skein lasagna module of a smooth 4-manifold was introduced by Morrison, Walker, and Wedrich \cite{MWW1} based on earlier work of Morrison and Walker on blob homology.  A crucial ingredient of this invariant is the functoriality of Khovanov--Rozansky link homology \cite{ETW} and its extension to functoriality in $\sphere^3 \times [0,1]$. 

Manolescu and Neithalath gave a formula purely in terms of Khovanov--Rozansky homology for 2-han\-dle\-bo\-dies \cite{ManNeith}.
This was reinterpreted in terms of a Kirby-type color by Hogancamp, Rose, and Wedrich \cite{HRW}.
Formulas for more general handle attachments were discovered by Manolescu, Walker, and Wedrich \cite{ManWW}.  

The first concrete calculations of this invariant were given by Manolescu and Neithalath \cite{ManNeith} for $\ball^2 \times \sphere^2$ as well as some partial results for other disk bundles over $\sphere^2$.  In particular, they showed that skein lasagna modules for $\CP^2$ and $\overline{\CP^2}$ are different, proving that the invariant is sensitive to orientation reversal.

Sullivan and Zhang \cite{SZ} showed that the invariant vanishes for $\sphere^2 \times \sphere^2$.  This resolved a question posed by Manolescu asking if the invariant was zero or infinite-dimensional for this manifold motivated by the fact that connect summing with $\sphere^2 \times \sphere^2$ turns homeomorphic smooth 4-manifolds into diffeomorphic 4-manifolds.  

In a remarkable recent paper of Ren and Willis \cite{RenWillis}, the authors showed that the skein lasagna module detects new exotic smooth structures.  
Genus bounds for surfaces in 4-manifolds were also recently determined by Morrison, Walker, and Wedrich \cite{MWW2}.

In this work, we extend the $\mathfrak{sl}_2$-action on equivariant Khovanov--Rozansky homology \cite{QRSW3} to the equivariant skein lasagna.  We hope that this extra symmetry will lead to some structural results for the invariant.
The $\mathfrak{sl}_2$-action also allows for some truncation procedures (such as taking the Zuckerman functor) which in some sense can cut down on the infinite-dimensionality of the invariant.
We hope, for example, that these ideas could give rise to an invariant such that for a 1-handle with a link $\sphere^1$ cross 4 points in the boundary, one obtains something finite-dimensional in each tridegree (see \cite[Theorem 1.5]{ManWW}).
Note that working in the equivariant setting makes this calculation more challenging.
We would also like to mention recent work \cite{RSWWZ} which constructs a new skein lasagna module which is locally finite-dimensional.

Extending the $\sll_2$-action to skein lasagna modules requires  compatibility of this action with the equivalence relation at the heart of the definition of such skein modules. This compatibility is ensured by twisting the $\sll_2$-action on Khovanov--Rozansky homology appropriately. Twists are encoded by green-dots. While this may seem artificial, they should be thought of as analogues of grading shifts.

We expect that the $\mathfrak{sl}_2$-action constructed here extends to a partial Witt action using the work of Guerin and Roz \cite{GR}.

In Section \ref{homological:sec} we review the necessary background homological algebra that we use.
In Section \ref{link:sec} we recall notions about foams, link homology and $\mathfrak{sl}_2$-actions developed in \cite{QRSW2, QRSW3}.
Section \ref{skein:sec} contains the definition of the skein lasagna module along with $\mathfrak{sl}_2$-actions on it. We give a concrete realization and explain how the construction fits into a framework of skein lasagna modules with more general enriching categories.
In Section \ref{dTL:sec} the equivariant dotted Temperley--Lieb category is reviewed along with a family of $\mathfrak{sl}_2$-actions on it.  The connection between this category and the skein lasagna was elucidated in \cite{HRW} building on \cite{ManNeith}.  Finally in Section \ref{example:sec} we compute the equivariant skein lasagna for $N=2$ with complex coefficients as an $\mathfrak{sl}_2$-representation for the 4-ball and for $\ball^2 \times \sphere^2$.

\subsection{Conventions}
\label{sec:conventions}
For a ring with unity $\scalars$, and for $x \in \scalars$, we set $\bar{x}=1-x$.

Throughout most of the paper, we will fix a natural number $N$. The
algebras $\RN=\ZZ[X_1, \dots,X_\myN]^{S_\myN}$ and
$\KN = \scalars[X_1, \dots, X_\myN]^{S_\myN}$ of symmetric polynomials
will play central roles in this paper. They are non-negatively graded
by imposing that $\deg{X_i} =2$. The $i$th elementary, complete
homogeneous, and power sum symmetric polynomials in
$X_1,\dots, X_\myN$ are denoted by $E_i$, $H_i$ and $P_i$
respectively, so that
\[
  \RN=\ZZ[E_1, \dots, E_\myN] \qquad \text{and} \qquad \KN=
  \scalars[E_1, \dots, E_\myN].
\]
Throughout most of this paper, we suppose $2$ is invertible in the ground ring.

For $a \in \NN$, $\Sym_{a}$ denotes the ring of symmetric polynomials
in $a$ variables with $\ZZ$ coefficients, in particular
$\RN=\Sym_\myN$. When working in such a ring, we will
let $e_i, h_i$ and $p_i$ be the $i$th elementary, complete
homogeneous, and power sum symmetric polynomials respectively without
reference to the variables.
The ring $\Sym_{a}$ is graded by imposing that $e_i$ is homogeneous of
degree $2i$. With this setting, we have:
\[ \deg{e_i} = \deg{h_i} = \deg {p_i} = 2i.\]

For $n \in \ZZ$, let $[n]=\frac{q^n-q^{-n}}{q-q^{-1}}$, for $k \in
\NN$, we let $[k]!= \prod_{j=1}^k[j]$. Finally, for $m\in \ZZ$ and $a
\in \NN$,  define: 
\[
  \qbinom{m}{a}=\prod_{i=1}^a \frac{[m+1-i]}{[i]}.\]
Note that if $m$ is non-negative, one has $\qbinom{m}{a}= \frac{[m]!}{[a]![m-a]!}$.

For a $\mathbb{Z}$-graded vector space $V$, let $V_i$ denote the
subspace in degree $i$. Let $q^n V$ denote the $\mathbb{Z}$-graded
vector space where $(q^n V)_i=V_{i-n}$.

Complexes will be taken to be cohomologically graded.  For a complex
$C$, we let $t^i C$ denote the shifted complex whose piece in
cohomological degree $i+j$ is the piece of $C$ in cohomological degree
$j$, in other words, $(t^j C)_i = C_{i-j}$.

Foams are read from bottom to top.

\subsection{Acknowledgments}
\label{sec:acknoledgment}
We would like to thank Rostislav Akhmechet, Laura Marino, and Felix Roz for interesting and enlightening conversations.  We would also like to thank Mikhail Khovanov for comments on an earlier draft of the paper.

Part of this research took place at the SwissMAP Research Station and workshops ``Foam Evaluation'' and ``Diagrammatic Categorification" held at ICERM. 

Y.Q.{} is partially supported by the Simons Foundation Collaboration Grants for Mathematicians and the NSF grant DMS-2401376.
LH.R.{} is not supported by any project related funding.
J.S.{} is
partially supported by the the Simons Foundation Collaboration Grants for Mathematicians, NSF grant DMS-2401375 and PSC CUNY Award
64012-00 52.
P.W.{} acknowledges support from the Deutsche
Forschungsgemeinschaft (DFG, German Research Foundation) under Germany's
Excellence Strategy - EXC 2121 ``Quantum Universe'' - 390833306 and the
Collaborative Research Center - SFB 1624 ``Higher structures, moduli spaces and
integrability'' - 506632645.
Some figures are recycled from papers of various subsets of the authors
with or without other collaborators.

\section{Homological background}
\label{homological:sec}

\label{sec:homol-non-sense}

Throughout this work, we will study $H$-modules, where $H$ is some
cocommutative Hopf algebra. We will freely make use of Sweedler's
notation $\Delta(h)=\sum_h h_1\otimes h_2$ and $(\Delta\otimes
\id_H)\circ \Delta(h)=\sum_h h_1\otimes h_2\otimes h_3$, etc.  In the
main results of this paper, the Hopf algebra $H$ is the universal enveloping algebra of $\mathfrak{sl}_2$.

Let $A$ be an $H$-module algebra.  This means that $A$ is an algebra
object in the module category of the Hopf algebra $H$. In particular,
for $h$ in $H$ and $a$ and $b$ in $A$, one has:
\begin{equation}
h\cdot (ab)=\sum_h (h_1\cdot a)(h_2\cdot b), \qquad \text{and,} \qquad h\cdot 1_A = \epsilon(h) 1_A.
\end{equation}
We may then form the \emph{smash product algebra}  $A\# H$.
As an abelian group, $A\# H$ is isomorphic to
$A\otimes H$.  The multiplicative structure is determined by
\begin{equation}
  (a\otimes h)(b\otimes k)=\sum_h a(h_1\cdot b)\otimes h_2 k,
\end{equation}
for any $a,b\in A$ and $h,k\in H$.
Thus, for an $H$-module algebra $A$,  $A\otimes 1$ and $1\otimes H$ sit in
$A\# H$ as subalgebras by construction. 
We often refer to modules and morphisms
in $A\#H\dmod$ 
as \emph{$H$-equivariant} $A$-modules and morphisms.

Since $H$ is a cocommutative Hopf algebra, if $B$ is an $H$-module algebra, then so is its opposite algebra $B^{\mathrm{op}}$. Thus, given two $H$-module algebras $A$ and $B$, their tensor product $A\otimes B^{\mathrm{op}}$ is an $H$-module algebra with its natural $H$-module structure. Following the
definition given above, the multiplication on $A\otimes B^{\mathrm{op}}$ is given by:
\[
  (a_1\otimes b_1 \otimes h )(a_2\otimes b_2 \otimes k) = \sum_h a_1 (h_1\cdot a_2)
  \otimes (h_2\cdot b_1)b_2\otimes h_3k,
\]
for any $a_1,a_2\in A$, $b_1,b_2 \in B$ and $h,k\in H$. A complex in $\mc{C}((A\otimes B^{\mathrm{op}})\# H)$ is also referred to as a complex of $H$-equivariant $(A,B)$-bimodules, and morphisms between such complexes are called $H$-equivariant bimodules homomorphisms.

There is an exact forgetful functor between the usual homotopy
categories of chain complexes of graded modules
\begin{equation}\label{eqn-forgetful-functor}
  \mathrm{For}: \mc{C}(A\# H)\lra \mc{C}(A).
\end{equation}
An object $K_\bullet $ in $\mc{C}(A\# H)$ is annihilated by the forgetful functor if and only if, when forgetting the $H$-module structure
on each term of $K_\bullet$, the complex of graded $A$-modules
$\mathrm{For}(K_\bullet)$ is null-homotopic. The null-homotopy map on
$\mathrm{For}(K_\bullet)$ though, is not required to intertwine
$H$-actions. We refer to such complexes as \emph{relatively null-homotopic}. Also, for ease of notation, we usually just abbreviate the complex notation $K_\bullet$ as $K$ when no confusion can be caused.

\begin{dfn}\label{def-relative-homotopy-category}
Given an $H$-module algebra $A$,
  the \emph{relative homotopy
  category} is the Verdier quotient
  \[\mc{C}^H(A):=\dfrac{\mc{C}(A\#
    H)}{\mathrm{Ker}(\mathrm{For})}.\]
\end{dfn}
The superscript $H$ in the definition is there to remind the reader of the
$H$-module structures on the objects.

The usual homology functor factors through the kernel of $\mathrm{For}$, and induces well-defined homology functors on the relative homotopy category. The homology spaces of objects in $\mc{C}^H(A)$ carry natural actions by the smash product algebra $A\# H$. 

When $A=\Bbbk$ is a field, then the functor $\mathrm{For}:\mc{C}(H)\lra \mc{C}(\Bbbk)$ agrees with the usual total homology functor, since a complex of $H$-modules is acyclic if and only if it is null-homotopic over the ground field. It follows that, in this case $\mc{C}^H(\Bbbk)=\mc{D}(H)$, the usual derived category of $H$-modules. In this paper, we consider ground rings that are more general than fields.

The category $\mc{C}^H(A)$ is triangulated. By construction,
there is a factorization of the forgetful functor
\begin{gather}
  \NB{
    \tikz[xscale=3, yscale =0.8]{
      \node (AH) at (-1, 1) {$\mc{C}(A\# H)$};
      \node (A) at ( 1, 1) {$\mc{C}(A)$};
      \node (Adif) at ( 0, -1) {$\mc{C}^H(A)$};
      \draw[-to] (AH) -- (A) node[pos =0.5, above] {$\mathrm{For}$};
      \draw[-to] (AH) -- (Adif);
      \draw[-to] (Adif) -- (A);      
    }
  }  \ .
\end{gather}

Let us briefly comment on the triangulated structure of the relative homotopy category $\mc{C}^H(A)$. By construction, the homological shift functors, denoted $t^i$ for $i\in \mathbb{Z}$, are inherited from the shift functors on $\mc{C}(A\# H)$, which shift complexes steps to the left or right. 

For the usual homotopy category $\mc{C}(A)$ of an algebra, standard distinguished triangles arise from short exact sequences
  \[
 0 \lra M \stackrel{f}{\lra} N\stackrel{g}{\lra} L \lra 0
  \]
of $A$-modules that are termwise split exact. The class of distinguished triangles in $\mc{C}(A)$ are declared to be those that are isomorphic to standard ones.
For distinguished triangles in the relative homotopy category, similarly, we have the following construction.

\begin{lem} \cite[Lemma 2.3]{QRSW1} \label{lem-construction-of-triangle}
  A short exact sequence of chain complexes of $A\#H$-modules
  \[
 0 \lra M \stackrel{f}{\lra} N \stackrel{g}{\lra} L \lra 0
  \]
 that is termwise $A$-split exact gives rise to a distinguished triangle in $\mc{C}^H(A)$. Conversely, any distinguished triangle in $\mc{C}^H(A)$ is isomorphic to one that arises in this form.
\end{lem}

A morphism of $A\# H$-modules $f:M\lra N$ becomes an isomorphism in $\mc{C}^H(A)$ if and only if its cone $\mathrm{C}(f)$ in $\mc{C}(A\# H)$ is relatively null-homotopic. 
In the remainder of this section, we describe
an enriched structure on the relative homotopy category of an $H$-module algebra. The enriched hom spaces naturally carry structures of $H$-representations. As usual, the enriched hom are naturally right adjoint to tensor product of chain complexes.

Let $A$ be an $H$-module algebra, and $M$, $N$ be two (bounded) complexes of left $A\# H$-modules. 
We set $\HOM_A(M,N)$ to be the usual space of all $A$-linear homomorphisms
\begin{equation}
    \HOM_A(M,N)=\oplus_{i\in \ZZ} \Hom_A(M,t^i N),
\end{equation}
equipped with the natural differential
\begin{equation}
(df)(m)=d(f(m))-(-1)^if(dm)
\end{equation}
for any $f\in \Hom_A(M, t^iN)$.
Then the Hopf algebra $H$ acts on this morphism space by
\begin{equation}\label{eqn-H-action}
    (h \cdot f)(m):=\sum_h h_1 \cdot (f(S(h_2)\cdot m))
\end{equation}
for any $h\in H$, $f\in \HOM_A(M, N)$ and $m\in M$, where $S$ is the antipode of $H$.
The fact that $h \cdot f$ is indeed $A$-linear is non-trivial and follows from the cocomutativity of $H$, see \cite[Section 2]{QRSW3} for more details.

One can also show, as in \cite[Lemma 5.2]{QYHopf}, that, under this $H$-action, the space of $H$-invariants is equal to
\begin{equation}\label{eqn-H-inv-in-HOM}
    \HOM_A(M,N)^H \cong \HOM_{A\# H} (M, N).
\end{equation}

Notice that we restrict the objects inserted into the bifunctor $\HOM_A(\mbox{-},\mbox{-}) $ to those in $\mc{C}(A\# H)$, i.e., those chain complexes of $A$-modules that are $H$-equivariant. When $M$ is an $H$-equivariant complex of $(A,B)$-bimodules, where $B$ is a second $H$-module algebra, we obtain a functor on the usual homotopy categories:
\begin{equation}
\HOM_A(M,\mbox{-}): \mc{C}(A\#H) \lra \mc{C}(B\# H).
\end{equation}

On the other hand, if $M$ is an $H$-equivariant complex of $(A,B)$-bimodules and $K$ is a complex of $H$-equivariant $B$-modules, the natural tensor product
\begin{equation}
    M\otimes_B K= \oplus_{i \in \ZZ}(M\otimes_B K)_i,\quad \quad (M\otimes_B K)_i:=\oplus_{k\in \ZZ} M_k\otimes_B K_{i-k}.
\end{equation}
gives one an $H$-equivariant complex of $A$-modules.
Here, the differential acts by, for any $m\in M_i$ and $k\in K_j$,
\begin{equation}
    d(m\otimes k)=d(m)\otimes k + (-1)^{i}m\otimes d(k);
\end{equation}
while the Hopf algebra $H$ acts by
\begin{equation}
    h\cdot (m\otimes k)=\sum_h (h_1\cdot m)\otimes (h_2\otimes k).
\end{equation}
In this way, tensor product with $M$ over $B$ defines a functor
\begin{equation}
    M\otimes_B(\mbox{-}): \mc{C}(B\#H) \lra \mc{C}(A\#H)
\end{equation}

It is then an easy exercise to check that the usual tensor-hom adjunction preserves the action by the Hopf algebra $H$ (see, for instance, \cite[Lemma 8.5]{QYHopf}):
\begin{equation}
    \HOM_A(M\otimes_B K, N) \cong \HOM_{B}(K,\HOM_A(M,N)).
\end{equation}
Taking $H$-invariants on both sides gives an isomorphism of the graded hom spaces in complexes of $B\# H$-modules and $A\# H$-modules (equation \eqref{eqn-H-inv-in-HOM}). Then taking the zeroth cohomology with respect to the differentials gives us the adjunction in the usual homotopy category:
\begin{equation}\label{eqn-homotopy-adj}
    \Hom_{\mc{C}(A\#H)}(M\otimes_B K, N) \cong \Hom_{\mc{C}(B\#H)}(K,\HOM_A(M,N)).
\end{equation}

The tensor-hom adjunction descends to the relative homotopy category.

\begin{thm} \cite[Theorem 2.5]{QRSW3} \label{thm-tensor-hom-adj}
    Let $M$ be a complex of $H$-equivariant $(A,B)$-bimodules.  Then there is an isomorphism
    \[
    \Hom_{\mc{C}^H(A)}(M\otimes_B K, N) \cong \Hom_{\mc{C}^H(B)}( K, \HOM_A(M, N))
    \]
    functorial in $K$ and $N$.
\end{thm}

\begin{cor} \cite[Corollary 2.6]{QRSW3} \label{cor-hom-as-derived-invariants}
    Let $M$ be a complex of $H$-equivariant $B$-modules. Then
    \[
    \Hom_{\mc{C}^{H}(\Bbbk)}(\Bbbk, \HOM_A(M,N)) \cong \Hom_{\mc{C}^H(A)}(M,N).
    \]
\end{cor}

Again, when $\Bbbk$ is a field, the left-hand side of Corollary \ref{cor-hom-as-derived-invariants} gives us a way of computing morphisms spaces in the relative homotopy category:
\[
  \Hom_{\mc{C}^H(A)}(M,N)  \cong  \Hom_{\mc{D}(H)}(\Bbbk, \HOM_A(M,N)).
\]
The latter space is the zeroth (hyper)cohomology of the complex of $H$-modules $\HOM_A(M,N)$.

\section{Symmetries of link homology}
\label{link:sec}
In this section we review a family of $\mathfrak{sl}_2$-actions  on foams and link homology constructed in \cite{QRSW2} and \cite{QRSW3}.

\subsection{Webs and foams}
\begin{dfn}
  Let $\surface$ be a surface. A \emph{closed web} or simply a
  \emph{web}
is a finite, oriented, trivalent graph
  $\web = (V(\Gamma), E(\Gamma))$ embedded in the interior of
  $\surface$ and endowed with a \emph{thickness function}
  $\ell\co E(\Gamma) \to \NN$ satisfying a flow condition: vertices
  and thicknesses of their adjacent edges must be one of these two types:
  \begin{equation} \label{mergesplit}
    \NB{\tikz[scale=0.6]{\begin{scope}[font=\tiny]
  \begin{scope}
    \draw [-<] (0,0) -- (-90:1) node[pos =1, below] {$a+b$};
    \draw [->] (0,0) -- (30:1) node[pos =1, above] {$b$};
    \draw [->] (0,0) -- (150:1) node[pos =1, above] {$a$};
  \end{scope}
\end{scope}
%%% Local Variables:
%%% mode: latex
%%% TeX-master: t
%%% End:
}} \qquad\text{or}\qquad \NB{\tikz[scale=0.6]{\begin{scope}[font=\tiny]
    \draw [->] (0,0) -- (90:1) node[pos =1, above] {$a+b$};
    \draw [-<] (0,0) -- (-30:1) node[pos =1, below] {$b$};
    \draw [-<] (0,0) -- (-150:1) node[pos =1, below] {$a$};
\end{scope}
%%% Local Variables:
%%% mode: latex
%%% TeX-master: t
%%% End:
}}.
  \end{equation}
  The first type is called a \emph{split} vertex, the second a \emph{merge}
  vertex. In each of these types, there is one \emph{thick} edge and
  two \emph{thin} edges. Oriented circles with non-negative thickness
  are regarded as edges without vertices and can be part of a web. The
  embedding of $\web$ in $\surface$ is smooth outside its vertices,
  and at the vertices should fit with the local models in \eqref{mergesplit}.
\end{dfn}

\begin{figure}
  \centering
  \NB{\tikz[font=\tiny]{\input{\imagesfolder/pf_exa_web}}}
    \caption{Example of a web in $\RR^2$.}
    \label{fig:exa-web}
\end{figure}

\begin{dfn}\label{def:foam}
  Let $M$ be an~oriented smooth 3-manifold with a~collared boundary.
  A~\emph{foam} $\foam \subset M$ is a~collection of \emph{facets},
  that are compact oriented surfaces labeled with non-negative integers
  and glued together along their boundary points, such that every
  point $p$ of $\foam$ has a~closed neighborhood homeomorphic to one
  of the~following:
  \begin{enumerate}
  \item a~disk, when $p$ belongs to a~unique facet,
  \item \label{it:Y}$Y \times [0,1]$, where $Y$ is the neighborhood of
    a~merge or split vertex of a web, when $p$ belongs to three facets, 
  \item the~cone over the~1-skeleton of a~tetrahedron with $p$ as
    the~vertex of the~cone (so that it belongs to six facets).
  \end{enumerate}
  See Figure~\ref{fig:foam-local-model} for a pictorial representation
  of these three cases. The set of points of the~second type is
  a~collection of curves called \emph{bindings} and the~points of
  the~third type are called \emph{singular vertices}.
  The~\emph{boundary} $\partial\foam$ of $\foam$ is the~closure of
  the~set of boundary points of facets that do not belong to
  a~binding. It is understood that $\foam$ coincides with
  $\partial\foam\times[0,1]$ on the~collar of $\partial M$. For each facet
  $\facet$ of $\foam$, we denote by $\ell(\facet)$ its~label, 
  called the \emph{thickness of $\facet$}. A~foam $\foam$ is
  \emph{decorated} if each facet $\facet$ of $\foam$ is assigned
  a~symmetric polynomial $P_f \in \Sym_{\ell(\facet)}$.  In the second local
  model (\ref{it:Y}) of Definition \ref{def:foam}, it is implicitly understood that thicknesses of
  the three facets are given by those of the edges in $Y$. In
  particular, they satisfy a flow condition and locally one has a thick
  facet and two thin ones. We also require that orientations of
  bindings are induced by those of the thin facets and by the opposite
  of the thick facet. Foams are regarded up to ambient isotopy relative to
  boundary. Foams without boundary are said to be \emph{closed}.
\end{dfn}

\begin{rmk}\label{rmk:deco-mult}
    Diagrammatically, decorations on facets are depicted by dots
      placed on facets 
      adorned with symmetric polynomials in the correct number of
      variables (the thickness of the facet they sit on). The
      decoration of a given facet is the product of all adornments of
      dots sitting on that facet, with the natural convention that an empty product is equal to $1$.
\end{rmk}

\begin{figure}[ht]
  \centering
  \NB{\tikz[]{\input{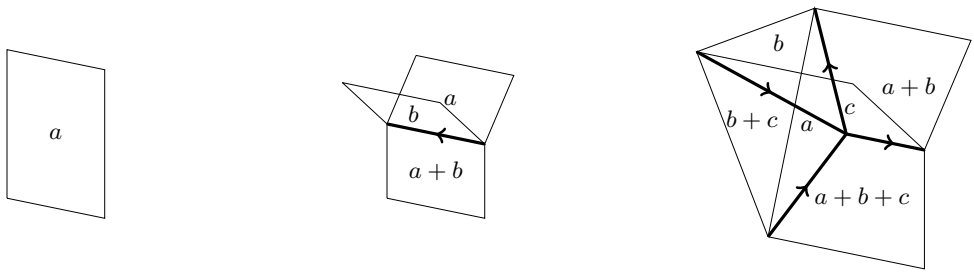}}}
  \caption{The three local models of a foam. Taking into account
    thicknesses, the model in the middle is denoted $Y^{(a,b)}$,
    and the model on the right is denoted $T^{(a,b,c)}$.}
  \label{fig:foam-local-model}
\end{figure}

\begin{notation}
  For a foam $\foam$, we write:
  \begin{itemize}
  \item $\foam^2$ for the~collection of facets of $\foam$,
  \item $\foam^1$ for the collection of bindings,
  \item $\foam^0$ for the set of singular vertices of $\foam$.
  \end{itemize}
  We partition $\foam^1$ as follows:
  $\foam^1= \foam^1_\circ\sqcup \foam^1_{-}$, where $\foam^1_\circ$ is
  the collection of circular bindings and $\foam^1_-$ is the
  collection of bindings diffeomorphic to intervals. If $s \in
  \foam^1_-$, any of its points has a neighborhood diffeomorphic to $Y^{(a,b)}$ for a given $a$ and $b$, and we set:
  \begin{equation}
    \degN{s} = ab + (a+b)(\myN-a-b).\label{eq:deg-binding}
  \end{equation}
  If $v \in
  \foam^0$, it has a neighborhood diffeomorphic to $T^{(a,b,c)}$ and we set:
  \begin{equation}
    \degN{v} = ab +bc + ac + (a+b+c)(\myN-a-b-c).\label{eq:deg-sing}
  \end{equation}
\end{notation}

\begin{dfn}\label{dfn:deg-foam}
  Let $\foam$ be a decorated foam and suppose that all decorations are
  homogeneous. For all $\myN$ in $\NN$, the \emph{$\myN$-degree of
    $\foam$} is the integer $\degN{\foam} \in \ZZ$ given by the
  following formula: 
  \begin{align}\label{eq:deg-foam}
    \degN{\foam}: = & \sum_{\facet \in \foam^2}
    \Big(\deg{P_\facet} -\ell(\facet)(\myN-\ell(\facet))\chi(\facet)
    \Big)  \\
     & + \sum_{s \in \foam^1_-} \degN{s} - \sum_{v \in \foam^0} \degN{v}. \nonumber
  \end{align}
\end{dfn}

The~boundary of a~foam $\foam\subset M$ is a~web in $\partial M$. In the
case $M = \surface\times[0,1]$ is a~thickened surface, a~generic
section $\foam_t := \foam \cap (\surface\times\{t\})$ is
a~web. The~bottom and top webs $\foam_0$ and $\foam_1$ are called the~\emph{input} and \emph{output} of $\foam$ respectively.

If $\surface$ is a surface, $\foamcat[\surface]$ is the category which
has webs in $\surface$ as objects and 
\begin{align*}
  {\Hom}_{\foamcat[\surface]}\left( \web_0,\web_1\right)
  = 
 \left\{
 \begin{array}{c}\text{decorated foams $\foam$ in $\surface\times [0,1]$}\\ \text{with $\foam_i =\web_i$ for $i\in \{0,1\}$}
 \end{array}
 \right\}.
\end{align*}
Composition is given by stacking foams on one another and
rescaling. Decorations behave multiplicatively as suggested by Remark~\ref{rmk:deco-mult}. The identity of $\web$
is $\web\times [0,1]$ decorated by the constant polynomial $1$ on
every facet. The $\myN$-degree of foams is additive under composition.
If $\web$ is a web in a surface $\surface$ and
$h:\surface\times[0,1] \to \surface$ is a smooth isotopy\footnote{For
  the sake of satisfying the collared condition, one should assume
  that $h_t=\id_\surface$ for $t\in [0,\epsilon[\cup]1-\epsilon, 1]$
  for an $\epsilon\in ]0,1]$.} of $\surface$, one can define the foam
$\foam(h)$ to be the trace of $h(\web)$ in $\surface\times[0,1]$: for
all $t \in [0,1]$, $\foam(h)_t=h_t(\web)$. Such foams are called
\emph{traces of isotopies}. They have degree $0$.

\begin{dfn}
 A foam in a surface $\surface\times[0,1]$ is \emph{basic} if it is a trace of isotopy or if it
 is equal to $\web\times[0,1]$ outside a cylinder $B\times[0,1]$, and where
 it is given inside by one of the local models given in Figure~\ref{fig:basic-foam}.  
 \begin{figure}[ht]
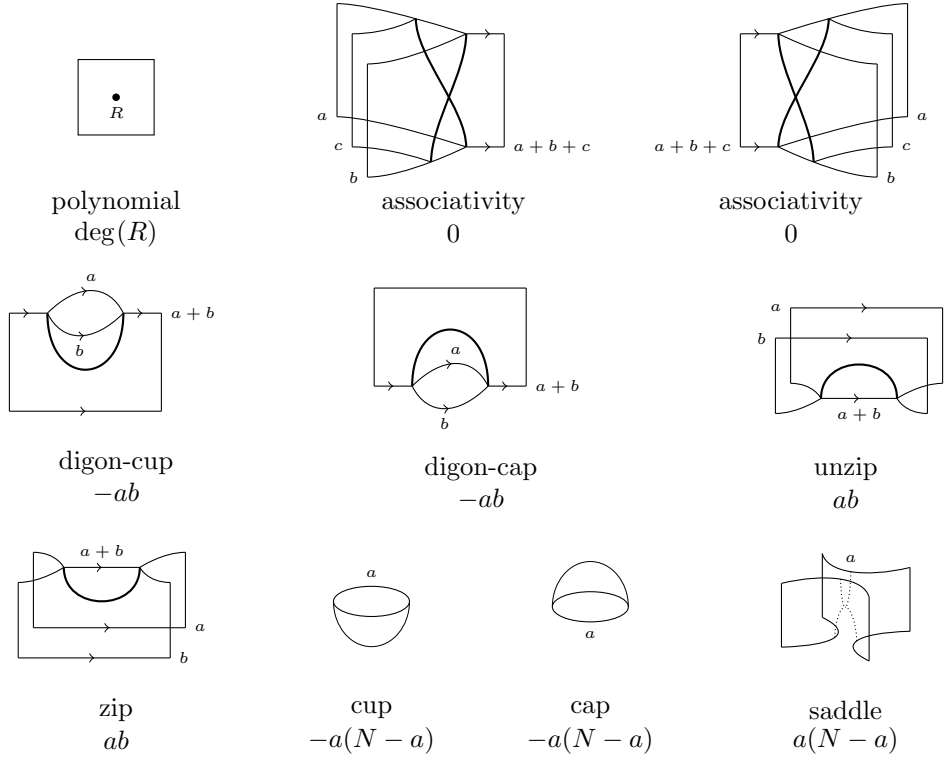

   \centering
   \begin{tikzpicture}[xscale=2.4, yscale =-3.2]
     \node (pol) at (0,0) {\NB{\tikz[font=\tiny]{\input{\imagesfolder/pf_foam-polR}}}};
     \node[yshift= -1.4cm] at (pol) {polynomial};
     \node[yshift = -1.8cm] at (pol) {$\deg{R}$};
     \node (asso) at(1.85,0) {\NB{\tikz[font=\tiny]{\input{\imagesfolder/pf_foam-assoc}}}};
     \node[yshift = -1.4cm] at (asso) {associativity};
     \node[yshift= -1.8cm] at (asso) {$0$};
     \node (coasso) at (3.7,0) {\NB{\tikz[font=\tiny]{\input{\imagesfolder/pf_foam-coassoc}}}};
     \node[yshift = -1.4cm] at (coasso) {associativity};
     \node[yshift= -1.8cm] at (coasso) {$0$};
     \node (digcup) at (0,1) {\NB{\tikz[font=\tiny]{\input{\imagesfolder/pf_foam-digon-cup}}}};
     \node[yshift = -1.64cm] at (digcup) {digon-cup};
     \node[yshift= -2.04cm] at (digcup) {$-ab$};
     \node (digcap) at (2,1.1) {\NB{\tikz[font=\tiny]{\input{\imagesfolder/pf_foam-digon-cap}}}};
     \node[yshift = -1.4cm] at (digcap) {digon-cap};
     \node[yshift= -1.8cm] at (digcap) {$-ab$};
     \node (zip) at (4,1.1) {\NB{\tikz[font=\tiny]{\input{\imagesfolder/pf_foam-unzip}}}};
     \node[yshift = -1.4cm] at (zip) {unzip};
     \node[yshift= -1.8cm] at (zip) {$ab$};
     \node (unzip) at (0,2.1) {\NB{\tikz[font=\tiny]{\input{\imagesfolder/pf_foam-zip}}}};
     \node[yshift = -1.4cm] at (unzip) {zip};
     \node[yshift= -1.8cm] at (unzip) {$ab$};
     \node (cup) at (1.4,2.1) {\NB{\tikz[font=\tiny]{\input{\imagesfolder/pf_foam-cupa}}}};
     \node[yshift = -1.4cm] at (cup) {cup};
     \node[yshift= -1.8cm] at (cup) {$-a(\myN-a)$};
     \node (cap) at (2.6,2.1) {\NB{\tikz[font=\tiny]{\input{\imagesfolder/pf_foam-capa}}}};
     \node[yshift = -1.4cm] at (cap) {cap};
     \node[yshift= -1.8cm] at (cap) {$-a(\myN-a)$};
     \node (saddle) at (4,2.1) {\NB{\tikz[font=\tiny]{\input{\imagesfolder/pf_foam-saddle}}}};
     \node[yshift = -1.4cm] at (saddle) {saddle};
     \node[yshift= -1.8cm] at (saddle) {$a(\myN-a)$};
   \end{tikzpicture}
   \caption{The degree of a basic foam is given below the name of each
     of the local models.}\label{fig:basic-foam}
 \end{figure}

 A foam in $\surface \times [0,1]$ is \emph{in good position} if it is
 a composition of basic foams. 

If $\web$ is a web in $\RR^2$, we denote by $\vectweb{\web}$ the free
$\mathbb{Z}$-module generated by foams in good position in $\RR^2\times
[0,1]$ with $\emptyset$ as input and $\web$ as output.
\end{dfn}

\begin{rmk}
  Every foam in $\surface \times [0,1]$ is isotopic to a foam in good
  position. 
\end{rmk}

In \cite{RW1}, a \emph{$\gll_\myN$-evaluation} of a foam $\foam$ was defined.  This evaluation
$\bracketN{\foam}$ is an element of the ring $\RN$. Its exact formula is not important for this paper and we refer to \cite{RW1} for details of the construction.

Decorations of foams that we consider are often be power sums, and so 
we introduce the following notation: 
\begin{equation}
  \NB{\tikz[scale=1.5, font=\small]{\begin{scope}
  \draw (0,0) rectangle (1,1) coordinate [midway] (A);
  % \fill (A) circle (0.5mm)
  \node at (A) {$\dotnewtoni$};
\end{scope}}} \ =\
  \NB{\tikz[scale=1.5, font=\small]{\begin{scope}
  \draw (0,0) rectangle (1,1) coordinate [midway] (A);
  \fill (A) circle (0.5mm) node[below] {$p_i$};
\end{scope}}}.
\end{equation}
Note that in particular, on a facet of thickness $a$,
$\dotnewtoni[0]=a$. It is convenient to introduce the following notation:
\begin{equation}
  \label{eq:18}
  \NB{\tikz[scale=1.5, font=\small]{\begin{scope}
  \draw (0,0) rectangle (1,1) coordinate [midway] (A);
  % \fill (A) circle (0.5mm)
  \node at (A) {$\wdotnewtoni$};
\end{scope}}} \ =\
  P_i\cdot\ \NB{\tikz[scale=1.5, font=\small]{\begin{scope}
  \draw (0,0) rectangle (1,1) coordinate [midway] (A);
\end{scope}}}\   -\ 
  \NB{\tikz[scale=1.5, font=\small]{}}.
\end{equation}
The symbol $\wdotnewtoni$ can be though of as decoration in an extended sense: using the $\RN$-module structure one can decorate a facets of thickness $\ell$ by elements of $\Sym_\ell\otimes \Sym_{N -\ell}$. For the reader familiar with \cite{RW1}, $\wdotnewtoni$ corresponds to the $i$th power sum in the variables which are not in the facet. 

\subsection{$\mathfrak{sl}_2$-actions} \label{sec:sl2action}
\begin{dfn}
Let $\mathfrak{sl}_2$ be the Lie algebra over $\mathbb{Z}$ generated by $\Le, \Lf, \Lh$ with relations
\[
[\Lh, \Le]=2 \Le, \quad
[\Lh, \Lf]=-2 \Lf, \quad
[\Le, \Lf]= \Lh  \ .
\]
\end{dfn}

We fix parameters $t_1, t_2 \in \scalars$ and define operators $\de$, $\df$ and $\dh$ below on basic foams and extend the action to satisfy
the Leibniz rule with respect to composition of foams. They map traces of isotopies to $0$. The operator $\de$ acts via 
$-\sum_i \frac{\partial}{\partial x_i} $ on
polynomials and by $0$ on any other basic foam. The operators $\dh$ and $\df$ are
defined as follows. 
\allowdisplaybreaks
\begin{gather}
\label{eq:h-act-pol} \dh\left(\NB{\tikz[scale=1.5, font=\small]{\begin{scope}
  \draw (0,0) rectangle (1,1) coordinate [midway] (A); %node[pos=0.2, font=\tiny] {$a$};
  \fill (A) circle (0.5mm) node[below] {$R$};
\end{scope}}}\right) 
  =-\deg{R}\cdot\  \NB{\tikz[scale=1.5, font=\small]{}} \\
\label{eq:h-act-assoc}  \dh\left(\NB{\tikz[scale=0.6, font=\tiny]{\begin{scope}[xscale = -1 ]
  \begin{scope}
    \coordinate (LL) at (0,0);
    \coordinate (L) at (0.5,0);
    \coordinate (R1) at (2.2,0.4);
    \coordinate (R2) at (2,0);
    \coordinate (R3) at (1.8, -0.4);
    \draw[-<-] (LL) -- (L);
    \draw (L) .. controls +(0,0) and +(-0.5, 0) .. (R1)
    coordinate[pos =0.5] (M);
    \draw (L) .. controls +(0,0) and +(-0.5, 0) .. (R3) ;
    \draw (M) .. controls +(0,0) and +(-0.5, 0) .. (R2) ;
  \end{scope}  
 \begin{scope}[yshift = -1.5cm]
    \coordinate (LLB) at (0,0);
    \coordinate (LB) at (0.5,0);
    \coordinate (R1B) at (2.2,0.4);
    \coordinate (R2B) at (2,0);
    \coordinate (R3B) at (1.8, -0.4);
    \draw[-<-] (LLB) -- (LB)  node[right, pos =0] {$a+b+c$};
    \draw (LB) .. controls +(0,0) and +(-0.5, 0) .. (R1B)    node[left] {$a$};
    \draw (LB) .. controls +(0,0) and +(-0.5, 0) .. (R3B)   node[left] {$b$}
    coordinate[pos =0.5] (MB);
    \draw (MB) .. controls +(0,0) and +(-0.5, 0) .. (R2B)    node[left] {$c$};
  \end{scope}  
  \draw (LL) -- (LLB);
  \draw (R1) -- (R1B);
  \draw (R2) -- (R2B);
  \draw (R3) -- (R3B);
  \draw[name path=path1, thick] (M) .. controls +(0,-0.5) and +(0,
  0.5) .. (LB);
  \draw[name path=path2, thick] (L) .. controls +(0,-0.5) and +(0,
  0.5) .. (MB);
  \path [name intersections={of=path1 and path2,by=O}];
  
\end{scope}}}\right)=
  \dh\left(\NB{\tikz[scale=0.6,font=\tiny]{\begin{scope}
  \begin{scope}
    \coordinate (LL) at (0,0);
    \coordinate (L) at (0.5,0);
    \coordinate (R1) at (2.2,0.4);
    \coordinate (R2) at (2,0);
    \coordinate (R3) at (1.8, -0.4);
    \draw[->-] (LL) -- (L);
    \draw (L) .. controls +(0,0) and +(-0.5, 0) .. (R1)
    coordinate[pos =0.5] (M);
    \draw (L) .. controls +(0,0) and +(-0.5, 0) .. (R3) ;
    \draw (M) .. controls +(0,0) and +(-0.5, 0) .. (R2) ;
  \end{scope}  
 \begin{scope}[yshift = -1.5cm]
    \coordinate (LLB) at (0,0);
    \coordinate (LB) at (0.5,0);
    \coordinate (R1B) at (2.2,0.4);
    \coordinate (R2B) at (2,0);
    \coordinate (R3B) at (1.8, -0.4);
    \draw[->-] (LLB) -- (LB)  node[left, pos =0] {$a+b+c$};
    \draw (LB) .. controls +(0,0) and +(-0.5, 0) .. (R1B)    node[right] {$a$};
    \draw (LB) .. controls +(0,0) and +(-0.5, 0) .. (R3B)   node[right] {$b$}
    coordinate[pos =0.5] (MB);
    \draw (MB) .. controls +(0,0) and +(-0.5, 0) .. (R2B)    node[right] {$c$};
  \end{scope}  
  \draw (LL) -- (LLB);
  \draw (R1) -- (R1B);
  \draw (R2) -- (R2B);
  \draw (R3) -- (R3B);
  \draw[name path=path1, thick] (M) .. controls +(0,-0.5) and +(0,
  0.5) .. (LB);
  \draw[name path=path2, thick] (L) .. controls +(0,-0.5) and +(0,
  0.5) .. (MB);
  \path [name intersections={of=path1 and path2,by=O}];  
\end{scope}
}} \right) =0 \\
\label{eq:h-act-dig-cup}  \dh\left( \NB{\tikz[font=\tiny]{\begin{scope}
  \begin{scope}
    \coordinate (L) at (0,0);
    \coordinate (R) at (2,0);
    \coordinate (ML) at (0.5, 0);
    \coordinate (MR) at (1.5, 0);
    \draw[->-] (L) -- (ML);
    \draw[->-] (MR) -- (R) node[right] {$a+b$};
    \draw[->-] (ML).. controls + (0.4, 0.4) and +(-0.2, 0.4) .. (MR)
    node[above, midway] {$a$};
    \draw[->-] (ML).. controls + (0.2, -0.4) and +(-0.4, -0.4) .. (MR)
    node[below, midway] {$b$};
  \end{scope}  
 \begin{scope}[yshift = -1.3cm]
    \coordinate (LB) at (0,0);
    \coordinate (RB) at (2,0);
    \draw[->-] (LB) -- (RB);
  \end{scope}  
  \draw (R) -- (RB);
  \draw (L) -- (LB);
  \draw[thick] (ML) .. controls +(0, -1) and +(0, -1) .. (MR);
\end{scope}}}\right)\ =\
  ab(\tone+\ttwo)\cdot\ \NB{\tikz[font=\tiny]{}} \\
\label{eq:h-act-dig-cap}  \dh\left( \NB{\tikz[font=\tiny]{\begin{scope}
  \begin{scope}
    \coordinate (L) at (0,0);
    \coordinate (R) at (2,0);
    \coordinate (ML) at (0.5, 0);
    \coordinate (MR) at (1.5, 0);
    \draw[->-] (L) -- (ML);
    \draw[->-] (MR) -- (R) node[right] {$a+b$};
    \draw[->-] (ML).. controls + (0.4, 0.4) and +(-0.2, 0.4) .. (MR)
    node[above, midway] {$a$};
    \draw[->-] (ML).. controls + (0.2, -0.4) and +(-0.4, -0.4) .. (MR)
    node[below, midway] {$b$};
  \end{scope}  
 \begin{scope}[yshift = 1.3cm]
    \coordinate (LB) at (0,0);
    \coordinate (RB) at (2,0);
    \draw (LB) -- (RB);
  \end{scope}  
  \draw (R) -- (RB);
  \draw (L) -- (LB);
  \draw[thick] (ML) .. controls +(0, 1) and +(0, 1) .. (MR);
\end{scope}}}\right)\ =\
 ab(\overline{\tone}+\overline{\ttwo})\cdot\  \NB{\tikz[font=\tiny]{}} \\
\label{eq:h-act-dig-zip}  \dh\left( \NB{\tikz[font=\tiny]{\begin{scope}
  \begin{scope}
    \coordinate (L1) at (0.2,0.4);
    \coordinate (L2) at (0,0);
    \coordinate (R1) at (2.2,0.4);
    \coordinate (R2) at (2,0);
    \coordinate (ML) at (0.6, 0.2);
    \coordinate (MR) at (1.6, 0.2);
    \draw[->-] (ML) -- (MR) node[above, midway] {$a+b$};
    \draw (MR) .. controls +(0, 0) and +(-0.3,0) .. (R1) ;
    \draw (MR) .. controls +(0, 0) and +(-0.3,0) .. (R2);
    \draw (L1) .. controls +( 0.3, 0) and +(0,0) .. (ML);
    \draw (L2) .. controls +( 0.3, 0) and +(0,0) .. (ML);
  \end{scope}  
 \begin{scope}[yshift = -1cm]
    \coordinate (L1B) at (0.2,0.4);
    \coordinate (L2B) at (0,0);
    \coordinate (R1B) at (2.2,0.4);
    \coordinate (R2B) at (2,0);
    \draw[->-] (L1B) .. controls +( 0, 0) and +(0,0) .. (R1B) node [right, pos
    = 1] {$a$};
    \draw[->-] (L2B) .. controls +( 0, 0) and +(0,0) .. (R2B) node [right, pos
    = 1] {$b$};
 \end{scope}  
  \draw (R1) -- (R1B);
  \draw (R2) -- (R2B);
  \draw (L1) -- (L1B);
  \draw (L2) -- (L2B);
  \draw[thick] (ML) .. controls +(0, -0.6) and +(0, -0.6) .. (MR);
\end{scope}

%%% Local Variables:
%%% mode: latex
%%% TeX-master: t
%%% End:
}}\right)\ =\
-ab(\overline{\tone} + \overline{\ttwo}) \cdot\ \NB{\tikz[font=\tiny]{}} 
 \\
\label{eq:h-act-dig-unzip} \dh\left( \NB{\tikz[font=\tiny]{\begin{scope}
  \begin{scope}
    \coordinate (L1) at (0.2,0.4);
    \coordinate (L2) at (0,0);
    \coordinate (R1) at (2.2,0.4);
    \coordinate (R2) at (2,0);
    \coordinate (ML) at (0.6, 0.2);
    \coordinate (MR) at (1.6, 0.2);
    \draw[->-] (ML) -- (MR) node[below, midway] {$a+b$};
    \draw (MR) .. controls +(0, 0) and +(-0.3,0) .. (R1) ;
    \draw (MR) .. controls +(0, 0) and +(-0.3,0) .. (R2);
    \draw (L1) .. controls +( 0.3, 0) and +(0,0) .. (ML);
    \draw (L2) .. controls +( 0.3, 0) and +(0,0) .. (ML);
  \end{scope}  
 \begin{scope}[yshift = 1cm]
    \coordinate (L1B) at (0.2,0.4);
    \coordinate (L2B) at (0,0);
    \coordinate (R1B) at (2.2,0.4);
    \coordinate (R2B) at (2,0);
    \draw[->-] (L1B) .. controls +( 0, 0) and +(0,0) .. (R1B) node [left, pos
    = 0] {$a$};
    \draw[->-] (L2B) .. controls +( 0, 0) and +(0,0) .. (R2B) node [left, pos
    = 0] {$b$};
 \end{scope}  
  \draw (R1) -- (R1B);
  \draw (R2) -- (R2B);
  \draw (L1) -- (L1B);
  \draw (L2) -- (L2B);
  \draw[thick] (ML) .. controls +(0, 0.6) and +(0, 0.6) .. (MR);
\end{scope}
}}\right)\ =\
 -ab({\tone} + {\ttwo}) \cdot\ \NB{\tikz[font=\tiny]{}} 
  \\
\label{eq:h-act-cup}  \dh\left( \NB{\tikz[font=\tiny, scale=1.2]{\begin{scope}
  \draw (0,0) arc (180 :0: 0.5cm and 0.2cm) node[above, pos =
  0.5] {$a$};
  \draw[very thin] (0,0) arc (180 :0: 0.5cm and -0.6cm) node[pos=0.5,
  above] {};
  \draw (0,0) arc (180 :0: 0.5cm and -0.2cm);
\end{scope}

%%% Local Variables:
%%% mode: latex
%%% TeX-master: t
%%% End:
}}\right)\ =\
     a(N-a) \cdot\ \NB{\tikz[font=\tiny, scale=1.2]{}}
  \\[3pt]
\label{eq:h-act-cap}  \dh\left( \NB{\tikz[font=\tiny, scale=1.2]{\begin{scope}[-]
  \draw (0,0) arc (180 :0: 0.5cm and 0.2cm);
  \draw[very thin] (0,0) arc (180 :0: 0.5cm and 0.6cm) node[pos=0.5,
  below] {};
  \draw (0,0) arc (180 :0: 0.5cm and -0.2cm)node[ below, pos =
  0.5] {$a$};
\end{scope}

%%% Local Variables:
%%% mode: latex
%%% TeX-master: t
%%% End:
}}\right)\ =\
  a(N-a) \cdot\ \NB{\tikz[font=\tiny,
    scale=1.2]{}} \\
\label{eq:h-act-saddle}  \dh\left( \NB{\tikz[font=\tiny, scale=1.2]{\begin{scope}[scale=0.6]
\tdplotsetmaincoords{70}{25}
\begin{scope}[scale = 1.5, tdplot_main_coords]
  \tikzset{yxplane/.style={canvas is xy plane at z=#1}}
  \begin{scope}[yxplane=1]
    \coordinate (AT) at ({cos(  45)}, {sin(  45)});
    \coordinate (BT) at ({cos(135)}, {sin(135)});
    \coordinate (CT) at ({cos(225)}, {sin(225)});
    \coordinate (DT) at ({cos(315)}, {sin(315)});
    \coordinate (aT) at (  0.3, 0);
    \coordinate (bT) at ( -0.3, 0);
    \draw (AT) .. controls (aT) and (bT) .. (BT)
    coordinate[pos=0.5] (eT) node [pos =0.5, above] {$a$};
    \draw (DT) .. controls (aT) and (bT) .. (CT) coordinate[pos=0.5] (fT);
  \end{scope}

  \begin{scope}[yxplane=0]
    \coordinate (AM) at ({cos(  45)}, {sin(  45)});
    \coordinate (BM) at ({cos(135)}, {sin(135)});
    \coordinate (CM) at ({cos(225)}, {sin(225)});
    \coordinate (DM) at ({cos(315)}, {sin(315)});
    \coordinate (aM) at (0, 0.3);
    \coordinate (bM) at (0, -0.3);
    \draw (AM) .. controls (aM) and (bM) .. (DM) coordinate[pos=0.5] (eM);
    \draw (BM) .. controls (aM) and (bM) .. (CM) coordinate[pos=0.5] (fM);
\end{scope}
  \coordinate (OT) at (0,0, 0.5);
\draw (AM) -- (AT);
  \draw (BM) -- (BT);
  \draw (CM) -- (CT);
  \draw (DM) -- (DT);
  \draw[densely dotted] (eM) ..controls +(0,0,0.2) and + (0.1,0,0).. (OT);
  \draw[densely dotted] (fM) ..controls +(0,0,0.2) and + (-0.1,0,0).. (OT);
  \draw[densely dotted] (eT) ..controls +(0,0,-0.2) and + (0,0.1,0).. (OT);
  \draw[densely dotted] (fT) ..controls +(0,0,-0.2) and + (0,-0.1,0).. (OT);
\end{scope}  
\end{scope}

%%% Local Variables:
%%% mode: latex
%%% TeX-master: t
%%% End:
}}\right)\ =\
    -a(N-a) \cdot\ \NB{\tikz[font=\tiny,
    scale=1.2]{}} \\ 
\label{eq:e-act-pol}  \df\left(\NB{\tikz[scale=1.5, font=\small]{}}\right)
  =\ \NB{\tikz[scale=1.5, font=\small]{\begin{scope}
  \draw (0,0) rectangle (1.5,1) coordinate [midway] (A);
  \fill (A) circle (0.5mm) node[below] {$\sum_i x_i^2 \frac{\partial}{\partial x_i}(R)$};
\end{scope}}} \\
\label{eq:e-act-assoc}   \df\left(\NB{\tikz[scale=0.6, font=\tiny]{}}\right)=
  \df\left(\NB{\tikz[scale=0.6,font=\tiny]{}} \right) =0 \\
\label{eq:e-act-dig-cup}  \df\left( \NB{\tikz[font=\tiny]{}}\right)\ =\
 - \tone\cdot\ \NB{\tikz[font=\tiny]{\begin{scope}[font=\tiny]
  \begin{scope}
    \coordinate (L) at (0,0);
    \coordinate (R) at (2,0);
    \coordinate (ML) at (0.5, 0);
    \coordinate (MR) at (1.5, 0);
    \draw[->-] (L) -- (ML);
    \draw[->-] (MR) -- (R) node[right] {$a+b$};
    \draw[->-] (ML).. controls + (0.4, 0.4) and +(-0.2, 0.4) .. (MR)
    node[above, pos=0.7 ] {$a$} node[below, pos =0.3] {$\dotnewtoni[1]$};
    \draw[->-] (ML).. controls +(0.2, -0.4) and +(-0.4, -0.4) .. (MR)
    node[below, pos =0.3] {$b$} node[below, pos =0.75] {$\dotnewtoni[0]$};
  \end{scope}  
 \begin{scope}[yshift = -1.3cm]
    \coordinate (LB) at (0,0);
    \coordinate (RB) at (2,0);
    \draw[->-] (LB) -- (RB);
  \end{scope}  
  \draw (R) -- (RB);
  \draw (L) -- (LB);
  \draw[thick] (ML) .. controls +(0, -1) and +(0, -1) .. (MR);
\end{scope}

%%% Local Variables:
%%% mode: latex
%%% TeX-master: t
%%% End:
}} 
  - \ttwo\cdot \ \NB{\tikz[font=\tiny]{\begin{scope}[font=\tiny]
  \begin{scope}
    \coordinate (L) at (0,0);
    \coordinate (R) at (2,0);
    \coordinate (ML) at (0.5, 0);
    \coordinate (MR) at (1.5, 0);
    \draw[->-] (L) -- (ML);
    \draw[->-] (MR) -- (R) node[right] {$a+b$};
    \draw[->-] (ML).. controls + (0.4, 0.4) and +(-0.2, 0.4) .. (MR)
    node[above, pos=0.7 ] {$a$} node[below, pos =0.3] {$\dotnewtoni[0]$};
    \draw[->-] (ML).. controls +(0.2, -0.4) and +(-0.4, -0.4) .. (MR)
    node[below, pos =0.3] {$b$} node[below, pos =0.75] {$\dotnewtoni[1]$};
  \end{scope}  
 \begin{scope}[yshift = -1.3cm]
    \coordinate (LB) at (0,0);
    \coordinate (RB) at (2,0);
    \draw[->-] (LB) -- (RB);
  \end{scope}  
  \draw (R) -- (RB);
  \draw (L) -- (LB);
  \draw[thick] (ML) .. controls +(0, -1) and +(0, -1) .. (MR);
\end{scope}

%%% Local Variables:
%%% mode: latex
%%% TeX-master: t
%%% End:
}}  \\
\label{eq:e-act-dig-cap}   \df\left( \NB{\tikz[font=\tiny]{}}\right)\ =\
   - \overline{\tone}
  \cdot\ \NB{\tikz[font=\tiny]{\begin{scope}[font =\tiny]
  \begin{scope}
    \coordinate (L) at (0,0);
    \coordinate (R) at (2,0);
    \coordinate (ML) at (0.5, 0);
    \coordinate (MR) at (1.5, 0);
    \draw[->-] (L) -- (ML);
    \draw[->-] (MR) -- (R) node[right] {$a+b$};
    \draw[->-] (ML).. controls + (0.4, 0.4) and +(-0.2, 0.4) .. (MR)
    node[above, pos =0.7] {$a$}     node[above, pos =0.3] {$\dotnewtoni[1]$};
    \draw[->-] (ML).. controls + (0.2, -0.4) and +(-0.4, -0.4) .. (MR)
    node[left, pos =0.3] {$b$} node[above, pos =0.7, yshift = -0.5mm] {$\dotnewtoni[0]$};
  \end{scope}  
 \begin{scope}[yshift = 1.3cm]
    \coordinate (LB) at (0,0);
    \coordinate (RB) at (2,0);
    \draw (LB) -- (RB);
  \end{scope}  
  \draw (R) -- (RB);
  \draw (L) -- (LB);
  \draw[thick] (ML) .. controls +(0, 1) and +(0, 1) .. (MR);
\end{scope}

%%% Local Variables:
%%% mode: latex
%%% TeX-master: t
%%% End:
}}  
  \!\!\!\! - \overline{\ttwo}
  \cdot \ \NB{\tikz[font=\tiny]{\begin{scope}[font =\tiny]
  \begin{scope}
    \coordinate (L) at (0,0);
    \coordinate (R) at (2,0);
    \coordinate (ML) at (0.5, 0);
    \coordinate (MR) at (1.5, 0);
    \draw[->-] (L) -- (ML);
    \draw[->-] (MR) -- (R) node[right] {$a+b$};
    \draw[->-] (ML).. controls + (0.4, 0.4) and +(-0.2, 0.4) .. (MR)
    node[above, pos =0.7] {$a$}     node[above, pos =0.3] {$\dotnewtoni[0]$};
    \draw[->-] (ML).. controls + (0.2, -0.4) and +(-0.4, -0.4) .. (MR)
    node[left, pos =0.3] {$b$} node[above, pos =0.7, yshift = -0.5mm] {$\dotnewtoni[1]$};
  \end{scope}  
 \begin{scope}[yshift = 1.3cm]
    \coordinate (LB) at (0,0);
    \coordinate (RB) at (2,0);
    \draw (LB) -- (RB);
  \end{scope}  
  \draw (R) -- (RB);
  \draw (L) -- (LB);
  \draw[thick] (ML) .. controls +(0, 1) and +(0, 1) .. (MR);
\end{scope}

%%% Local Variables:
%%% mode: latex
%%% TeX-master: t
%%% End:
}}  \\
\label{eq:e-act-zip}   \df\left( \NB{\tikz[font=\tiny]{}}\right)\ =\
    \overline{\tone} \cdot\ \NB{\tikz[font=\tiny]{\begin{scope}
  \begin{scope}
    \coordinate (L1) at (0.2,0.4);
    \coordinate (L2) at (0,0);
    \coordinate (R1) at (2.2,0.4);
    \coordinate (R2) at (2,0);
    \coordinate (ML) at (0.6, 0.2);
    \coordinate (MR) at (1.6, 0.2);
    \draw[->-] (ML) -- (MR) node[above, midway] {$a+b$};
    \draw (MR) .. controls +(0, 0) and +(-0.3,0) .. (R1) ;
    \draw (MR) .. controls +(0, 0) and +(-0.3,0) .. (R2);
    \draw (L1) .. controls +( 0.3, 0) and +(0,0) .. (ML);
    \draw (L2) .. controls +( 0.3, 0) and +(0,0) .. (ML);
  \end{scope}  
 \begin{scope}[yshift = -1cm]
    \coordinate (L1B) at (0.2,0.4);
    \coordinate (L2B) at (0,0);
    \coordinate (R1B) at (2.2,0.4);
    \coordinate (R2B) at (2,0);
    \draw[->-] (L1B) .. controls +( 0, 0) and +(0,0) .. (R1B) node [right, pos
    = 1] {$a$};
    \draw[->-] (L2B) .. controls +( 0, 0) and +(0,0) .. (R2B) node [right, pos
    = 1] {$b$}   node [pos = 0.2, above] {$\dotnewtoni[0]$};
 \end{scope}  
  \draw (R1) -- (R1B) node [pos = 0.2, left] {$\dotnewtoni[1]$};
  \draw (R2) -- (R2B);
  \draw (L1) -- (L1B);
  \draw (L2) -- (L2B);
  \draw[thick] (ML) .. controls +(0, -0.6) and +(0, -0.6) .. (MR);
\end{scope}

%%% Local Variables:
%%% mode: latex
%%% TeX-master: t
%%% End:
}} 
  +  \overline{\ttwo}\cdot \ \NB{\tikz[font=\tiny]{\begin{scope}
  \begin{scope}
    \coordinate (L1) at (0.2,0.4);
    \coordinate (L2) at (0,0);
    \coordinate (R1) at (2.2,0.4);
    \coordinate (R2) at (2,0);
    \coordinate (ML) at (0.6, 0.2);
    \coordinate (MR) at (1.6, 0.2);
    \draw[->-] (ML) -- (MR) node[above, midway] {$a+b$};
    \draw (MR) .. controls +(0, 0) and +(-0.3,0) .. (R1) ;
    \draw (MR) .. controls +(0, 0) and +(-0.3,0) .. (R2);
    \draw (L1) .. controls +( 0.3, 0) and +(0,0) .. (ML);
    \draw (L2) .. controls +( 0.3, 0) and +(0,0) .. (ML);
  \end{scope}  
 \begin{scope}[yshift = -1cm]
    \coordinate (L1B) at (0.2,0.4);
    \coordinate (L2B) at (0,0);
    \coordinate (R1B) at (2.2,0.4);
    \coordinate (R2B) at (2,0);
    \draw[->-] (L1B) .. controls +( 0, 0) and +(0,0) .. (R1B) node [right, pos
    = 1] {$a$};
    \draw[->-] (L2B) .. controls +( 0, 0) and +(0,0) .. (R2B) node [right, pos
    = 1] {$b$}   node [pos = 0.2, above] {$\dotnewtoni[1]$};
 \end{scope}  
  \draw (R1) -- (R1B) node [pos = 0.2, left] {$\dotnewtoni[0]$};
  \draw (R2) -- (R2B);
  \draw (L1) -- (L1B);
  \draw (L2) -- (L2B);
  \draw[thick] (ML) .. controls +(0, -0.6) and +(0, -0.6) .. (MR);
\end{scope}

%%% Local Variables:
%%% mode: latex
%%% TeX-master: t
%%% End:
}}  \\
\label{eq:e-act-unzip}   \df\left( \NB{\tikz[font=\tiny]{}}\right)\ =\
    {\tone}\cdot\ \NB{\tikz[font=\tiny]{\begin{scope}
  \begin{scope}
    \coordinate (L1) at (0.2,0.4);
    \coordinate (L2) at (0,0);
    \coordinate (R1) at (2.2,0.4);
    \coordinate (R2) at (2,0);
    \coordinate (ML) at (0.6, 0.2);
    \coordinate (MR) at (1.6, 0.2);
    \draw[->-] (ML) -- (MR) node[below, midway] {$a+b$};
    \draw (MR) .. controls +(0, 0) and +(-0.3,0) .. (R1) ;
    \draw (MR) .. controls +(0, 0) and +(-0.3,0) .. (R2);
    \draw (L1) .. controls +( 0.3, 0) and +(0,0) .. (ML);
    \draw (L2) .. controls +( 0.3, 0) and +(0,0) .. (ML) node [above,
    pos=0.2] {$\dotnewtoni[0]$};
  \end{scope}  
 \begin{scope}[yshift = 1cm]
    \coordinate (L1B) at (0.2,0.4);
    \coordinate (L2B) at (0,0);
    \coordinate (R1B) at (2.2,0.4);
    \coordinate (R2B) at (2,0);
    \draw[->-] (L1B) .. controls +( 0, 0) and +(0,0) .. (R1B) node [left, pos
    = 0] {$a$} node
  [pos=0.8, below] {$\dotnewtoni[1]$};
    \draw[->-] (L2B) .. controls +( 0, 0) and +(0,0) .. (R2B) node [left, pos
    = 0] {$b$};
 \end{scope}  
  \draw (R1) -- (R1B);
  \draw (R2) -- (R2B);
  \draw (L1) -- (L1B);
  \draw (L2) -- (L2B);
  \draw[thick] (ML) .. controls +(0, 0.6) and +(0, 0.6) .. (MR);
\end{scope}

%%% Local Variables:
%%% mode: latex
%%% TeX-master: t
%%% End:
}} 
   +  {\ttwo}\cdot \ \NB{\tikz[font=\tiny]{\begin{scope}
  \begin{scope}
    \coordinate (L1) at (0.2,0.4);
    \coordinate (L2) at (0,0);
    \coordinate (R1) at (2.2,0.4);
    \coordinate (R2) at (2,0);
    \coordinate (ML) at (0.6, 0.2);
    \coordinate (MR) at (1.6, 0.2);
    \draw[->-] (ML) -- (MR) node[below, midway] {$a+b$};
    \draw (MR) .. controls +(0, 0) and +(-0.3,0) .. (R1) ;
    \draw (MR) .. controls +(0, 0) and +(-0.3,0) .. (R2);
    \draw (L1) .. controls +( 0.3, 0) and +(0,0) .. (ML);
    \draw (L2) .. controls +( 0.3, 0) and +(0,0) .. (ML) node [above,
    pos=0.2] {$\dotnewtoni[1]$};
  \end{scope}  
 \begin{scope}[yshift = 1cm]
    \coordinate (L1B) at (0.2,0.4);
    \coordinate (L2B) at (0,0);
    \coordinate (R1B) at (2.2,0.4);
    \coordinate (R2B) at (2,0);
    \draw[->-] (L1B) .. controls +( 0, 0) and +(0,0) .. (R1B) node [left, pos
    = 0] {$a$} node
  [pos=0.8, below] {$\dotnewtoni[0]$};
    \draw[->-] (L2B) .. controls +( 0, 0) and +(0,0) .. (R2B) node [left, pos
    = 0] {$b$};
 \end{scope}  
  \draw (R1) -- (R1B);
  \draw (R2) -- (R2B);
  \draw (L1) -- (L1B);
  \draw (L2) -- (L2B);
  \draw[thick] (ML) .. controls +(0, 0.6) and +(0, 0.6) .. (MR);
\end{scope}

%%% Local Variables:
%%% mode: latex
%%% TeX-master: t
%%% End:
}}  
  \\
\label{eq:e-act-cup}   \df\left( \NB{\tikz[font=\tiny, scale=1.2]{}}\right)\ =\
     - \frac{1}{2}  \cdot\ \NB{\tikz[font=\tiny, scale=1.2]{\begin{scope}
  \draw (0,0) arc (180 :0: 0.5cm and 0.2cm) node[above, pos =
  0.5] {$a$};
  \draw[very thin] (0,0) arc (180 :0: 0.5cm and -0.6cm) node[pos=0.5,
  above] {$\dotnewtoni[0] \wdotnewtoni[1]$};
  \draw (0,0) arc (180 :0: 0.5cm and -0.2cm);
\end{scope}

%%% Local Variables:
%%% mode: latex
%%% TeX-master: t
%%% End:
}}
    \ 
   -\   \frac{1}{2}  \cdot\ \NB{\tikz[font=\tiny, scale=1.2]{\begin{scope}
  \draw (0,0) arc (180 :0: 0.5cm and 0.2cm) node[above, pos =
  0.5] {$a$};
  \draw[very thin] (0,0) arc (180 :0: 0.5cm and -0.6cm) node[pos=0.5,
  above] {$\dotnewtoni[1] \wdotnewtoni[0]$};
  \draw (0,0) arc (180 :0: 0.5cm and -0.2cm);
\end{scope}

%%% Local Variables:
%%% mode: latex
%%% TeX-master: t
%%% End:
}}  \\[3pt]
\label{eq:e-act-cap}   \df\left( \NB{\tikz[font=\tiny, scale=1.2]{}}\right)\ =\
    -\frac{1}{2} \cdot\ \NB{\tikz[font=\tiny,
    scale=1.2]{\begin{scope}
  \draw (0,0) arc (180 :0: 0.5cm and 0.2cm);
  \draw[very thin] (0,0) arc (180 :0: 0.5cm and 0.6cm) node[pos=0.5,
  below] {$\dotnewtoni[0] \wdotnewtoni[1]$};
  \draw (0,0) arc (180 :0: 0.5cm and -0.2cm)node[ below, pos =
  0.5] {$a$};
\end{scope}

%%% Local Variables:
%%% mode: latex
%%% TeX-master: t
%%% End:
}} \ 
  - \  \frac{1}{2}  \cdot\ \NB{\tikz[font=\tiny, scale=1.2]{\begin{scope}
  \draw (0,0) arc (180 :0: 0.5cm and 0.2cm);
  \draw[very thin] (0,0) arc (180 :0: 0.5cm and 0.6cm) node[pos=0.5,
  below] {$\dotnewtoni[1] \wdotnewtoni[0]$};
  \draw (0,0) arc (180 :0: 0.5cm and -0.2cm)node[ below, pos =
  0.5] {$a$};
\end{scope}

%%% Local Variables:
%%% mode: latex
%%% TeX-master: t
%%% End:
}} \\
  \label{eq:e-act-saddle}   \df\left( \NB{\tikz[font=\tiny, scale=1.2]{}}\right)\ =\
    \frac{1}{2} \cdot\ \NB{\tikz[font=\tiny,
    scale=1.2]{\begin{scope}[scale=0.6]
\tdplotsetmaincoords{70}{25}
\begin{scope}[scale = 1.5, tdplot_main_coords]
  \tikzset{yxplane/.style={canvas is xy plane at z=#1}}
  \begin{scope}[yxplane=1]
    \coordinate (AT) at ({cos(  45)}, {sin(  45)});
    \coordinate (BT) at ({cos(135)}, {sin(135)});
    \coordinate (CT) at ({cos(225)}, {sin(225)});
    \coordinate (DT) at ({cos(315)}, {sin(315)});
    \coordinate (aT) at (  0.3, 0);
    \coordinate (bT) at ( -0.3, 0);
    \draw (AT) .. controls (aT) and (bT) .. (BT)
    coordinate[pos=0.5] (eT) node [pos =0.5, above] {$a$};
    \draw (DT) .. controls (aT) and (bT) .. (CT) coordinate[pos=0.5] (fT);
  \end{scope}
  \node at (-0.75, -0.1, 0.3) {$\dotnewtoni[0]$};
    \node at (0.5, 0.5, 0.3) {$\wdotnewtoni[1]$};
  \begin{scope}[yxplane=0]
    \coordinate (AM) at ({cos(  45)}, {sin(  45)});
    \coordinate (BM) at ({cos(135)}, {sin(135)});
    \coordinate (CM) at ({cos(225)}, {sin(225)});
    \coordinate (DM) at ({cos(315)}, {sin(315)});
    \coordinate (aM) at (0, 0.3);
    \coordinate (bM) at (0, -0.3);
    \draw (AM) .. controls (aM) and (bM) .. (DM) coordinate[pos=0.5] (eM);
    \draw (BM) .. controls (aM) and (bM) .. (CM) coordinate[pos=0.5] (fM);
\end{scope}
  \coordinate (OT) at (0,0, 0.5);
\draw (AM) -- (AT);
  \draw (BM) -- (BT);
  \draw (CM) -- (CT);
  \draw (DM) -- (DT);
  \draw[densely dotted] (eM) ..controls +(0,0,0.2) and + (0.1,0,0).. (OT);
  \draw[densely dotted] (fM) ..controls +(0,0,0.2) and + (-0.1,0,0).. (OT);
  \draw[densely dotted] (eT) ..controls +(0,0,-0.2) and + (0,0.1,0).. (OT);
  \draw[densely dotted] (fT) ..controls +(0,0,-0.2) and + (0,-0.1,0).. (OT);
\end{scope}  
\end{scope}

%%% Local Variables:
%%% mode: latex
%%% TeX-master: t
%%% End:
}} \ 
  + \  \frac{1}{2} \cdot\ \NB{\tikz[font=\tiny, scale=1.2]{\begin{scope}[scale=0.6]
\tdplotsetmaincoords{70}{25}
\begin{scope}[scale = 1.5, tdplot_main_coords]
  \tikzset{yxplane/.style={canvas is xy plane at z=#1}}
  \begin{scope}[yxplane=1]
    \coordinate (AT) at ({cos(  45)}, {sin(  45)});
    \coordinate (BT) at ({cos(135)}, {sin(135)});
    \coordinate (CT) at ({cos(225)}, {sin(225)});
    \coordinate (DT) at ({cos(315)}, {sin(315)});
    \coordinate (aT) at (  0.3, 0);
    \coordinate (bT) at ( -0.3, 0);
    \draw (AT) .. controls (aT) and (bT) .. (BT)
    coordinate[pos=0.5] (eT) node [pos =0.5, above] {$a$};
    \draw (DT) .. controls (aT) and (bT) .. (CT) coordinate[pos=0.5] (fT);
  \end{scope}
  \node at (-0.75, -0.1, 0.3) {$\dotnewtoni[1]$};
    \node at (0.5, 0.5, 0.3) {$\wdotnewtoni[0]$};
  \begin{scope}[yxplane=0]
    \coordinate (AM) at ({cos(  45)}, {sin(  45)});
    \coordinate (BM) at ({cos(135)}, {sin(135)});
    \coordinate (CM) at ({cos(225)}, {sin(225)});
    \coordinate (DM) at ({cos(315)}, {sin(315)});
    \coordinate (aM) at (0, 0.3);
    \coordinate (bM) at (0, -0.3);
    \draw (AM) .. controls (aM) and (bM) .. (DM) coordinate[pos=0.5] (eM);
    \draw (BM) .. controls (aM) and (bM) .. (CM) coordinate[pos=0.5] (fM);
\end{scope}
  \coordinate (OT) at (0,0, 0.5);
\draw (AM) -- (AT);
  \draw (BM) -- (BT);
  \draw (CM) -- (CT);
  \draw (DM) -- (DT);
  \draw[densely dotted] (eM) ..controls +(0,0,0.2) and + (0.1,0,0).. (OT);
  \draw[densely dotted] (fM) ..controls +(0,0,0.2) and + (-0.1,0,0).. (OT);
  \draw[densely dotted] (eT) ..controls +(0,0,-0.2) and + (0,0.1,0).. (OT);
  \draw[densely dotted] (fT) ..controls +(0,0,-0.2) and + (0,-0.1,0).. (OT);
\end{scope}  
\end{scope}

%%% Local Variables:
%%% mode: latex
%%% TeX-master: t
%%% End:
}} 
\end{gather}
\allowdisplaybreaks[0]

Note that if $t_1+t_2=1$, then $\dh$ acts a negative degree operator.
We now record the fact that the formulas above do indeed give rise to an $\mathfrak{sl}_2$-action.
\begin{lem} \cite[Lemma 4.6]{QRSW2} \label{lem:sl2-acts-good-position}
  Mapping $\Le$ to $\de$, $\Lh$ to $\dh$ and $\Lf$ to $\df$ defines an
  action of $\sll_2$ on the $\scalars$-module generated by
  foams in good position.
\end{lem}

The actions of $\de, \df, \dh$ defined above descend to an action on symmetric polynomials.
They are compatible with the $\mathfrak{sl}_2$-action on symmetric polynomials via foam evaluation.

\begin{prop}[{\cite[Proposition 4.3]{QRSW2}}]\label{prop:sl2-acts-good-pos}
  Let $\foam$ be a closed foam in good position, then for 
  $\mathbf{x} \in \mathfrak{sl}_2$: \begin{equation} \bracketN{\mathbf{x} \cdot \foam}= \mathbf{x} \cdot \bracketN{\foam} .\end{equation}
\end{prop}

\subsection{Green dotted webs and state spaces}
\label{subsec:greendottedwebs}
Using foam evaluation (as in \cite{Khsl3, RW1, QRSW2}), one can associate a
$\RN$-module, called a {\em $\gll_N$-state space $\statespaceN{\web}$} to each web
$\web$. Actually $\statespaceN{\mbox{-}}$ is a functor. It associates to foams, regarded a cobordisms between webs, $\RN$-module maps.
To prevent overloaded diagrams, we omit
$\statespaceN{\cdot}$ when dealing with actual webs or foams.

The $\mathfrak{sl}_2$-action on foams gives rise to an
$\mathfrak{sl}_2$-action on the state space $\statespaceN{\web}$.
Proposition \ref{prop:sl2-acts-good-pos} guarantees that this action
is well-defined.

For each edge $e$ with label $a$ in a web $\web$, the
  algebra 
  \begin{equation}
  D_e = \scalars[x_1, \dots, x_a,
  y_1,\dots,y_{N-a}]^{S_a\times S_{N-a}}
  \end{equation}
  acts on the $\gll_N$-state space $\statespaceN{\web}$ associated to $\web$ by adding a decoration on
  the facet bounding to the edge $e$. Consequently the algebra
  \begin{equation}
      D_\web:=\bigotimes_{e\in E(\web)} D_e 
  \end{equation} acts on
  $\statespaceN{\web}$. For each edge $e$, the algebra $D_e$ is
  endowed with a natural $\mathfrak{sl}_2$-module structure and is an
  $\mathfrak{sl}_2$-module algebra and thus so is $D_\web$. 

It is possible to twist the $\mathfrak{sl}_2$-action on
  $\statespaceN{\web}$.  This idea goes back to \cite{KRWitt} where
  the Witt-type action was twisted in order to make differentials
  occuring in link homology equivariant with respect to the Witt
  action.  Similar ideas were then used in \cite{QRSW1,QiSussanLink}
  in order to categorify the colored Jones polynomial in the context
  of a $p$-DG structure.  The twisting of the Witt-type action on
  foams was described in \cite{QRSW2, GR}.  Here we review this twisting
  in the simplified setting of an $\mathfrak{sl}_2$-action.

  Let $\foam\co \web_1 \to \web_2$ be a foam.  Then
  $\tqftfunc\left(\foam\right)$ induces
a $\RN$-linear map from $\statespaceN{\web_1}$ to
  $\statespaceN{\web_2}$. This map
  intertwines the action of $\mathfrak{sl}_2$ if and only if
  $\tqftfunc\left(g \cdot\foam\right) =0$ for all
  $g\in \mathfrak{sl}_2$. This prevents many foams from being
  morphisms in categories taking into account the
  $\mathfrak{sl}_2$-action. Having in mind the construction of a link
  homology with an $\mathfrak{sl}_2$-action, we must circumvent this
  problem.  We thus introduce twists on state spaces. 
We follow the diagrammatic formalism of \cite{QRSW2}.

  \begin{dfn}
    A \emph{green-dotted} web is a web $\web$ endowed with
a finite collection $D$ of \emph{green dots}, that are marked
    points with multiplicities (in $\scalars$) located in the interior
    of edges of $\web$. These green dots are of two types $\gdot$ and
$ \gsoliddot$.  If a given edge carries several green dots of the
    same type, they may be replaced by one green dot of that type on
    this edge with the sum of all multiplicities. See \cite[Example 3.12]{QRSW3} for an illustration of a green-dotted web.
      
\end{dfn}

  Let $(\web, D)$ be a green-dotted web. For each
  green dot $d$ of multiplicity $\lambda \in \scalars$, define
  $\web_d $ to be the foam $\web\times [0,1]$ with a twisted action of
  $\mathfrak{sl}_2$.  Each green dot lives on an edge, hence each
  foam bounding $\web$ has a neighborhood of that green dots
  homeomorphic to $]0,1[\times [0,1[$. The twists induced by green
  dots are local and we depict the modified $\sll_2$-action on these
  neighborhoods. Recall that to act on a concrete foam, one uses the
  Leibniz rule, so that the twists induced by various green-dot on a
  given web add up.
\allowdisplaybreaks
\begin{gather}
\label{eq:e-act-pol-twisthollow} \de\left(\NB{\tikz[scale=1.5, font=\small]{\begin{scope}
  \draw[thin] (0,0.5) rectangle (1,1);
  \draw[thick] (0,1) -- (1,1)  coordinate [pos=0.4] (A);
  \filldraw[draw= green!50!black, fill = white] (A) circle (.5mm)
  node[below,  green!50!black] {$\lambda$};
\end{scope}}}\right) 
  =0 \\
\label{eq:e-act-pol-twistsolid}  \de\left(\NB{\tikz[scale=1.5, font=\small]{\begin{scope}
  \draw[thin] (0,0.5) rectangle (1,1);
  \draw[thick] (0,1) -- (1,1)  coordinate [pos=0.4] (A);
  \filldraw[draw= green!50!black, fill = green!70!black] (A) circle (.5mm)
  node[below,  green!50!black] {$\lambda$};
\end{scope}
}}\right)=
0 \\
\label{eq:h-act-pol-twisthollow} \dh\left(\NB{\tikz[scale=1.5, font=\small]{}}\right) 
  =-\lambda\cdot\  \NB{\tikz[scale=1.5, font=\small]{\begin{scope}
  \draw[thin] (0,0.5) rectangle (1,1);
  \draw[thick] (0,1) -- (1,1)  coordinate [pos=0.4] (A);
  \filldraw[draw= green!50!black, fill = white] (A) circle (.5mm)
  node[below,  green!50!black] {$\lambda$};
  \node (B) at (.75,.75) {$\dotnewtoni[0]$};
\end{scope}

%\begin{scope}
%  \draw (0,0) rectangle (1,1) coordinate [midway] (A);
%  \fill (A) 
%  %circle (0.5mm) 
%  node[below] {$\dotnewtoni[0]$};
%\end{scope}
}} \\
\label{eq:h-act-pol-twistsolid}  \dh\left(\NB{\tikz[scale=1.5, font=\small]{}}\right)=
-\lambda \cdot\  \NB{\tikz[scale=1.5, font=\small]{\begin{scope}
  \draw[thin] (0,0.5) rectangle (1,1);
  \draw[thick] (0,1) -- (1,1)  coordinate [pos=0.4] (A);
  \filldraw[draw= green!50!black, fill = green!70!black] (A) circle (.5mm)
  node[below,  green!50!black] {$\lambda$};
  \node (B) at (.75,.75)  {$\wdotnewtoni[0]$};
\end{scope}
}} \\
\label{eq:f-act-pol-twisthollow} \df\left(\NB{\tikz[scale=1.5, font=\small]{}}\right) 
  =\lambda\cdot\  \NB{\tikz[scale=1.5, font=\small]{\begin{scope}
  \draw[thin] (0,0.5) rectangle (1,1);
  \draw[thick] (0,1) -- (1,1)  coordinate [pos=0.4] (A);
  \filldraw[draw= green!50!black, fill = white] (A) circle (.5mm)
  node[below,  green!50!black] {$\lambda$};
  \node (B) at (.75,.75) {$\dotnewtoni[1]$};
\end{scope}

%\begin{scope}
%  \draw (0,0) rectangle (1,1) coordinate [midway] (A);
%  \fill (A) 
%  %circle (0.5mm) 
%  node[below] {$\dotnewtoni[0]$};
%\end{scope}
}} \\
\label{eq:f-act-pol-twistsolid}  \df\left(\NB{\tikz[scale=1.5, font=\small]{}}\right)=
\lambda \cdot\  \NB{\tikz[scale=1.5, font=\small]{\begin{scope}
  \draw[thin] (0,0.5) rectangle (1,1);
  \draw[thick] (0,1) -- (1,1)  coordinate [pos=0.4] (A);
  \filldraw[draw= green!50!black, fill = green!70!black] (A) circle (.5mm)
  node[below,  green!50!black] {$\lambda$};
  \node (B) at (.75,.75)  {$\wdotnewtoni[1]$};
\end{scope}
}}  .
\end{gather}
\allowdisplaybreaks[0]

It is convenient to introduce floating green dots on the plane: one
can view them as being on edges of thickness $0$. A floating hollow
green dot $\gdot$ does not alter the Lie algebra action since in this context
$\dotnewtoni[0]=\dotnewtoni[1]=0$. However a solid green dot
$\gsoliddot$ twists the action of $\df$ by $E_1$ and the action of
$\dh$ by $-N$ (because $\wdotnewtoni[0]=N$ and $\wdotnewtoni[1]=E_1$)

With this new convention, one has the following local relation:

\[
\NB{\tikz[]{\begin{scope}
  \coordinate (m) at (  0,  0);
  \coordinate (t) at (  2, 0);
  \draw[->] (m) -- (t) node[pos = 1, right] {$a$} coordinate[pos=0.3]
  (ga) coordinate[pos=0.7] (gb) ;
  \filldraw[draw= green!50!black, fill = green!70!black] (ga) circle (1mm)
  node[yshift= 2mm, green!50!black] {$r$};
  \filldraw[draw= green!50!black, fill = white] (gb) circle (1mm)
  node[yshift=2mm, green!50!black] {$r$};
\end{scope}
%%% Local Variables:
%%% mode: latex
%%% TeX-master: t
%%% End:
}} = \NB{\tikz[]{\begin{scope}
  \coordinate (m) at (  0,  0);
  \coordinate (t) at (  2, 0);
  \draw[->] (m) -- (t) node[pos = 1, right] {$a$} coordinate [pos=0.5,
  yshift = 2mm]
  (ga);
  \filldraw[draw= green!50!black, fill = green!70!black, above] (ga) circle (1mm)
  node[left, green!50!black] {$r$};
  % \filldraw[draw= green!50!black, fill = white] (gb) circle (1mm)
  % node[above, green!50!black] {$r$};
\end{scope}
%%% Local Variables:
%%% mode: latex
%%% TeX-master: t
%%% End:
}}.
\]

\subsection{More about twists}
\renewcommand{\gg}{\ensuremath{\mathfrak{g}}}

The
$\gll_N$-state spaces are endowed with an $\sll_2$-action and a
$D_\web$-action. These two actions intertwine nicely. This allows us to
twist the first one by the second as we shall see. In order to show
what structures are at play, we  formulate this in a slightly more
general context.

Let $\gg$ be a Lie algebra, $A$ a commutative $\gg$-module algebra and $M$ an
$A\# \gg$-module\footnote{In the notation of Section \ref{homological:sec}, it should really be $A\# \mc{U}(\gg)$, where  $\mc{U}(\gg)$ denotes the universal enveloping algebra of $\gg$. However, we write this smash product algebra as $A\# \gg$ for simplicity.}, that is, a $\gg$-module with an $A$-module structure, which
satisfies the following identities for all $g$ in $\gg$, $a$ in $A$
and $m$ in $M$:
\[g\cdot_{\gg}(a\cdot_{A} m) = (g\cdot a)\cdot_A m +
  a\cdot_{A}(g\cdot_{\gg} m).\]
In our context, $\gg =\sll_2$, $M$ is a $\gll_N$-state space of a web $\web$
 and $A$ is $D_\web$, the polynomial algebra of decorations of $\Gamma$.

A linear map $\tau\co \gg \to A$ is \emph{flat} if for all $g_1, g_2$
in $\gg$, one has:
\begin{equation}
\tau([g_1,g_2]) = g_1\cdot \tau(g_2) - g_2\cdot \tau(g_1).
\end{equation}
If $\tau$ is flat, one can check that there is a new $A\# \gg$-module structure on the rank-one free module $A$ itself defined by
\begin{equation}
    g\cdot_{\gg^\tau} a:= g\cdot_\gg a+\tau(g)a.
\end{equation}
More generally, the following formula
defines a new $\gg$-action on any $A\# \gg$-module $M$ via:
\begin{equation}
g\cdot_{\gg^\tau} m := g\cdot_{\gg} m + \tau(g)\cdot_A m. 
\end{equation}
One also checks that the twisted action $\cdot_{\gg^\tau}$, similar to $\cdot_\gg$, is compatible with the $A$-action and gives $M$ a new $A\#\gg$-module
structure. If $\tau$ is flat, we denote by $M^{\tau}$ the module endowed with the
action $\cdot_{\gg^{\tau}}$.
One can readily see that, as $A\# \gg$-modules,
\begin{equation}
    M^\tau = A^\tau\otimes_A M.
\end{equation}
This means, in particular, that if
$\tau, \sigma \co \gg \to A$ are flat, then 
$(M^\tau)^\sigma \cong M^{\tau+\sigma}$.

  \begin{prop} \cite[Proposition 3.13]{QRSW3}\label{prop:classification-sl2-twists}
    For any green-dotted web $(\web, D)$,
    twisting the actions of $\de,\df,$ and $\dh$ as above
    endows $\statespaceN{\web}$ with an $\mathfrak{sl}_2$-module structure.
  \end{prop}

Alternatively, if $e$ is an edge of thickness $a$ in a web $\web$, the
  algebra 
  $$D_e = \scalars[x_1, \dots, x_a,
  y_1,\dots,y_{N-a}]^{S_a\times S_{N-a}}$$ 
  is an $\sll_2$-module algebra with
   $\de$, $\df$ and $\dh$ acting by differential operators
  \begin{equation}
      \de  =-\sum_{i=1}^a \frac{\partial}{\partial x_i}-\sum_{j=1}^{N-a} \frac{\partial}{\partial y_j}, \quad \quad 
      \df  =\sum_{i=1}^a x_i^2\frac{\partial}{\partial x_i}+\sum_{j=1}^{N-a}y_j^2 \frac{\partial}{\partial y_j},
  \end{equation}
  \begin{equation}
 \dh =-2\sum_{i=1}^a x_i\frac{\partial}{\partial x_i}-2\sum_{j=1}^{N-a} y_j\frac{\partial}{\partial y_j}.
  \end{equation}
  Proposition \ref{prop:classification-sl2-twists} shows that there is a $2$-parameter $\sll_2$-equivariant $D_e$-bimodule structure on the rank-one module $D_e \cdot v_{\alpha,\beta}$, for some $\alpha,\beta\in \Bbbk$. Here $\de$, $\df$ and $\dh$ act on the module generator $v_{\alpha,\beta}$ by
 \begin{equation} \label{eqn-sl2-twist-on-mod}
     \de (v_{\alpha,\beta})=0, \quad \dh (v_{\alpha,\beta})=-(a\alpha+(N-a)\beta)v_{\alpha,\beta}, \quad \df(v_{\alpha,\beta})= (\alpha \dotnewtoni[1] + \beta \wdotnewtoni[1])v_{\alpha,\beta}.
   \end{equation}   
This equivariant bimodule gives rise to the twist endo-functor on $D_e\# \sll_2$-modules as the tensor product
$   (D_e\cdot v_{\alpha,\beta})\otimes_{D_e}(\mbox{-}) $.

Basic properties of symmetric functions allow green dots to migrate
along a web in the following ways, which we record for later use.  See also \cite[Lemma
3.11]{QRSW1}. These manipulation of green dots are referred to as
\emph{green dot migration}.
  \begin{equation}
    \NB{\tikz[scale=0.45,font=\tiny]{\begin{scope}
  \coordinate (m) at (  0,  0);
  \coordinate (t) at (  0, 1);
  \coordinate (br) at (+.5,  -1);
  \coordinate (bl) at (-.5,  -1);
  \draw[>-] (bl) .. controls +(0,0.5) and + (0, 0) .. (m) node[pos =
  0, below] {$a$} coordinate[pos = 0.3] (ga);
  \draw[>-] (br) .. controls +(0,0.5) and + (0, 0) .. (m) node[pos =
  0, below] {$b$} coordinate[pos = 0.3] (gb);
  \draw[->] (m) -- (t) node[pos = 1, above] {$a+b$};
  \filldraw[draw= green!50!black, fill = white] (ga) circle (1mm)
  node[left, green!50!black] {$r$};
  \filldraw[draw= green!50!black, fill = white] (gb) circle (1mm)
  node[right, green!50!black] {$r$};
\end{scope}}} =
    \NB{\tikz[scale=0.45,font=\tiny]{\begin{scope}
  \coordinate (m) at (  0,  0);
  \coordinate (t) at (  0, 1);
  \coordinate (br) at (+.5,  -1);
  \coordinate (bl) at (-.5,  -1);
  \draw[>-] (bl) .. controls +(0,0.5) and + (0, 0) .. (m) node[pos =
  0, below] {$a$} coordinate[pos = 0.3] (ga);
  \draw[>-] (br) .. controls +(0,0.5) and + (0, 0) .. (m) node[pos =
  0, below] {$b$} coordinate[pos = 0.3] (gb);
  \draw[->] (m) -- (t) node[pos = 1, above] {$a+b$} coordinate[pos = 0.7] (gc);
  %\filldraw[draw= green!50!black, fill = white] (ga) circle (1mm)
  %node[left, green!50!black] {$r$};
  %\filldraw[draw= green!50!black, fill = white] (gb) circle (1mm)
 % node[right, green!50!black] {$r$};
    \filldraw[draw= green!50!black, fill = white] (gc) circle (1mm)
  node[right, green!50!black] {$r$};
\end{scope}}} , \quad 
   \NB{\tikz[scale=0.45,font=\tiny]{\begin{scope}
  \coordinate (m) at (  0,  0);
  \coordinate (b) at (  0, -1);
  \coordinate (tr) at (+.5,  1);
  \coordinate (tl) at (-.5,  1);
  \draw[->] (m) .. controls +(0,0) and + (0, -0.5) .. (tl) node[pos =
  1, above] {$a$} coordinate[pos = 0.7] (ga);
  \draw[->] (m) .. controls +(0,0) and + (0, -0.5) .. (tr) node[pos =
  1, above] {$b$} coordinate[pos = 0.7] (gb);
  \draw[>-] (b) -- (m) node[pos = 0, below] {$a+b$} coordinate[pos = 0.3] (gc);
  %\filldraw[draw= green!50!black, fill = white] (ga) circle (1mm)
  %node[left, green!50!black] {$r$};
  %\filldraw[draw= green!50!black, fill = white] (gb) circle (1mm)
  %node[right, green!50!black] {$r$};
   \filldraw[draw= green!50!black, fill = white] (gc) circle (1mm)
  node[left, green!50!black] {$r$};
\end{scope}}} =
   \NB{\tikz[scale=0.45,font=\tiny]{\begin{scope}
  \coordinate (m) at (  0,  0);
  \coordinate (b) at (  0, -1);
  \coordinate (tr) at (+.5,  1);
  \coordinate (tl) at (-.5,  1);
  \draw[->] (m) .. controls +(0,0) and + (0, -0.5) .. (tl) node[pos =
  1, above] {$a$} coordinate[pos = 0.7] (ga);
  \draw[->] (m) .. controls +(0,0) and + (0, -0.5) .. (tr) node[pos =
  1, above] {$b$} coordinate[pos = 0.7] (gb);
  \draw[>-] (b) -- (m) node[pos = 0, below] {$a+b$};
  \filldraw[draw= green!50!black, fill = white] (ga) circle (1mm)
  node[left, green!50!black] {$r$};
  \filldraw[draw= green!50!black, fill = white] (gb) circle (1mm)
  node[right, green!50!black] {$r$};
\end{scope}}} ,\quad
   \NB{\tikz[scale=0.45,font=\tiny]{\begin{scope}
  \coordinate (m) at (  0,  0);
  \coordinate (t) at (  0, 1);
  \coordinate (br) at (+.5,  -1);
  \coordinate (bl) at (-.5,  -1);
  \draw[>-] (bl) .. controls +(0,0.5) and + (0, 0) .. (m) node[pos =
  0, below] {$a$} coordinate[pos = 0.3] (ga);
  \draw[>-] (br) .. controls +(0,0.5) and + (0, 0) .. (m) node[pos =
  0, below] {$b$} coordinate[pos = 0.3] (gb);
  \draw[->] (m) -- (t) node[pos = 1, above] {$a+b$};
  \filldraw[draw= green!50!black, fill = green] (ga) circle (1mm)
  node[left, green!50!black] {$r$};
  % \filldraw[draw= green!50!black, fill = green] (gb) circle (1mm)
  % node[right, green!50!black] {$r$};
\end{scope}}} =
   \NB{\tikz[scale=0.45,font=\tiny]{\begin{scope}
  \coordinate (m) at (  0,  0);
  \coordinate (t) at (  0, 1);
  \coordinate (br) at (+.5,  -1);
  \coordinate (bl) at (-.5,  -1);
  \draw[>-] (bl) .. controls +(0,0.5) and + (0, 0) .. (m) node[pos =
  0, below] {$a$} coordinate[pos = 0.3] (ga);
  \draw[>-] (br) .. controls +(0,0.5) and + (0, 0) .. (m) node[pos =
  0, below] {$b$} coordinate[pos = 0.3] (gb);
  \draw[->] (m) -- (t) node[pos = 1, above] {$a+b$} coordinate[pos = 0.7] (gc);
  %\filldraw[draw= green!50!black, fill = white] (ga) circle (1mm)
  %node[left, green!50!black] {$r$};
  \filldraw[draw= green!50!black, fill=white] (gb) circle (1mm)
  node[right, green!50!black] {$r$};
  \filldraw[draw= green!50!black, fill = green] (gc) circle (1mm)
  node[right, green!50!black] {$r$};
\end{scope}}} , \quad
  \NB{\tikz[scale=0.45,font=\tiny]{\begin{scope}
  \coordinate (m) at (  0,  0);
  \coordinate (b) at (  0, -1);
  \coordinate (tr) at (+.5,  1);
  \coordinate (tl) at (-.5,  1);
  \draw[->] (m) .. controls +(0,0) and + (0, -0.5) .. (tl) node[pos =
  1, above] {$a$} coordinate[pos = 0.7] (ga);
  \draw[->] (m) .. controls +(0,0) and + (0, -0.5) .. (tr) node[pos =
  1, above] {$b$} coordinate[pos = 0.7] (gb);
  \draw[>-] (b) -- (m) node[pos = 0, below] {$a+b$};
  \filldraw[draw= green!50!black, fill = green] (ga) circle (1mm)
  node[left, green!50!black] {$r$};
  % \filldraw[draw= green!50!black, fill = green] (gb) circle (1mm)
  % node[right, green!50!black] {$r$};
\end{scope}}} =
  \NB{\tikz[scale=0.45,font=\tiny]{\begin{scope}
  \coordinate (m) at (  0,  0);
  \coordinate (b) at (  0, -1);
  \coordinate (tr) at (+.5,  1);
  \coordinate (tl) at (-.5,  1);
  \draw[->] (m) .. controls +(0,0) and + (0, -0.5) .. (tl) node[pos =
  1, above] {$a$} coordinate[pos = 0.7] (ga);
  \draw[->] (m) .. controls +(0,0) and + (0, -0.5) .. (tr) node[pos =
  1, above] {$b$} coordinate[pos = 0.7] (gb);
  \draw[>-] (b) -- (m) node[pos = 0, below] {$a+b$} coordinate[pos = 0.3] (gc);
  %\filldraw[draw= green!50!black, fill = white] (ga) circle (1mm)
  %node[left, green!50!black] {$r$};
  \filldraw[draw= green!50!black, fill = white] (gb) circle (1mm)
  node[right, green!50!black] {$r$};
   \filldraw[draw= green!50!black, fill = green] (gc) circle (1mm)
  node[left, green!50!black] {$r$};
\end{scope}}} .
\end{equation}

 \subsection{Link homology}
As usual for Khovanov-like link homology theories, we associate a
hypercube shaped complex to any link diagram by specifying locally a
length 2 complex to any crossing. We define the following
cohomologically graded braiding complexes:

\begin{equation} \label{eqn:def-T}
  T= \NB{\tikz[xscale = 0.6]{\begin{scope}[font=\tiny]
  \draw[->] (0.5, -0.5) ..controls +(0,0.3) and +(0,-0.3) .. (-0.5,
  0.5);% node[pos=1, above] {} coordinate[pos =0.2] (t1);
  \fill[white] (0,0) circle (2mm);
  \draw[->] (-0.5, -0.5) ..controls +(0,0.3) and +(0,-0.3) .. (0.5,
  0.5);% node[pos=1, above] {} coordinate[pos =0.2] (t2);
  % \filldraw[draw= green!50!black, fill = white] (t2) circle (1mm)
  % node[left, green!50!black] {$c$};
  % \filldraw[draw= green!50!black, fill = white] (t1) circle (1mm)
  % node[right, green!50!black] {$d$};
\end{scope}}} :=
    \NB{\tikz[xscale = 3.5, yscale = 3]{
    \node (i0) at (0, 0) {$q^{-1}$ \NB{\tikz[font= \tiny,
  scale=0.6]{\begin{scope}
  \coordinate (bl) at (-0.5, -1);
  \coordinate (br) at ( 0.5, -1);
  \coordinate (tl) at (-0.5,  1);
  \coordinate (tr) at ( 0.5,  1);
    \coordinate (ml) at (-0.5,  .6);
 \coordinate (mr) at (0.5,  .6);
 \coordinate (G) at (0.8,-0.1);

   \draw[->] (bl) -- (tl);
        \filldraw[draw= green!50!black, fill = white] (ml) circle (1mm) 
  node[left, green!50!black] 
  {$-t_1$};
  
    \draw[->] (br) -- (tr);
        \filldraw[draw= green!50!black, fill = white] (mr) circle (1mm) 
  node[right, green!50!black] 
  {$-t_2$};

\end{scope}}} };
     \node (i1) at (-1, 0) {\NB{\tikz[font= \tiny,
  scale=0.6]{\begin{scope}
  \coordinate (bl) at (-0.5, -1);
  \coordinate (br) at ( 0.5, -1);
  \coordinate (bm) at (  0,-0.3);
  \coordinate (tl) at (-0.5,  1);
  \coordinate (tr) at ( 0.5,  1);
  \coordinate (tm) at (  0, 0.3);
  \draw[>-]  (bl) .. controls +( 0, 0.5) and +(0,0) .. (bm);
%  node[below, pos = 0] {$1$};
  \draw[>-]  (br) .. controls +( 0, 0.5) and +(0,0) .. (bm);
%  node[below, pos = 0] {$1$};
  \draw[<-]  (tl) .. controls +( 0, -0.5) and +(0,0) .. (tm);
%  node[above, pos = 0] {$1$} coordinate[pos = 0.25] (ga) ;
  %  \filldraw[draw= green!50!black, fill = white] (ga) circle (1mm)
  %node[left, green!50!black] {$t_1$};

  \draw[<-]  (tr) .. controls +( 0, -0.5) and +(0,0) .. (tm);
%  node[above, pos = 0] {$1$} coordinate[pos = 0.25] (gb) ;
 %   \filldraw[draw= green!50!black, fill = white] (gb) circle (1mm)
%  node[left, green!50!black] {$t_2$};
  \draw [double] (bm) -- (tm);% node[left, pos = 0.5] {$2$};
 
\end{scope}}} };
\draw[->] (i1) -- (i0) coordinate[pos=0.5] (a);
\node[above] at (a) {\NB{\tikz[font=\tiny, scale=.5]{\begin{scope}
  \begin{scope}
    \coordinate (L1) at (0.2,0.4);
    \coordinate (L2) at (0,0);
    \coordinate (R1) at (2.2,0.4);
    \coordinate (R2) at (2,0);
    \coordinate (ML) at (0.6, 0.2);
    \coordinate (MR) at (1.6, 0.2);
    \draw[double] (ML) -- (MR);% node[below, midway] {$2$};
    \draw (MR) .. controls +(0, 0) and +(-0.3,0) .. (R1) ;
    \draw (MR) .. controls +(0, 0) and +(-0.3,0) .. (R2);
    \draw (L1) .. controls +( 0.3, 0) and +(0,0) .. (ML);
    \draw (L2) .. controls +( 0.3, 0) and +(0,0) .. (ML);
  \end{scope}  
 \begin{scope}[yshift = 1cm]
    \coordinate (L1B) at (0.2,0.4);
    \coordinate (L2B) at (0,0);
    \coordinate (R1B) at (2.2,0.4);
    \coordinate (R2B) at (2,0);
    \draw (L1B) .. controls +( 0, 0) and +(0,0) .. (R1B); %node  [left, pos  = 0] {$1$};
    \draw (L2B) .. controls +( 0, 0) and +(0,0) .. (R2B);% node     [left, pos   = 0] {$1$};
 \end{scope}  
  \draw (R1) -- (R1B);
  \draw (R2) -- (R2B);
  \draw (L1) -- (L1B);
  \draw (L2) -- (L2B);
  \draw[thick] (ML) .. controls +(0, 0.6) and +(0, 0.6) .. (MR);
\end{scope}
}}};
  }}
    \end{equation}
    \begin{equation}\label{eqn:def-T-prime}
  T'= \NB{\tikz[xscale = 0.6]{\begin{scope}[font=\tiny]
  \draw[->] (-0.5, -0.5) ..controls +(0,0.3) and +(0,-0.3) .. (0.5,
  0.5);% node[pos=1, above] {} coordinate[pos =0.2] (t2);
  \fill[white] (0,0) circle (2mm);
  \draw[->] (0.5, -0.5) ..controls +(0,0.3) and +(0,-0.3) .. (-0.5,
  0.5);% node[pos=1, above] {} coordinate[pos =0.2] (t1);
  % \filldraw[draw= green!50!black, fill = white] (t2) circle (1mm)
  % node[left, green!50!black] {$c$};
  % \filldraw[draw= green!50!black, fill = white] (t1) circle (1mm)
  % node[right, green!50!black] {$d$};
\end{scope}}}:=
    \NB{\tikz[xscale = 3.5, yscale = 3]{
    \node (i0) at (-1, 0) { $q$\ \NB{\tikz[font= \tiny,
  scale=0.6]{\begin{scope}
  \coordinate (bl) at (-0.5, -1);
  \coordinate (br) at ( 0.5, -1);
  \coordinate (tl) at (-0.5,  1);
  \coordinate (tr) at ( 0.5,  1);
    \coordinate (ml) at (-0.5,  .6);
 \coordinate (mr) at (0.5,  .6);
 \coordinate (G) at (0.8,-0.1);

   \draw[->] (bl) -- (tl);
        \filldraw[draw= green!50!black, fill = white] (ml) circle (1mm) 
  node[left, green!50!black] 
  {$\bar{t}_1$};
  
    \draw[->] (br) -- (tr);
        \filldraw[draw= green!50!black, fill = white] (mr) circle (1mm) 
  node[right, green!50!black] 
  {$\bar{t}_2$};

\end{scope}}} };
     \node (i1) at (0, 0) { \NB{\tikz[font= \tiny,
  scale=0.6]{}} };
\draw[->] (i0) -- (i1) coordinate[pos=0.5] (b);
\node[above] at (b) {\NB{\tikz[font=\tiny, scale=.5]{\begin{scope}
  \begin{scope}
    \coordinate (L1) at (0.2,0.4);
    \coordinate (L2) at (0,0);
    \coordinate (R1) at (2.2,0.4);
    \coordinate (R2) at (2,0);
    \coordinate (ML) at (0.6, 0.2);
    \coordinate (MR) at (1.6, 0.2);
    \draw[double] (ML) -- (MR);% node[above, midway] {$2$};
    \draw (MR) .. controls +(0, 0) and +(-0.3,0) .. (R1) ;
    \draw (MR) .. controls +(0, 0) and +(-0.3,0) .. (R2);
    \draw (L1) .. controls +( 0.3, 0) and +(0,0) .. (ML);
    \draw (L2) .. controls +( 0.3, 0) and +(0,0) .. (ML);
  \end{scope}  
 \begin{scope}[yshift = -1cm]
    \coordinate (L1B) at (0.2,0.4);
    \coordinate (L2B) at (0,0);
    \coordinate (R1B) at (2.2,0.4);
    \coordinate (R2B) at (2,0);
    \draw (L1B) .. controls +( 0, 0) and +(0,0) .. (R1B);% node
                                % [right, pos     = 1] {$1$};
    \draw (L2B) .. controls +( 0, 0) and +(0,0) .. (R2B); % node [right, pos    = 1] {$1$};
 \end{scope}  
  \draw (R1) -- (R1B);
  \draw (R2) -- (R2B);
  \draw (L1) -- (L1B);
  \draw (L2) -- (L2B);
  \draw[thick] (ML) .. controls +(0, -0.6) and +(0, -0.6) .. (MR);
\end{scope}

%%% Local Variables:
%%% mode: latex
%%% TeX-master: t
%%% End:
}}}; 
      }} 
  \,
    \end{equation}
where in both complexes we assume (as in \cite{QRSW1}) that the terms
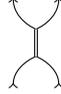
\[
  \NB{\tikz[font= \tiny,
  scale=0.6]{}} 
\]
sit in cohomological degree $0$. In these diagrams $\statespaceN{\cdot}$ has
been omitted to maintain readability.

For a link $L$, define
$\KR_N^{\mathfrak{sl}_2}(L;R) :=
\KR_{N;t_1,t_2}^{\mathfrak{sl}_2}(L;R) $ to be the Khovanov--Rozansky
$\gll_N$-homology of $L$ with coefficients in a ring $R$,
equipped with the action of the Hopf algebra
$\mathcal{U}(\mathfrak{sl}_2)$. When the coefficient ring $R$ is clear from context we will also 
write $\KR_N^{\mathfrak{sl}_2}(L)$ for simplicity.

In what follows, we will simply draw (pieces of) link diagrams (with green dots) 
to represent the $\sll_2$-equivariant Khovanov--Rozansky
$\gll_N$-homology of these diagrams. In other words, we will mostly drop
$\KR_N^{\mathfrak{sl}_2}(\cdot\,;R)$ around diagrams to prevent
overloaded figures.

\begin{thm} \cite[Theorem 4.3]{QRSW3} \label{thm:sl2inv} The homology
$\KR_{N;t_1,t_2}^{\mathfrak{sl}_2}(L;\RN)$ is an invariant of framed
  oriented links.
\end{thm}

\subsection{Foam category}

Let 
$\gFm$ be the category whose objects are green-dotted webs and whose
morphisms are (linear combinations of) foams up to isotopy and equivalence relations induced
by foam evaluation \cite{RW1}.

The \emph{graded hom space} in $\gFm$, between two green-dotted webs $\Gamma_1$ and $\Gamma_2$ is given by foams whose incoming boundary is $\Gamma_1$ and outgoing boundary is $\Gamma_2$. We denote this hom space by $\Hom_{\gFm}(\Gamma_1,\Gamma_2)$. For an element $F\in \Hom_{\gFm}(\Gamma_1,\Gamma_2)$ in this category, and $g \in \mathfrak{sl}_2$, 
\begin{equation} \label{enriched0}
g * F = g(F) - g(\Gamma_1) F + g(\Gamma_2) F \ ,    
\end{equation} 
where $g(\Gamma_j) F$, for $j=1,2$, is the foam $F$ with extra
decorations coming from $g(\Gamma_j)$ (given by formulas
\eqref{eq:e-act-pol-twisthollow}--\eqref{eq:f-act-pol-twistsolid}). 

The flatness of the twists associated with green dots implies that
formula \eqref{enriched0} defines an $\sll_2$-action on the hom spaces
of $\gFm$. From the definition of $*$, one in turn obtains that $\gFm$
is enriched in the monoidal category of $\mathcal{U}(\sll_2)$-modules.
Note also that a morphism $F$ is $\sll_2$-equivariant if and only if
$g * F =0$ for all $g\in \sll_2$. Using this action, we may identify
the state space $\mc{F}_N(\Gamma)$ associated with a web $\Gamma$ as the graded hom space $\Hom_{\gFm}(\emptyset, \Gamma)$, with the induced
$\sll_2$-action $*$. Thus $\mc{F}_N(\Gamma)$ is a graded equivariant module over the commutative $\sll_2$-module algebra $D_\Gamma$ (see Section \ref{subsec:greendottedwebs}).

Define $\gFm\#\sll_2$ to be the category that has the same objects of $\gFm$, but with morphisms that commute with the $\sll_2$-action.  It is equipped with an enriched structure (c.f. equation \eqref{eqn-H-action}) on its \emph{inner hom space}, which is equal to $\Hom_{\gFm}(\Gamma_1,\Gamma_2)$ for any green dotted webs $\Gamma_1$ and $\Gamma_2$. By equation~\eqref{eqn-H-inv-in-HOM},
\begin{equation}
    \Hom_{\gFm\# \sll_2}(\Gamma_1,\Gamma_2)=\Hom_{\gFm}(\Gamma_1,\Gamma_2)^{\sll_2}.
\end{equation}
Also denote by $\Com(\gFm \# \sll_2 )$ the abelian category of complexes in $\gFm\# \sll_2$. The inner hom space extends naturally: given two complexes of green-dotted webs $C$ and $D$
\begin{equation}
    \HOM_{\gFm }(C,D)=\oplus_{i\in \ZZ}\Hom_{\gFm }(C,t^iD).
\end{equation}
Let $\mc{C}(\gFm \# \sll_2)$ be the corresponding homotopy category, with the enriched structure.

Let $\Com(\gFm)$ denote the usual abelian category of complexes of graded foams, $\mc{C}(\gFm)$ be its homotopy category. As in Section \ref{sec:homol-non-sense}, we define $\mc{C}^{\mathfrak{sl}_2}(\gFm)$ to be the corresponding
relative homotopy category, obtained as the Verdier localization of $\mc{C}(\gFm\#\sll_2)$ along the exact forgetful functor
\begin{equation}
    \mathrm{For}: \mc{C}(\gFm\#\sll_2) \lra \mc{C}(\gFm).
\end{equation}
Translating the general construction of Section \ref{homological:sec} into our context, we have the following.

\begin{prop} \label{prop:comenriched}
The category $\mc{C}^{\sll_2}(\gFm)$ is enriched in the category of $\mathcal{U}(\mathfrak{sl}_2)$-modules, with the inner hom space between two green-dotted webs acted on by $\sll_2$ according to equation \eqref{enriched0}. \hfill $\square$
\end{prop}

\subsection{Link cobordisms}

Full functoriality of Khovanov--Rozansky $\mathfrak{gl}_N$-link homologies was proved by Ehrig, Tubbenhauer, and Wedrich \cite{ETW}. This extended earlier work of functoriality for the $\mathfrak{sl}_2$ case in \cite{Cap, CMW,blan1, EST1, Vog, BHPW}.  Functoriality for the $\mathfrak{sl}_2$ case was first investigated by Jacobsson \cite{Jac} and Khovanov \cite{Khcob} who proved it up to a sign.  

In this section we review an extension of \cite{ETW} which states that the target of their functor $\mathcal{F}_N$ lands in the relative homotopy category of foams, with inner hom spaces carrying $\mathcal{U}(\mathfrak{sl}_2)$-module structures \cite[Section 5.2]{QRSW3}.  See Section \ref{homological:sec} for the general framework. We remind the reader that this means that inner morphism sets are objects in the monoidal category and composition of morphisms corresponds to tensor product of objects.
We extend this result to link cobordisms in $S^3 \times [0,1]$.

\begin{dfn}
Let $ \Links$ be the category whose objects are oriented framed
links in $S^3$ and whose morphisms are two-dimensional oriented framed cobordisms between oriented links in $S^3 \times [0,1]$.
\end{dfn}

For links in $\mathbb{R}^3$,
a presentation of this category is given by Beliakova and Wehrli
\cite{BW} building upon work of Carter and Saito \cite{CarSai} in the unframed case.

Following the notation of \cite[Figure 9]{ETW}, we denote some of the
generating movie moves by $\MGH$, $\MGS$, $\MGTW$, and $\MGTH$ when reading the movies from left to right.
Reading the movies in the reverse direction, we denote the generating
moves by $\MGH'$, $\MGS'$, $\MGTW'$, and $\MGTH'$ respectively.  We omit writing down the other oriented variations of these generators.

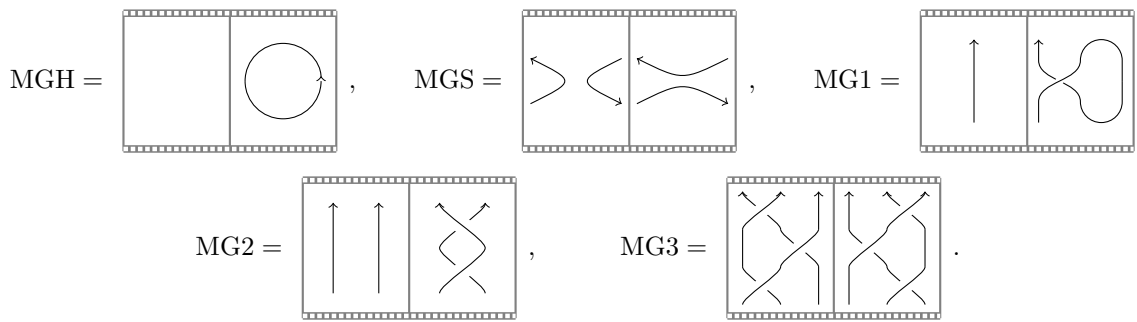
\[ \MGH = 
  \mymovie[yscale = 0.6, xscale =0.35]{}{\NB{\tikz[]{\begin{scope}
  \coordinate (bl) at (0, -1);
  \coordinate (br) at (2, -1);
  \coordinate (tl) at (0,  1);
  \coordinate (tr) at (2,  1);
    \coordinate (ml) at (-0.5,  -.8);
        \coordinate (Ml) at (-0.5,  .8);
 \coordinate (mr) at (0.5,  -.6);
\coordinate (Mr) at (0.5,  .6);

%   \draw[<-<] (bl) -- (tl) node[pos = 0, below] {$1$} node[pos = 1,
%  above] {$1$};
%     \draw[>->] (br) -- (tr) node[pos = 0, below] {$1$} node[pos = 1,
%  above] {$1$} coordinate[pos = 0.25] (A) coordinate[pos = 0.75] (B);
    \draw [-<] (.5, 0) arc (180:0:0.5) ;
    \draw (.5, 0) arc (-180:0:0.5) coordinate[pos = 0.2] (X) coordinate[pos = 0.8] (Y);
%    \filldraw[draw= green!50!black, fill = green] (X) circle (1mm)
 % node[below, green!50!black] {$\frac{1}{2}$};
 %   \filldraw[draw= green!50!black, fill = white] (Y) circle (1mm)
  %node[below, green!50!black] {$\frac{N-1}{2}$};
%\filldraw[draw= green!50!black, fill = white] (A) circle (1mm)
 % node[right, green!50!black] {$-\frac{N-1}{2}$}; 
%  \filldraw[draw= green!50!black, fill = green] (B) circle (1mm)
 % node[right, green!50!black] {$\frac{-1}{2}$};
\end{scope}}}}, \qquad \MGS = 
  \mymovie[yscale = 0.6, xscale =0.35]{\NB{\tikz[scale=0.6]{\begin{scope}[xshift = 2.5cm]
  \draw [->](0, 0) .. controls +(1,.5) 
  .. (0, 1);
    \draw [<-](2, 0) .. controls +(-1,.5) 
  .. (2, 1);
  %%%%%%
 % \draw [->](0, 0) .. controls +(1,.5) 
%  .. (2, 0);
%    \draw [<-](0, 1) .. controls +(1,-.5) 
%  .. (2, 1);  

\end{scope}}}}{\NB{\tikz[scale=0.6]{\input{\imagesfolder/pdg_horres2}}}},
\qquad \MGO = 
   \mymovie[yscale = 0.6, xscale =0.35]{\NB{\tikz[scale=0.55]{\begin{scope}
  \coordinate (bm) at ( 0, -1);
  \coordinate (tm) at ( 0, 1);
  \draw[->] (bm) -- (tm) node[above, pos =1] {} node[below, pos =0] {};
\end{scope}}}}{\NB{\tikz[scale=0.55]{%\begin{scope}
%  \draw (0, -1) -- +(0,2);
%\end{scope}
%\node at (1.25, 0) {$\rightsquigarrow$};
\begin{scope}[xshift = 2.5cm]
  \draw [<-] (0, 1) -- (0, 0.5) .. controls +(0,-0.5) and +(0, 0.5)
  .. (1, -0.5) arc (180:360:0.5) -- (2,0);  
    \fill[white] (0.5, 0) circle (1mm);
  \draw (0, -1) -- (0, -0.5) .. controls +(0,0.5) and +(0, -0.5)
  .. (1, 0.5) arc (180:0:0.5) -- (2,0);
 %%%%%%
\end{scope}}}}
   \]   
\[ \MGTW = 
  \mymovie[yscale = 0.6, xscale
  =0.35]{\NB{\tikz[scale=0.6]{\begin{scope}
  \coordinate (bl) at (-0.5, -1);
  \coordinate (br) at ( 0.5, -1);
  \coordinate (tl) at (-0.5,  1);
  \coordinate (tr) at ( 0.5,  1);
    \coordinate (ml) at (-0.5,  -.8);
        \coordinate (Ml) at (-0.5,  .8);
 \coordinate (mr) at (0.5,  -.6);
\coordinate (Mr) at (0.5,  .6);

 %\draw[>->] (bl) -- (tl) node[pos = 0, below] {$2$} node[pos = 1,
  %above] {$1$} coordinate[pos = 0.4] (ml);
   \draw[->] (bl) -- (tl) node[pos = 0, below] {} node[pos = 1,
  above] {};
  %\draw[>->] (br) -- (tr) node[pos = 0, below] {$1$} node[pos = 1, above] {$2$} coordinate[pos = 0.6] (mr);
  
    \draw[->] (br) -- (tr) node[pos = 0, below] {} node[pos = 1, above] {};
  
  %\draw[->-] (ml) -- (mr) node [pos= 0.5, above] {$1$};
  %  \draw[->-] (Ml) -- (Mr) node [pos= 0.5, above] {$1$};

\end{scope}}}}{\NB{\tikz[scale=0.6]{\begin{scope}
  \draw[] (3, 0) .. controls +(0, 0.2) and +(0, -0.2) ..  +(-1,1);
  \fill[white] (2.5, 0.5) circle (2mm);
  \draw (2, 0) .. controls +(0, 0.2) and +(0, -0.2) ..  +(1,1);
  %%%
    \draw[->] (2, 1) .. controls +(0, .2) and +(0, -0.2) ..  +(1,1); 
       \fill[white] (2.5, 1.5) circle (2mm);

   \draw[->] (3, 1) .. controls +(0, .2) and +(0, -0.2) ..  +(-1,1);
  % \draw (2, 0) arc (180:0:-0.5) -- (1,2) arc (180:0:0.5) coordinate[pos = 0.25] (X) coordinate[pos = 0.75] (Y);
%    \draw (3, 0) arc (-180:0:0.5) -- (4,2) arc (-180:0:-0.5) coordinate[pos = 0.25] (X) coordinate[pos = 0.75] (Y);
\end{scope}}}},\quad \qquad
\MGTH = 
  \mymovie[yscale = 0.6, xscale =0.35]{\NB{\tikz[font= \tiny,
  scale=0.5]{\begin{scope}
  \draw[] (3, 0) .. controls +(0, 0.2) and +(0, -0.2) ..  +(-1,1);
  \fill[white] (2.5, 0.5) circle (2mm);
  \draw (2, 0) .. controls +(0, 0.2) and +(0, -0.2) ..  +(1,1);
  \draw[] (4,0) -- (4,1);
  %%%
   \draw[] (2,1) -- (2,2);
   \draw[] (4, 1) .. controls +(0, 0.2) and +(0, -0.2) ..  +(-1,1);
  \fill[white] (3.5, 1.5) circle (2mm);
  \draw (3, 1) .. controls +(0, 0.2) and +(0, -0.2) ..  +(1,1);
  
 %%%% 
   \draw[->] (3, 2) .. controls +(0, .2) and +(0, -0.2) ..  +(-1,1);

       \fill[white] (2.5, 2.5) circle (2mm);

     \draw[->] (2, 2) .. controls +(0, .2) and +(0, -0.2) ..  +(1,1); 
  \draw[->] (4,2) -- (4,3);

\end{scope}}}}{\NB{\tikz[font= \tiny,
  scale=0.5]{\begin{scope}
  \draw[] (4, 0) .. controls +(0, 0.2) and +(0, -0.2) ..  +(-1,1);
  \fill[white] (3.5, 0.5) circle (2mm);
  \draw (3, 0) .. controls +(0, 0.2) and +(0, -0.2) ..  +(1,1);
  \draw[] (2,0) -- (2,1);
  %%%
   \draw[] (4,1) -- (4,2);
   \draw[] (3, 1) .. controls +(0, 0.2) and +(0, -0.2) ..  +(-1,1);
  \fill[white] (2.5, 1.5) circle (2mm);
  \draw (2, 1) .. controls +(0, 0.2) and +(0, -0.2) ..  +(1,1);
  
 %%%% 
    \draw[->] (4, 2) .. controls +(0, .2) and +(0, -0.2) ..  +(-1,1);
 
       \fill[white] (3.5, 2.5) circle (2mm);

   \draw[->] (3, 2) .. controls +(0, .2) and +(0, -0.2) ..  +(1,1); 
  \draw[->] (2,2) -- (2,3);

\end{scope}}}}.
  \]    
These generators satisfy certain relations described in \cite{BW, CarSai} and other references (such as \cite[Section 4]{ETW}).  
  
For each movie generator $G \colon X \rightarrow Y$, there is an associated foam or morphism of complexes of green-dotted webs $\mc{F}_N(G) \colon \mc{F}_N(X) \rightarrow \mc{F}_N(Y)$.
If $G=\MGO, \MGTW, \MGTH, \MGTH$, the morphism $\mc{F}_N(G)$ is a relative homotopy equivalence.  That is, if $G \colon X \rightarrow Y$, then one could write $\mc{F}_N(G)= (\alpha_{n-1}')^{-1} \circ \alpha_{n-1} \circ \cdots \circ (\alpha_{0}')^{-1} \circ \alpha_{0}$ as a sequence of ``roofs''
\begin{gather}
  \NB{
    \tikz[xscale=1, yscale =1]{
\node (A1) at (0,0) {$X_0$};
      \node[anchor = east] at (A1) {$\mc{F}_N(X)=\,\,\,\,$};
    \node (A2) at (2,0) {$X_1$};
    \node (B1) at (1,-1) {$X_0'$};
    \node (C1) at (4,-.5) {$\cdots$};
 \node (A3) at (6,0) {$X_{n-1}$};
\node (A4) at (8,0) {$X_{n}$};
     \node[anchor=west] at (A4) {$\,\,\,=\mc{F}_N(Y)$};
        \node (B2) at (7,-1) {$X_{n-1}'$};
              \node (B3) at (3,-1) {};
                \node (B4) at (5,-1) {};
 \draw[-to] (A1) -- (B1) node[pos =0.7, left] {$\alpha_0$};
      \draw[-to] (A2) -- (B1) node[pos =0.7, right] {$\alpha_0'$}; 
     \draw[-to] (A3) -- (B2) node[pos =0.7, left] {$\alpha_{n-1}$};
      \draw[-to] (A4) -- (B2) node[pos =0.7, right] {$\alpha_{n-1}'$};      
          \draw[-to] (A2) -- (B3) node[pos =0.7, left] {};
     \draw[-to] (A3) -- (B4) node[pos =0.7, left] {};
}
  }  \ ,
\end{gather}
where the $\alpha_i$ and $ \alpha_i'$ are relative homotopy equivalences such that $g * \alpha_i=0$ for $g \in \mathfrak{sl}_2$ and for all $i$.  For degree reasons (see \cite[Lemma 4.6]{ETW}), if $G=\MGO, \MGTW, \MGTH$, the map
$\mc{F}_N(G)$ must agree with the map in \cite{ETW} up to a scalar.  We choose the $\alpha_i$ and $\alpha_i'$ so that $\mc{F}_N(G)$ precisely agrees with the map in \cite{ETW}.  We can extend the functoriality result of $\mathfrak{gl}_N$-link homology \cite[Theorem 4.5]{ETW} and \cite[Section 4]{MWW1} to our setting.

\begin{thm}[{\cite[Theorem 5.3]{QRSW3}}] \label{thm:functorFN}
  The $\sll_2$-enhanced Khovanov--Rozanksy homology extends to a functor $\KR_N^{\sll_2} \colon \Links \rightarrow \mc{C}^{\mathfrak{sl}_2}(\gFm)$ in which the target category of the functor is enriched in the category of $\mathfrak{sl}_2$-modules. 

Furthermore, the $\mathfrak{sl}_2$-structure on the morphism space is an invariant of the framed links. \end{thm}

\begin{proof}
The functor $\mathcal{F}_N \colon \Links \rightarrow \mc{C}(\gFm)$ is constructed in \cite[Theorem 4.5]{ETW}. 
That is, to every movie generator $M$, there is an associated map of complexes $\mathcal{F}_N(M)$ of (green-dotted) webs.  If there is a movie relation $M_1 = M_2$, then the maps $\mathcal{F}_N(M_1)$ and
$\mathcal{F}_N(M_2)$ are homotopic.
For links in $S^3$, there is an additional relation on movie generators called the sweep-around property \cite[Theorem 4.9]{MWW1}.  It was shown that the two movies involved in this relation induce equal maps on chain complexes \cite[Theorem 3.3]{MWW1}.

Let $\alpha, \beta \colon \mathcal{F}_N(L_1) \rightarrow \mathcal{F}_N(L_2)$ be two relative homotopic maps. Thus in the homotopy category $\mc{C}^{\sll_2}(\gFm)$, their classes $[\alpha]$ and $[\beta]$ are equal.  Let us show that for any $g \in \mathcal{U}(\mathfrak{sl}_2)$, $g*[\alpha]=g*[\beta]$ in the relative homotopy category. Since $[\alpha]=[\beta]$, there is a homotopy $H$, when forgetting the $\sll_2$-actions, such that
\begin{equation} \label{eq:gactonhom}
\alpha-\beta= H d + d H .    
\end{equation}
Acting by $g$ on both sides of \eqref{eq:gactonhom}, using the Leibniz rule and noting that $\sll_2$-actions commute with differentials, we obtain
\[
g * \alpha - g * \beta = (g*H) d + d (g*H),
\]
and therefore $g*[\alpha]=g*[\beta]$.

The last part follows from Theorem \ref{thm:sl2inv} and Proposition \ref{prop:comenriched}.
\end{proof}

As for foams, one can consider link cobordisms which are $\sll_2$-invariant and therefore draw attention on the $\sll_2$ structure of the homology spaces associated with links rather than on the hom-spaces. As for foams, in order to have a non-trivial theory, it is important to decorate links with green-dots.

Hence we consider the category $\dLinks_N$ whose objects are links decorated with green dots (of solid and hollow types) and whose morphisms are link cobordisms with compatibility conditions that we explain below.

If $\Sigma\colon L_0 \to L_1$, then for each connected component $S$ of the surface $\Sigma$, let $s_{0}$ (resp.{} $h_{0}$) be the sum of labels of solid (resp.{} hollow) green dots located on $L_0$  and by $s_1$ (resp.{} $h_1$) the sum of
labels of solid (resp.{} hollow) green dots located on $L_1$. The compatibility conditions read:\begin{equation}
      \label{eq:gdlinks}
      \begin{cases}
        s_{{1}} - s_{{0}} &= \frac{\chi(S)}{2}\qquad \text{and},\\
        h_{{1}} - h_{{0}} &= \frac{(N-1)\chi(S)}{2}.
      \end{cases}
    \end{equation}
This condition implies that maps induced by link cobordisms are $\sll_2$-equivariant, so similarly to \cite[Corollary 5.4]{QRSW3}, one obtains:
\begin{thm} \label{thm:functorFN-2nd-version} 
Khovanov--Rozansky homology induces a functor $\KR_N^{\sll_2} \colon \dLinks_N \rightarrow \mc{C}^{\mathfrak{sl}_2}(\Bbbk_N)$. 
\end{thm}

\section{Skein lasagna module}
\label{skein:sec}

The skein lasagna module was defined by Morrison, Walker, and Wedrich \cite{MWW1} in the context of Khovanov--Rozansky homology over the base ring $\mathbb{Z}$.  This was  extended by the same authors \cite{MWW2} to equivariant Khovanov--Rozansky homology (which is defined over $R=\mathbb{Z}[E_1, \ldots, E_N]$).

The goal of this section is to equip the skein lasagna module over $R=\mathbb{Z}[E_1, \ldots, E_N]$ with an $\mathfrak{sl}_2$-action.

Let $W$ be an oriented smooth four-manifold and $L \subset \partial W$ a framed green-dotted link.  A lasagna filling $F=(\Sigma,\{(B_i,L_i,v_i)\})$ of $W$ with a link $L$ in the boundary consists of
\begin{itemize}
    \item a finite collection of disjoint 4-balls $B_i$ embedded in the interior of $W$;
    \item a framed oriented surface $\Sigma$ properly embedded in $W- int(\bigcup_i B_i)$ intersecting $\partial W$ at $L$ and intersecting each $\partial B_i$ at a green-dotted link $L_i$;
    \item a homogeneous label $v_i \in \KR_N(L_i;R)$ for each $i$.
    \end{itemize}
    Similarly to what happens for morphisms in $\dLinks$, for each connected component $S$ of the surface $\Sigma$, the green dots located on the boundary of $S$ should be compatible with the Euler characteristic of $S$. To explain what this means, let us denote by $s_{\mathrm{out}}$ (resp.{} $h_{\mathrm{out}}$) the sum of labels of solid (resp.{} hollow) green dots located on $\partial W \cap S$ and by $s_{\mathrm{in}}$ (resp.{} $h_{\mathrm{in}}$ ) the sum of labels of solid (resp.{} hollow) green dots located on $\bigcup \partial B_i \cap S$. The compatibility conditions read:
    \begin{equation}
      \label{eq:gd-lasagana}
      \begin{cases}
        s_{\mathrm{out}} - s_{\mathrm{in}} &= \frac{\chi(S)}{2}\quad \text{and},\\
        h_{\mathrm{out}} - h_{\mathrm{in}} &= \frac{(N-1)\chi(S)}{2}.
      \end{cases}
    \end{equation}

    Note that these are the same compatibility conditions as \eqref{eq:gdlinks} provided one sees $\Sigma$ as a cobordism from $\bigcup L_i$ to $L$.

If $W$ is a 4-ball, then there is a cobordism map
\[
\KR^{\sll_2}_N(\Sigma) \colon \otimes_i \KR_N^{\sll_2}(\partial B_i, L_i) \rightarrow \KR^{\sll_2}_N(\partial W,L) \ ,
\]
giving rise to an evaluation 
\[
\KR^{\sll_2}_N(F) = \KR^{\sll_2}_N(\Sigma)(\otimes_i v_i) \in \KR_N^{\sll_2}(\partial W,L) \ .
\]
The skein lasagna module is defined as the bigraded $R$-module 
\begin{equation}
  S_0^N(W,L;R) = R\{\text{lasagna fillings F of W with boundary L}\} / \sim
\end{equation}
where $\sim$ is the transitive and linear closure of the following relations:
\begin{itemize}
    \item linear combinations of lasagna fillings are set to be multilinear in the labels $v_i$;
    \item two lasagna fillings $F_1$ and $F_2$ are set to be equivalent if $F_1$ has an input ball $B_i$ with label $v_i$, and $F_2$ is obtained from $F_1$ by replacing $B_i$ with another lasagna filling $F_3$ of a 4-ball such that $v_i = \KR_N^{\sll_2}(F_3)$, followed by an isotopy relative to $\partial W$, where the isotopy is allowed to move the input balls.
\end{itemize}

\begin{figure}[ht]
\centering
\NB{\tikz[scale = 0.8]{\input{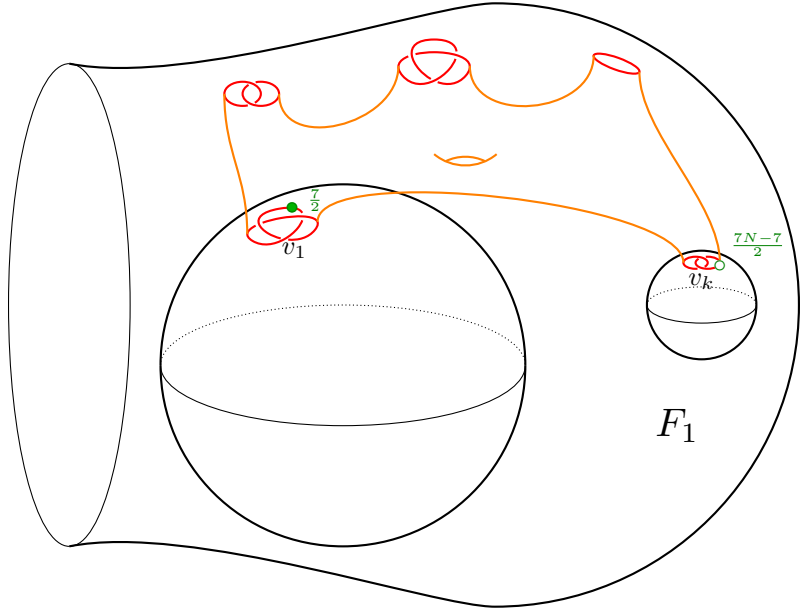}}}
\caption{A filling $F_1$ of $W$. }\end{figure}

        \begin{figure}[ht]
\centering
\NB{\tikz[scale = 0.8]{\input{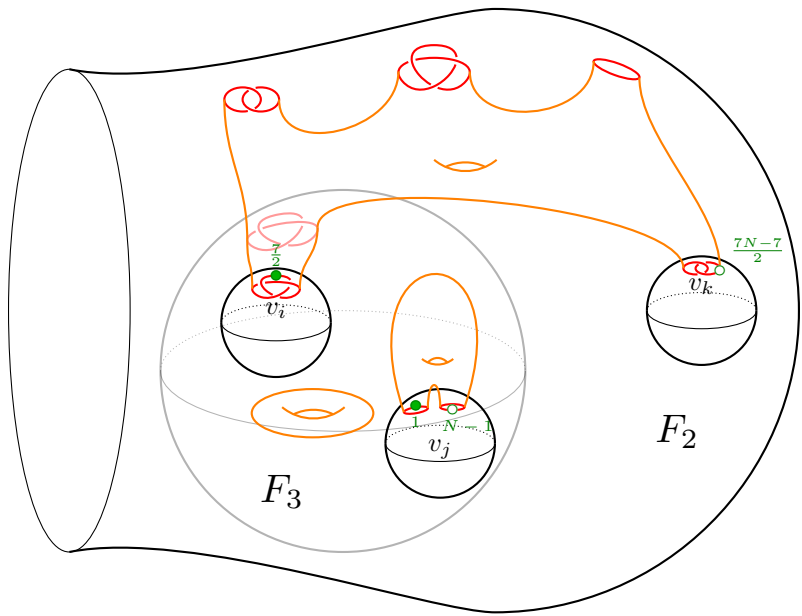}}}
\caption{A filling $F_2$ of $W$ which is equivalent to the filling $F_1$ (for good choices of $v_1$, $v_i$ and $v_j$). Note that the various green dots satisfy conditions given in \eqref{eq:gd-lasagana}.}\end{figure}

There is a natural bigrading on the skein lasagna module coming from the bigrading on Khovanov--Rozansky homology.  There is a third grading called the homology degree coming from fillings represented by classes in $H_2(W,L; \mathbb{Z})$.

\begin{thm} \label{thm:sl2las}
    For $R=\mathbb{Z}[E_1,\ldots,E_N]$, the $\sll_2$-structure on $\KR_N^{\sll_2}$ induces a structure of $\sll_2$-module on  the skein lasagna module $S_0^N(W,L;R)$ 
\end{thm}

\begin{proof}
It suffices to show that the relations of the skein lasagna are compatible by this $\mathfrak{sl}_2$-action.

    Clearly multilinearity of the labels $v_i$ in Khovanov--Rozansky homology are preserved by this action.
    The fact that the second equivalence relation is preserved by the $\mathfrak{sl}_2$-action follows from Theorem \ref{thm:functorFN-2nd-version} and from the compatibility conditions imposed on green dots.
\end{proof}

\begin{rmk}
    In principle, this representation depends upon parameters $t_1,t_2$ built into the action of $\mathfrak{sl}_2$ on link homology in Theorem \ref{thm:sl2inv}.
\end{rmk}

\begin{rmk} More generally, skein lasagna modules can be constructed from link homology functors valued in any symmetric monoidal cocomplete target 1-category $\mathcal{V}$, whose tensor product preserves colimits separately in each variable, provided they satisfy certain conditions, see \cite[Theorem 2.1]{wedrich2025linkhomologytopologicalquantum}. Here we take $\mathcal{V}$ to be the relative homotopy category of bigraded $\Bbbk_N\# \mathfrak{sl}_2 $-modules (see equation \eqref{eqn-H-inv-in-HOM} and Theorem \ref{thm-tensor-hom-adj}). Beyond the description given above, \cite[Section 2.4]{wedrich2025linkhomologytopologicalquantum} lists two alternative definitions of the skein lasagna module, as a colimit of the link homology functor, valued in $\mathcal{V}$, over an indexing category with objects given by lasagna diagrams and morphisms given by enclosement of input balls by a larger input ball. When adopting this perspective, an important subtlety has to be addressed: $\mathcal{V}$ is usually the self-enrichment over an underlying closed symmetric monoidal category $\mathcal{V}_0$ (e.g. bigraded abelian groups with homogeneous homomorphisms versus grading-preserving homomorphisms). Morphisms associated to link cobordisms, e.g. from the enclosement relation, are typically morphisms in $\mathcal{V}$, but not in $\mathcal{V}_0$. Nevertheless, it is sometimes possible to compute the colimit in $\mathcal{V}_0$ by twisting the values of the link homology functor by autoequivalences of $\mathcal{V}_0$, depending on lasagna filling. In the classical case, these are the grading shifts depending on the Euler characteristic of the lasagna surface, here they are given by the green dots.
\end{rmk}

\section{Dotted Temperley--Lieb category}
\label{dTL:sec}
Following \cite{HRW}, we study an equivariant version of the dotted Temperley--Lieb category.
This gives a reformulation of the equivariant skein lasagna module for equivariant Khovanov homology for 2-handlebodies.
We study an $\mathfrak{sl}_2$-action on this category.

\subsection{Definitions}
The equivariant dotted Temperley--Lieb category $\dTL$ is the $\mathbb{Z}$-graded pivotal category generated by an object $c$.
Morphisms are generated by a degree $2$ endomorphism of $c$:
\[
  \NB{\tikz[scale = 0.6]{\begin{scope}
  \draw [] (0, -1) -- +(0,2);
 %   \coordinate (A) at (0, 0);
  %      \filldraw[draw= green!50!black, fill = white] (A) circle (1mm)
  %node[left, green!50!black] 
  %%
     \coordinate (A) at (0, 0);
        \filldraw[draw= black!50!black, fill = black] (A) circle (1mm);
%  node[left, green!50!black] 
 % {$-\frac{1}{2}(N-1)$};
 
\end{scope}}} \colon c \rightarrow c
\]
along with cups and caps of degree 0 from the pivotal structure:
\[
  \NB{\tikz[scale = 0.6]{\begin{scope}
  \coordinate (bl) at (0, -1);
  \coordinate (br) at (2, -1);
  \coordinate (tl) at (0,  1);
  \coordinate (tr) at (2,  1);
    \coordinate (ml) at (-0.5,  -.8);
        \coordinate (Ml) at (-0.5,  .8);
 \coordinate (mr) at (0.5,  -.6);
\coordinate (Mr) at (0.5,  .6);

%   \draw[<-<] (bl) -- (tl) node[pos = 0, below] {$1$} node[pos = 1,
%  above] {$1$};
%     \draw[>->] (br) -- (tr) node[pos = 0, below] {$1$} node[pos = 1,
%  above] {$1$} coordinate[pos = 0.25] (A) coordinate[pos = 0.75] (B);
  %  \draw [] (.5, 0) arc (180:0:0.5) ;
   \draw (.5, 0) arc (-180:0:0.5) coordinate[pos = 0.2] (X) coordinate[pos = 0.8] (Y);
%    \filldraw[draw= green!50!black, fill = green] (X) circle (1mm)
 % node[below, green!50!black] {$\frac{1}{2}$};
 %   \filldraw[draw= green!50!black, fill = white] (Y) circle (1mm)
  %node[below, green!50!black] {$\frac{N-1}{2}$};
%\filldraw[draw= green!50!black, fill = white] (A) circle (1mm)
 % node[right, green!50!black] {$-\frac{N-1}{2}$}; 
%  \filldraw[draw= green!50!black, fill = green] (B) circle (1mm)
 % node[right, green!50!black] {$\frac{-1}{2}$};
\end{scope}}} \colon \Id \rightarrow c^2
  \quad \quad 
  \NB{\tikz[scale = 0.6]{\begin{scope}
  \coordinate (bl) at (0, -1);
  \coordinate (br) at (2, -1);
  \coordinate (tl) at (0,  1);
  \coordinate (tr) at (2,  1);
    \coordinate (ml) at (-0.5,  -.8);
        \coordinate (Ml) at (-0.5,  .8);
 \coordinate (mr) at (0.5,  -.6);
\coordinate (Mr) at (0.5,  .6);

%   \draw[<-<] (bl) -- (tl) node[pos = 0, below] {$1$} node[pos = 1,
%  above] {$1$};
%     \draw[>->] (br) -- (tr) node[pos = 0, below] {$1$} node[pos = 1,
%  above] {$1$} coordinate[pos = 0.25] (A) coordinate[pos = 0.75] (B);
   \draw [] (.5, 0) arc (180:0:0.5) ;
  % \draw (.5, 0) arc (-180:0:0.5) coordinate[pos = 0.2] (X) coordinate[pos = 0.8] (Y);
%    \filldraw[draw= green!50!black, fill = green] (X) circle (1mm)
 % node[below, green!50!black] {$\frac{1}{2}$};
 %   \filldraw[draw= green!50!black, fill = white] (Y) circle (1mm)
  %node[below, green!50!black] {$\frac{N-1}{2}$};
%\filldraw[draw= green!50!black, fill = white] (A) circle (1mm)
 % node[right, green!50!black] {$-\frac{N-1}{2}$}; 
%  \filldraw[draw= green!50!black, fill = green] (B) circle (1mm)
 % node[right, green!50!black] {$\frac{-1}{2}$};
\end{scope}}} \colon c^2 \rightarrow \Id.
\]
They satisfy the relations
\begin{gather*}
  \NB{\tikz[scale = 0.6]{\begin{scope}
  \coordinate (bl) at (0, -1);
  \coordinate (br) at (2, -1);
  \coordinate (tl) at (0,  1);
  \coordinate (tr) at (2,  1);
    \coordinate (ml) at (-0.5,  -.8);
        \coordinate (Ml) at (-0.5,  .8);
 \coordinate (mr) at (0.5,  -.6);
\coordinate (Mr) at (0.5,  .6);

%   \draw[<-<] (bl) -- (tl) node[pos = 0, below] {$1$} node[pos = 1,
%  above] {$1$};
%     \draw[>->] (br) -- (tr) node[pos = 0, below] {$1$} node[pos = 1,
%  above] {$1$} coordinate[pos = 0.25] (A) coordinate[pos = 0.75] (B);
    \draw [] (.5, 0) arc (180:0:0.5) ;
    \draw (.5, 0) arc (-180:0:0.5) coordinate[pos = 0.2] (X) coordinate[pos = 0.8] (Y);
%    \filldraw[draw= green!50!black, fill = green] (X) circle (1mm)
 % node[below, green!50!black] {$\frac{1}{2}$};
 %   \filldraw[draw= green!50!black, fill = white] (Y) circle (1mm)
  %node[below, green!50!black] {$\frac{N-1}{2}$};
%\filldraw[draw= green!50!black, fill = white] (A) circle (1mm)
 % node[right, green!50!black] {$-\frac{N-1}{2}$}; 
%  \filldraw[draw= green!50!black, fill = green] (B) circle (1mm)
 % node[right, green!50!black] {$\frac{-1}{2}$};
\end{scope}}} = 2,\qquad
   \NB{\tikz[scale = 0.6]{\begin{scope}
  \draw [] (0, -1) -- +(0,2);
 %   \coordinate (A) at (0, 0);
  %      \filldraw[draw= green!50!black, fill = white] (A) circle (1mm)
  %node[left, green!50!black] 
  %%
     \coordinate (A) at (0, -0.5);
        \filldraw[draw= black!50!black, fill = black] (A) circle (1mm);
    \coordinate (B) at (0, 0.5);
        \filldraw[draw= black!50!black, fill = black] (B) circle (1mm);
\end{scope}}}
 ~=~
 E_1  \NB{\tikz[scale = 0.6]{}}
 - E_2
   \NB{\tikz[scale = 0.6]{\begin{scope}
  \coordinate (bm) at ( 0, -1);
  \coordinate (tm) at ( 0, 1);
  \draw[] (bm) -- (tm) node[above, pos =1] {} node[below, pos =0] {};
\end{scope}}}\ , \qquad \text{and}
\qquad
\NB{\tikz[scale = 0.6]{\begin{scope}
  \coordinate (bl) at (-0.5, -1);
  \coordinate (br) at ( 0.5, -1);
  \coordinate (tl) at (-0.5,  1);
  \coordinate (tr) at ( 0.5,  1);
    \coordinate (ml) at (-0.5,  0);
        \coordinate (Ml) at (-0.5,  .8);
 \coordinate (mr) at (0.5,  -.6);
\coordinate (Mr) at (0.5,  .6);

 %\draw[>->] (bl) -- (tl) node[pos = 0, below] {$2$} node[pos = 1,
  %above] {$1$} coordinate[pos = 0.4] (ml);
   \draw[] (bl) -- (tl);% node[pos = 0, below] {$1$} node[pos = 1,above] {$1$};
  %\draw[>->] (br) -- (tr) node[pos = 0, below] {$1$} node[pos = 1, above] {$2$} coordinate[pos = 0.6] (mr);
  
    \draw[] (br) -- (tr);% node[pos = 0, below] {$1$} node[pos = 1, above] {$1$};
  
  %\draw[->-] (ml) -- (mr) node [pos= 0.5, above] {$1$};
  %  \draw[->-] (Ml) -- (Mr) node [pos= 0.5, above] {$1$};

  \filldraw[draw= black!50!black, fill = black] (ml) circle (1mm);

\end{scope}}} +
\NB{\tikz[scale = 0.6]{\begin{scope}
  \coordinate (bl) at (-0.5, -1);
  \coordinate (br) at ( 0.5, -1);
  \coordinate (tl) at (-0.5,  1);
  \coordinate (tr) at ( 0.5,  1);
    \coordinate (ml) at (-0.5,  0);
        \coordinate (Ml) at (-0.5,  .8);
 \coordinate (mr) at (0.5,  0);
\coordinate (Mr) at (0.5,  .6);

 %\draw[>->] (bl) -- (tl) node[pos = 0, below] {$2$} node[pos = 1,
  %above] {$1$} coordinate[pos = 0.4] (ml);
   \draw[] (bl) -- (tl);% node[pos = 0, below] {$1$} node[pos = 1,above] {$1$};
  %\draw[>->] (br) -- (tr) node[pos = 0, below] {$1$} node[pos = 1, above] {$2$} coordinate[pos = 0.6] (mr);
  
    \draw[] (br) -- (tr);% node[pos = 0, below] {$1$} node[pos = 1, above] {$1$};
  
  %\draw[->-] (ml) -- (mr) node [pos= 0.5, above] {$1$};
  %  \draw[->-] (Ml) -- (Mr) node [pos= 0.5, above] {$1$};

  \filldraw[draw= black!50!black, fill = black] (mr) circle (1mm);

\end{scope}}}
=
E_1
 \NB{\tikz[scale = 0.6]{\begin{scope}
  \coordinate (bl) at (-0.5, -1);
  \coordinate (br) at ( 0.5, -1);
  \coordinate (tl) at (-0.5,  1);
  \coordinate (tr) at ( 0.5,  1);
    \coordinate (ml) at (-0.5,  -.8);
        \coordinate (Ml) at (-0.5,  .8);
 \coordinate (mr) at (0.5,  -.6);
\coordinate (Mr) at (0.5,  .6);

 %\draw[>->] (bl) -- (tl) node[pos = 0, below] {$2$} node[pos = 1,
  %above] {$1$} coordinate[pos = 0.4] (ml);
   \draw[] (bl) -- (tl);% node[pos = 0, below] {$1$} node[pos = 1,above] {$1$};
  %\draw[>->] (br) -- (tr) node[pos = 0, below] {$1$} node[pos = 1, above] {$2$} coordinate[pos = 0.6] (mr);
  
    \draw[] (br) -- (tr);% node[pos = 0, below] {$1$} node[pos = 1, above] {$1$};
  
  %\draw[->-] (ml) -- (mr) node [pos= 0.5, above] {$1$};
  %  \draw[->-] (Ml) -- (Mr) node [pos= 0.5, above] {$1$};

\end{scope}}}
 -E_1
  \NB{\tikz[scale = 0.6]{\begin{scope}
  \coordinate (bl) at (0, -1);
  \coordinate (br) at (2, -1);
  \coordinate (tl) at (0,  1);
  \coordinate (tr) at (2,  1);
    \coordinate (ml) at (-0.5,  -.8);
        \coordinate (Ml) at (-0.5,  .8);
 \coordinate (mr) at (0.5,  -.6);
\coordinate (Mr) at (0.5,  .6);

%   \draw[<-<] (bl) -- (tl) node[pos = 0, below] {$1$} node[pos = 1,
%  above] {$1$};
%     \draw[>->] (br) -- (tr) node[pos = 0, below] {$1$} node[pos = 1,
%  above] {$1$} coordinate[pos = 0.25] (A) coordinate[pos = 0.75] (B);
   \draw [] (.5, 1) arc (-180:0:0.5) ;
   \draw (.5, -1) arc (180:0:0.5) coordinate[pos = 0.2] (X) coordinate[pos = 0.8] (Y);
%    \filldraw[draw= green!50!black, fill = green] (X) circle (1mm)
 % node[below, green!50!black] {$\frac{1}{2}$};
 %   \filldraw[draw= green!50!black, fill = white] (Y) circle (1mm)
  %node[below, green!50!black] {$\frac{N-1}{2}$};
%\filldraw[draw= green!50!black, fill = white] (A) circle (1mm)
 % node[right, green!50!black] {$-\frac{N-1}{2}$}; 
%  \filldraw[draw= green!50!black, fill = green] (B) circle (1mm)
 % node[right, green!50!black] {$\frac{-1}{2}$};
\end{scope}}}
  + 
    \NB{\tikz[scale = 0.6]{\begin{scope}
  \coordinate (bl) at (0, -1);
  \coordinate (br) at (2, -1);
  \coordinate (tl) at (0,  1);
  \coordinate (tr) at (2,  1);
    \coordinate (ml) at (-0.5,  -.8);
        \coordinate (Ml) at (-0.5,  .8);
 \coordinate (mr) at (0.5,  -.6);
\coordinate (Mr) at (0.5,  .6);

%   \draw[<-<] (bl) -- (tl) node[pos = 0, below] {$1$} node[pos = 1,
%  above] {$1$};
%     \draw[>->] (br) -- (tr) node[pos = 0, below] {$1$} node[pos = 1,
%  above] {$1$} coordinate[pos = 0.25] (A) coordinate[pos = 0.75] (B);
   \draw [] (.5, 1) arc (-180:0:0.5) coordinate[pos = 0.5] (Z) ;
   \draw (.5, -1) arc (180:0:0.5) coordinate[pos = 0.2] (X) coordinate[pos = 0.8] (Y);
    \filldraw[draw= black!50!black, fill = black] (Z) circle (1mm);
 % node[below, green!50!black] {$\frac{1}{2}$};
 %   \filldraw[draw= green!50!black, fill = white] (Y) circle (1mm)
  %node[below, green!50!black] {$\frac{N-1}{2}$};
%\filldraw[draw= green!50!black, fill = white] (A) circle (1mm)
 % node[right, green!50!black] {$-\frac{N-1}{2}$}; 
%  \filldraw[draw= green!50!black, fill = green] (B) circle (1mm)
 % node[right, green!50!black] {$\frac{-1}{2}$};
\end{scope}}}
    +
     \NB{\tikz[scale = 0.6]{\begin{scope}
  \coordinate (bl) at (0, -1);
  \coordinate (br) at (2, -1);
  \coordinate (tl) at (0,  1);
  \coordinate (tr) at (2,  1);
    \coordinate (ml) at (-0.5,  -.8);
        \coordinate (Ml) at (-0.5,  .8);
 \coordinate (mr) at (0.5,  -.6);
\coordinate (Mr) at (0.5,  .6);

%   \draw[<-<] (bl) -- (tl) node[pos = 0, below] {$1$} node[pos = 1,
%  above] {$1$};
%     \draw[>->] (br) -- (tr) node[pos = 0, below] {$1$} node[pos = 1,
%  above] {$1$} coordinate[pos = 0.25] (A) coordinate[pos = 0.75] (B);
   \draw [] (.5, 1) arc (-180:0:0.5) coordinate[pos = 0.5] (Z) ;
   \draw (.5, -1) arc (180:0:0.5) coordinate[pos = 0.5] (X);
    \filldraw[draw= black!50!black, fill = black] (X) circle (1mm);
 % node[below, green!50!black] {$\frac{1}{2}$};
 %   \filldraw[draw= green!50!black, fill = white] (Y) circle (1mm)
  %node[below, green!50!black] {$\frac{N-1}{2}$};
%\filldraw[draw= green!50!black, fill = white] (A) circle (1mm)
 % node[right, green!50!black] {$-\frac{N-1}{2}$}; 
%  \filldraw[draw= green!50!black, fill = green] (B) circle (1mm)
 % node[right, green!50!black] {$\frac{-1}{2}$};
\end{scope}}}.
\end{gather*}

\begin{prop}
For parameters $a_1, a_2$, there is an action of 
$\mathfrak{sl}_2$ on $\dTL$ defined by
\begin{align*}
&\de(E_1)=-2 \ , 
&&\df(E_1)=E_1^2-2E_2 \ , 
&&\dh(E_1)=-2E_1 \ ,
\\
&\de(E_2)=-E_1 \ , 
&&\df(E_2)=E_1 E_2 \ , 
&&\dh(E_2)=-4E_2 \ ,
\\
&\de
\left( \NB{\tikz[scale = 0.6]{}} \right)
=
- \NB{\tikz[scale = 0.6]{}} \ , 
&&\df
\left( \NB{\tikz[scale = 0.6]{}} \right)
=
 \NB{\tikz[scale = 0.6]{}} \ , 
&& \dh
\left( \NB{\tikz[scale = 0.6]{}} \right)
=
(-2)  \NB{\tikz[scale = 0.6]{}} \ ,
\\
&\de
\left( \NB{\tikz[scale = 0.6]{}} \right)
  = 0
&& \df
\left( \NB{\tikz[scale = 0.6]{}} \right)
=
a_1 \NB{\tikz[scale = 0.6]{\begin{scope}
  \coordinate (bl) at (0, -1);
  \coordinate (br) at (2, -1);
  \coordinate (tl) at (0,  1);
  \coordinate (tr) at (2,  1);
    \coordinate (ml) at (-0.5,  -.8);
        \coordinate (Ml) at (-0.5,  .8);
 \coordinate (mr) at (0.5,  -.6);
\coordinate (Mr) at (0.5,  .6);

%   \draw[<-<] (bl) -- (tl) node[pos = 0, below] {$1$} node[pos = 1,
%  above] {$1$};
%     \draw[>->] (br) -- (tr) node[pos = 0, below] {$1$} node[pos = 1,
%  above] {$1$} coordinate[pos = 0.25] (A) coordinate[pos = 0.75] (B);
  % \draw [] (.5, 1) arc (-180:0:0.5) coordinate[pos = 0.5] (Z) ;
   \draw (.5, -1) arc (180:0:0.5) coordinate[pos = 0.5] (X) coordinate[pos = 0.8] (Y);
    \filldraw[draw= black!50!black, fill = black] (X) circle (1mm);
 % node[below, green!50!black] {$\frac{1}{2}$};
 %   \filldraw[draw= green!50!black, fill = white] (Y) circle (1mm)
  %node[below, green!50!black] {$\frac{N-1}{2}$};
%\filldraw[draw= green!50!black, fill = white] (A) circle (1mm)
 % node[right, green!50!black] {$-\frac{N-1}{2}$}; 
%  \filldraw[draw= green!50!black, fill = green] (B) circle (1mm)
 % node[right, green!50!black] {$\frac{-1}{2}$};
\end{scope}}}
+
  a_2 E_1 \NB{\tikz[scale = 0.6]{}} \ ,
&&
\dh
\left( \NB{\tikz[scale = 0.6]{}} \right)
=
(-a_1-2a_2) \left( \NB{\tikz[scale = 0.6]{}} \right) \ ,
\\
&\de
\left( \NB{\tikz[scale = 0.6]{}} \right)
=0 \ ,\qquad
&&\df
\left( \NB{\tikz[scale = 0.6]{}} \right)
=
-a_1
\NB{\tikz[scale = 0.6]{\begin{scope}
  \coordinate (bl) at (0, -1);
  \coordinate (br) at (2, -1);
  \coordinate (tl) at (0,  1);
  \coordinate (tr) at (2,  1);
    \coordinate (ml) at (-0.5,  -.8);
        \coordinate (Ml) at (-0.5,  .8);
 \coordinate (mr) at (0.5,  -.6);
\coordinate (Mr) at (0.5,  .6);

%   \draw[<-<] (bl) -- (tl) node[pos = 0, below] {$1$} node[pos = 1,
%  above] {$1$};
%     \draw[>->] (br) -- (tr) node[pos = 0, below] {$1$} node[pos = 1,
%  above] {$1$} coordinate[pos = 0.25] (A) coordinate[pos = 0.75] (B);
  %  \draw [] (.5, 0) arc (180:0:0.5) ;
  % \draw (.5, 0) arc (-180:0:0.5) coordinate[pos = 0.2] (X) coordinate[pos = 0.8] (Y);

     \draw [] (.5, 1) arc (-180:0:0.5) coordinate[pos = 0.5] (Z) ;
  % \draw (.5, -1) arc (180:0:0.5) coordinate[pos = 0.2] (X) coordinate[pos = 0.8] (Y);
    \filldraw[draw= black!50!black, fill = black] (Z) circle (1mm);
%    \filldraw[draw= green!50!black, fill = green] (X) circle (1mm)
 % node[below, green!50!black] {$\frac{1}{2}$};
 %   \filldraw[draw= green!50!black, fill = white] (Y) circle (1mm)
  %node[below, green!50!black] {$\frac{N-1}{2}$};
%\filldraw[draw= green!50!black, fill = white] (A) circle (1mm)
 % node[right, green!50!black] {$-\frac{N-1}{2}$}; 
%  \filldraw[draw= green!50!black, fill = green] (B) circle (1mm)
 % node[right, green!50!black] {$\frac{-1}{2}$};
\end{scope}}}
-a_2 E_1
    \NB{\tikz[scale = 0.6]{}} \ ,\qquad
&&\dh
\left( \NB{\tikz[scale = 0.6]{}} \right)
=
(a_1 + 2a_2) \left( \NB{\tikz[scale = 0.6]{}} \right) \ .
\end{align*}

Furthermore, $\dh$ is a negative degree operator if $a_1+2 a_2=0$.
\end{prop}

\begin{proof}
    It is straightforward to check that the relations of $\dTL$ are preserved under the actions of $\de, \df, \dh$ and that $\de, \df, \dh$ satisfy $\mathfrak{sl}_2$-relations.
\end{proof}

Define
\[
  \NB{\tikz[scale = 0.6]{\begin{scope}
  \coordinate (lb) at ( 0, -1);
  \coordinate (lt) at ( 0, 1);
   \coordinate (rb) at ( 2, -1);
      \coordinate (rt) at ( 2, 1);
  \draw[] (lb) -- (rt) node[above, pos =1] {} node[below, pos =0] {};
    \draw[] (lt) -- (rb) node[above, pos =1] {} node[below, pos =0] {};
\end{scope}}}
 :=
    \NB{\tikz[scale = 0.6]{}}
    -
   \NB{\tikz[scale = 0.6]{}} 
\]
and 
more generally, for $i=1,\ldots, n-1$, define the endomorphism $s_i$ by
\[
s_i =: 
\NB{\tikz[scale = 0.6]{\begin{scope}
 \coordinate (llb) at ( -2, -1);
     \coordinate(B) at ( 1, 0) node[right] {};        
  \coordinate (llt) at ( -2, 1);
   \coordinate (rrb) at ( 4, -1);
  \coordinate (rrt) at ( 4, 1);
  \coordinate (lb) at ( 0, -1);
  \coordinate (lt) at ( 0, 1);
   \coordinate (rb) at ( 2, -1);
      \coordinate (rt) at ( 2, 1);
  \draw[] (lb) -- (rt) node[above, pos =1] {} node[below, pos =0] {};
    \draw[] (lt) -- (rb) node[above, pos =1] {} node[below, pos =0] {};
      \draw[] (llb) -- (llt) node[right, pos =.5] {$\cdots$} node[below, pos =0] {};
          \draw[] (rrb) -- (rrt) node[left, pos =.5] {$\cdots$}  node[below, pos =0] {};
 % \coordinate(A) at ( -1, 0) node[left] {$\cdots$};  

  %   \draw[] (B1) -- (B2) node[above, pos =1] {} node[right, pos =.5] {$\cdots$};
\end{scope}}}\ .
\]
This gives rise to a morphism from the symmetric group $S_n$ on $\End(c^{\otimes n})$.
When $\Bbbk=\mathbb{Q}$, this leads to an endomorphism $p_n \colon c^{\otimes n} \rightarrow c^{\otimes n}$ defined by
\begin{equation} \label{eq:defpn}
p_n = \frac{1}{n!} \sum_{n \in S_n} w \ .
\end{equation}
Graphically we depict this morphism by:
\[ p_n =
  \NB{\tikz[scale = 0.6]{\begin{scope}
%  \draw (0,0) rectangle (4,1) coordinate [midway] (A) node[below] {$p_n$};
  \draw (0.1, -0.2) -- +(0,1.4);
  \draw (0.6, -0.2) -- +(0,1.4);
  \draw (1.1,  -0.2) -- +(0,1.4);
  \node at (2.25, -0.2) {$\dots$};
  \node at (2.25, 1.2) {$\dots$};
  \draw (3.4, -0.2) -- +(0,1.4);
  \draw (3.9, -0.2) -- +(0,1.4);
  \filldraw[fill=white] (0,0) rectangle (4,1) node[pos=.5] {$p_n$};
  
\end{scope}
%%% Local Variables:
%%% mode: latex
%%% TeX-master: t
%%% End:
}} \ .
\]
We also will make use of the following endomorphism of $c^{\otimes n}$:
\begin{equation}\label{eq:defzn} z_n =
  \NB{\tikz[scale = 0.6]{\begin{scope}
%\coordinate (A1) at (0,0);
%\coordinate (A2) at (0,2);
\coordinate (B1) at (1,0);
\coordinate (B2) at (1,2);
\coordinate (A1) at (0,0);
\coordinate (A2) at (0,2);
\coordinate (C1) at (3,0);
\coordinate (C2) at (3,2);

    \draw[] (A1) -- (A2) coordinate[pos =.5] (A3) node[below, pos =0] {};
    \filldraw[draw= black!50!black, fill = black] (A3) circle (1mm);
     \draw[] (B1) -- (B2) node[above, pos =1] {} node[right, pos =.5] {$\cdots$};
       \draw[] (C1) -- (C2) node[above, pos =1] {} node[below, pos =0] {};

\end{scope}}}
  -
  \NB{\tikz[scale = 0.6]{\begin{scope}
\coordinate (A1) at (0,0);
\coordinate (A2) at (0,2);
\coordinate (B1) at (1,0);
\coordinate (B2) at (1,2);
\coordinate (A1) at (0,0);
\coordinate (A2) at (0,2);
\coordinate (C1) at (3,0);
\coordinate (C2) at (3,2);

    \draw[] (A1) -- (A2) coordinate[pos =.5] (A3) node[below, pos =0] {};
   % \filldraw[draw= black!50!black, fill = black] (A3) circle (1mm);
     \draw[] (B1) -- (B2) coordinate[pos =.5] (B3)
     node[above, pos =1] {} node[right, pos =.5] {$\cdots$};
         \filldraw[draw= black!50!black, fill = black] (B3) circle (1mm);
       \draw[] (C1) -- (C2) node[above, pos =1] {} node[below, pos =0] {};

\end{scope}}}
  + \cdots + (-1)^{n-1}
    \NB{\tikz[scale = 0.6]{\begin{scope}
\coordinate (A1) at (0,0);
\coordinate (A2) at (0,2);
\coordinate (B1) at (1,0);
\coordinate (B2) at (1,2);
\coordinate (A1) at (0,0);
\coordinate (A2) at (0,2);
\coordinate (C1) at (3,0);
\coordinate (C2) at (3,2);

    \draw[] (A1) -- (A2) coordinate[pos =.5] (A3) node[below, pos =0] {};
   % \filldraw[draw= black!50!black, fill = black] (A3) circle (1mm);
     \draw[] (B1) -- (B2) coordinate[pos =.5] (B3)
     node[above, pos =1] {} node[right, pos =.5] {$\cdots$};
       %  \filldraw[draw= black!50!black, fill = black] (B3) circle (1mm);
       \draw[] (C1) -- (C2) coordinate[pos =.5] (C3) node[above, pos =1] {} node[below, pos =0] {};
         \filldraw[draw= black!50!black, fill = black] (C3) circle (1mm);
\end{scope}}} \ .
\end{equation}

\begin{lem} \label{lem:fonp_2}
One has
\[
\de(p_2)=\dh(p_2)=0 \ , \quad 
\df(p_2)=\frac{1}{2}a_1
\left(~
 \NB{\tikz[scale = 0.6]{}}
    -
     \NB{\tikz[scale = 0.6]{}}
~\right) \ .
\]

As a left (resp.{} right) $\dTL$-module, $\dTL p_2$  
(resp.{} $p_2 \dTL$) is not stable under $\df$ unless $a_1=0$.
\end{lem}

\begin{proof}
    The actions of $\de, \dh, \df$ are calculated in a straightforward manner.
    
    Using the relations in the category, one could check that
$\df(p_2) p_2$ and $\df(p_2)(1-p_2)$ are both non-zero unless $a_1=0$.  Similarly 
$p_2 \df(p_2) $ and $(1-p_2)\df(p_2)$ are both non-zero unless $a_1=0$.

\end{proof}

\subsection{Karoubi envelope}
In this section we let $\Bbbk=\mathbb{Q}$.
In the Karoubi envelope of $\dTL$, let $P_n=(c^{\otimes n},p_n)$ where $p_n$ was defined in \eqref{eq:defpn}.
Furthermore, we  assume $a_1=0$ in order to use Lemma \ref{lem:fonp_2}.

Define maps $ U_n \colon P_n \rightarrow P_{n+2}$ and $ D_n \colon P_n \rightarrow P_{n-2}$ by
\begin{equation}
    U_n =
     \NB{\tikz[scale = 0.6]{\begin{scope}

  \draw[-] (.5,-0.3) -- (.5,3.3);
    \draw[-] (3.5,-0.3) -- (3.5,3.3);
        \draw[-] (6,3) -- (6,3.3);
         \draw[-] (5,3) -- (5,3.3);
  \node at (2, -0.2) {$\dots$};
  \node at (2, 3.2) {$\dots$};
    \node at (2, 1.5) {$\dots$};
  \filldraw[fill=white] (0,0) rectangle (4,1) node[pos=.5] {$p_n$};
 % \fill (A) circle (0.5mm) node[below] {$p_n$};
   \filldraw[fill=white] (0,2) rectangle (6.5,3) node[pos=.5] {$p_{n+2}$};
      \draw  (5, 2) arc (-180:0:0.5) coordinate[pos = 0.5] (Z) ;
   %\draw (.5, -1) arc (180:0:0.5) coordinate[pos = 0.2] (X) coordinate[pos = 0.8] (Y);
    \filldraw[draw= black!50!black, fill = black] (Z) circle (1mm);

       % \draw (0,0) rectangle (4,1) node[pos=.5] {$p_n$};

\end{scope}}}
    \end{equation}

    \begin{equation}
    D_n = n(n-1)
     \NB{\tikz[scale = 0.6]{\begin{scope}
    \filldraw[draw= black!50!black, fill = black] (Z) circle (1mm);
  \draw[-] (.5,-0.3) -- (.5,3.3) ;
    \draw[-] (3.5,-0.3) -- (3.5,3.3) ;
        \draw[-] (6,-0.3) -- (6,0) ;
          \draw[-] (5,-0.3) -- (5,0) ;
  \node at (2, -0.2) {$\dots$};
  \node at (2, 3.2) {$\dots$};
    \node at (2, 1.5) {$\dots$};
     \filldraw[fill=white] (0,0) rectangle (6.5,1) node[pos=.5] {$p_n$};
 % \fill (A) circle (0.5mm) node[below] {$p_n$};
 \filldraw[fill=white] (0,2) rectangle (4,3) node[pos=.5] {$p_{n-2}$};
      \draw [] (5, 1) arc (180:0:0.5) coordinate[pos = 0.5] (Z) ;
   %\draw (.5, -1) arc (180:0:0.5) coordinate[pos = 0.2] (X) coordinate[pos = 0.8] (Y);
\end{scope}}}
    \end{equation}

\begin{lem} \label{lem:sl2onUDz}
Assume $a_1=0$.  Then
$\mathfrak{sl}_2$ acts on $U_n$, $D_n$, and $z_n$ as follows:    
\begin{equation}
    \de(U_n)=0 \ , \quad \quad \df(U_n) = (1-a_2) E_1 U_n \ , \quad \quad
    \dh(U_n) = (2a_2-2) U_n \ ,
\end{equation}
\begin{equation}
    \de(D_n)=0 \ , \quad \quad \df(D_n) = (1+a_2) E_1 D_n \ , \quad \quad
    \dh(D_n) = (-2a_2-2) D_n \ ,
\end{equation}
\begin{equation}
    \de(z_n)=\frac{(-1)^{n}-1}{2} \ , \quad \quad 
    \df(z_n) = E_1 z_n + (\frac{(-1)^{n}-1}{2}) E_2 \ , \quad \quad
    \dh(z_n) = -2 z_n \ .
\end{equation}
\end{lem}

\begin{proof}
    This is a straightforward calculation using relations in the category.
\end{proof}

We now enhance our category by adding twisted objects $P_n^{a E_1}$ for each object $P_n$ and each $a \in \frac{1}{2} \mathbb{Z}$.
Viewing $\End(P_n^{a E_1})$ as a rank 1 module over itself generated by a vector $1_a$, the $\mathfrak{sl}_2$-action is twisted by
\[
f \cdot 1_a = aE_1 1_{a} \ , \quad \quad
h \cdot 1_a = -2a 1_a \ .
\]

Following \cite{HRW}, define the Kirby $k$-color  $\omega_k$ as a colimit of the sequence of objects 
\begin{equation} \label{def:kirbycolor}
    \omega_k = \colim \left(
\NB{\tikz[xscale =3, yscale=2]{
  \node (A) at (-1,2) {$q^{-k} P_k^{-\frac{k}{2}(1-a_2)E_1}$};
   \node (B) at (.5,2) {$q^{-k-2}P^{-\frac{k+2}{2}(1-a_2)E_1}_{k+2}$};
    \node (C) at (1.5,2) {$\cdots$};
\draw[-to] (A) -- (B) node[pos=0.5, above, scale = 0.7] {$U_k$};
         \draw[-to] (B) -- (C) node[pos=0.5, above, scale = 0.7] {$U_{k+2}$};
}}
    \right).
\end{equation}
Note that with the twisting, each $U_{k+2j}$ is annihilated by the $\mathfrak{sl}_2$-action, so that in the directed system defining $\omega_k$, the morphims are $\sll_2$-equivariant.

We could now describe the Karoubi envelope of $\dTL$ over $\mathbb{Q}$ as the module category of an algebra with 2 blocks $C_0$ and $C_1$.  

\begin{lem} \label{lem:quiver}
The graded algebras $C_0$ and $C_1$ are quotients of the path algebras of the quivers in \eqref{eq:A0} and \eqref{eq:A1} respectively where the degree of each edge is 2, subject to relations
\[
D_{n+2} U_{n} = -z_{n}^2 \hspace{.1in} \Mod ~(E_1,E_2) \ , \quad 
U_{n-2} D_{n} = -z_{n}^2 \hspace{.1in} \Mod ~(E_1,E_2) \ , \quad 
\]
\[
z_n^{n+1} = 0~ \Mod~ (E_1,E_2) \ , \quad 
z_n U_{n-2} = U_{n-2} z_{n-2} \ , \quad
z_n D_{n+2} = D_{n+2} z_{n+2} \ .
\]

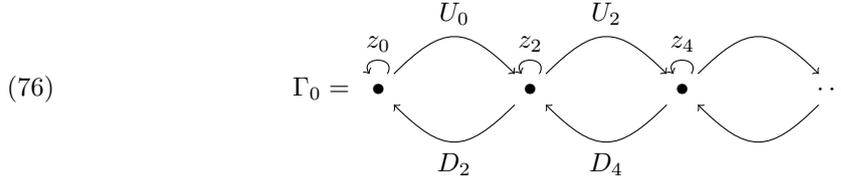
\begin{equation} \label{eq:A0}
     \Gamma_0 = 
     \NB{\tikz[xscale =2, yscale=2]{
       \node (A) at (-1,2) {$\bullet$};
              \node (A1) at (-1,2.3) {$z_0$};
         \node (B1) at (0,2.3) {$z_2$};
                  \node (C1) at (1,2.3) {$z_4$};
                 \node (X1) at (-.5,2.5) {$U_0$};
                    \node (X2) at (-.5,1.5) {$D_2$};
              \node (Y1) at (.5,2.5) {$U_2$};
                    \node (Y2) at (.5,1.5) {$D_4$}; 
                                      \node (B) at (0,2) {$\bullet$};
       \node (C) at (1,2) {$\bullet$};
        \node (D) at (2,2) {$\cdots$};
        \draw[->] (A) to[out=45, in=135, looseness=1.5] (B);
         \draw[->] (A) to[out=60, in=120, looseness=3] (A);
 \draw[<-] (A) to[out=-45, in=-135, looseness=1.5] (B);
     \draw[->] (B) to[out=45, in=135, looseness=1.5] (C);
             \draw[<-] (B) to[out=-45, in=-135, looseness=1.5] (C);
               \draw[->] (B) to[out=60, in=120, looseness=3] (B);
        \draw[->] (C) to[out=45, in=135, looseness=1.5] (D);
                 \draw[<-] (C) to[out=-45, in=-135, looseness=1.5] (D);
                          \draw[->] (C) to[out=60, in=120, looseness=3] (C);
 }}
     \end{equation}

     \begin{equation} \label{eq:A1}
     \Gamma_1 = 
     \NB
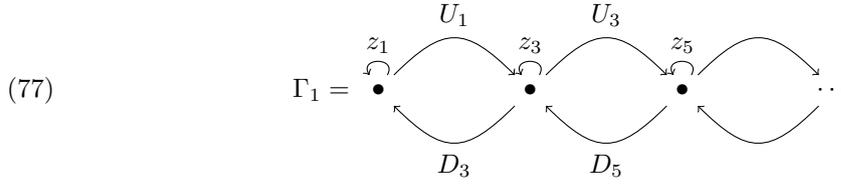
{\tikz[xscale =2, yscale=2]{
       \node (A) at (-1,2) {$\bullet$};
              \node (A1) at (-1,2.3) {$z_1$};
         \node (B1) at (0,2.3) {$z_3$};
                  \node (C1) at (1,2.3) {$z_5$};
                 \node (X1) at (-.5,2.5) {$U_1$};
                    \node (X2) at (-.5,1.5) {$D_3$};
              \node (Y1) at (.5,2.5) {$U_3$};
                    \node (Y2) at (.5,1.5) {$D_5$}; 
                                      \node (B) at (0,2) {$\bullet$};
       \node (C) at (1,2) {$\bullet$};
        \node (D) at (2,2) {$\cdots$};
        \draw[->] (A) to[out=45, in=135, looseness=1.5] (B);
         \draw[->] (A) to[out=60, in=120, looseness=3] (A);
 \draw[<-] (A) to[out=-45, in=-135, looseness=1.5] (B);
     \draw[->] (B) to[out=45, in=135, looseness=1.5] (C);
             \draw[<-] (B) to[out=-45, in=-135, looseness=1.5] (C);
               \draw[->] (B) to[out=60, in=120, looseness=3] (B);
        \draw[->] (C) to[out=45, in=135, looseness=1.5] (D);
                 \draw[<-] (C) to[out=-45, in=-135, looseness=1.5] (D);
                          \draw[->] (C) to[out=60, in=120, looseness=3] (C);
 }}
     \end{equation}

It inherits an $\mathfrak{sl}_2$-action from Lemma \ref{lem:sl2onUDz}.
\end{lem}

\begin{proof}
For the algebra $C_0$, the vertices from left to right correspond to objects $P_0, P_2, \cdots$.
For the algebra $C_1$ the vertices from left to right correspond to objects $P_1, P_3, \cdots$.
\end{proof}

\subsection{Module category}
Let $\MdTL$ be the $\mathbb{Z}$-graded $\Bbbk[E_1,E_2]$-linear category with objects $(n,Z) \in \mathbb{N} \times \{L,R\}$ generated as a module category over $\dTL$ by morphisms
\[
  \NB{\tikz[scale = 0.47, font=\small]{\begin{scope}
  \draw [blue] (0, -1) -- +(0,2) node[above, pos
    = 1] {$Z$} node[below, pos
    = 0] {$Z$};
 %   \coordinate (A) at (0, 0);
  %      \filldraw[draw= green!50!black, fill = white] (A) circle (1mm)
  %node[left, green!50!black] 
  %%
     \coordinate (A) at (0, 0);
        \filldraw[draw= black!50!black, fill = black] (A) circle (1mm);
%  node[left, green!50!black] 
 % {$-\frac{1}{2}(N-1)$};

 %\draw[->-] (ML) -- (MR) node[above, midway] {$a+b$};
 
\end{scope}}} \ , \quad \quad
    \NB{\tikz[scale = 0.47, font=\small]{\begin{scope}
  \draw [blue] (0, -1) -- +(0,2) node[above, pos
    = 1] {$Y$} node[below, pos
    = 0] {$Z$};
     \coordinate (A) at (0, 0);
          \coordinate (B) at (-1, -1);
\draw [] (B) -- (A);
     %   \filldraw[draw= black!50!black, fill = black] (A) circle (1mm);

\end{scope}}} \ , \quad \quad
    \NB{\tikz[scale = 0.47, font=\small]{\begin{scope}
  \draw [blue] (0, -1) -- +(0,2) node[above, pos
    = 1] {$Z$} node[below, pos
    = 0] {$Y$};
     \coordinate (A) at (0, 0);
          \coordinate (B) at (-1, 1);
\draw [] (B) -- (A);
     %   \filldraw[draw= black!50!black, fill = black] (A) circle (1mm);

\end{scope}}} \ ,
\]
where $Y \neq Z$ subject to relations that
the last generator 
could be obtained from the other by adjunction in $\dTL$, and

\begin{equation} \label{eq:mdtlrel1}
  \NB{\tikz[scale = 0.47, font=\small]{\begin{scope}
  \draw [blue] (0, -1) -- +(0,2) node[above, pos
    = 1] {$Z$} node[below, pos
    = 0] {$Z$};
     \coordinate (A) at (0, -.5);
          \coordinate (C) at (0, .5);
 \coordinate (B) at (-1, -1);
 \coordinate (D) at (-1, 1);
\draw [] (B) -- (A);
\draw [] (D) -- (C);
     %   \filldraw[draw= black!50!black, fill = black] (A) circle (1mm);

\end{scope}}} = 
    \NB{\tikz[scale = 0.47, font=\small]{\begin{scope}
  \draw (-1,-1) -- +(0,2);
  \draw [blue] (0, -1) -- +(0,2) node[above, pos
    = 1] {$Z$} node[below, pos
    = 0] {$Z$};
 %   \coordinate (A) at (0, 0);
  %      \filldraw[draw= green!50!black, fill = white] (A) circle (1mm)
  %node[left, green!50!black] 
  %%
     \coordinate (A) at (0, 0);
        \filldraw[draw= black!50!black, fill = black] (A) circle (1mm);
%  node[left, green!50!black] 
 % {$-\frac{1}{2}(N-1)$};

 %\draw[->-] (ML) -- (MR) node[above, midway] {$a+b$};
 
\end{scope}}} +
    \NB{\tikz[scale = 0.47, font=\small]{\begin{scope}
  \draw (-1,-1) -- +(0,2);
  \draw [blue] (0, -1) -- +(0,2) node[above, pos
    = 1] {$Z$} node[below, pos
    = 0] {$Z$};
 %   \coordinate (A) at (0, 0);
  %      \filldraw[draw= green!50!black, fill = white] (A) circle (1mm)
  %node[left, green!50!black] 
  %%
     \coordinate (A) at (-1, 0);
        \filldraw[draw= black!50!black, fill = black] (A) circle (1mm);
%  node[left, green!50!black] 
 % {$-\frac{1}{2}(N-1)$};

 %\draw[->-] (ML) -- (MR) node[above, midway] {$a+b$};
 
\end{scope}}} - E_1
    \NB{\tikz[scale = 0.47, font=\small]{\begin{scope}
  \draw (-1,-1) -- +(0,2);
  \draw [blue] (0, -1) -- +(0,2) node[above, pos
    = 1] {$Z$} node[below, pos
    = 0] {$Z$};
 %    \coordinate (A) at (-1, 0);
  %      \filldraw[draw= black!50!black, fill = black] (A) circle (1mm);
%  node[left, green!50!black] 
 % {$-\frac{1}{2}(N-1)$};

 %\draw[->-] (ML) -- (MR) node[above, midway] {$a+b$};
 
\end{scope}}} \ , 
    \quad \quad
    \NB{\tikz[scale = 0.47, font=\small]{\begin{scope}
  \draw [blue] (0, -1) -- +(0,2) node[above, pos = 1] {$Z$} node[below, pos= 0] {$Z$};
 %   \coordinate (A) at (0, 0);
  %      \filldraw[draw= green!50!black, fill = white] (A) circle (1mm)
  %node[left, green!50!black] 
  %%
  \coordinate (A) at (0, 0.3);
  \coordinate (A1) at (0, -0.3);
        \filldraw[draw= black!50!black, fill = black] (A) circle (1mm); % node[left, green!50!black] {$2$};
        \filldraw[draw= black!50!black, fill = black] (A1) circle (1mm);  %node[left, green!50!black] {$2$};

 %\draw[->-] (ML) -- (MR) node[above, midway] {$a+b$};
 
\end{scope}}} =
    E_1 \NB{\tikz[scale = 0.47, font=\small]{}}
    - E_2 \NB{\tikz[scale = 0.47, font=\small]{\begin{scope}
  \draw [blue] (0, -1) -- +(0,2) node[above, pos
    = 1] {$Z$} node[below, pos
    = 0] {$Z$};
 %   \coordinate (A) at (0, 0);
  %      \filldraw[draw= green!50!black, fill = white] (A) circle (1mm)
  %node[left, green!50!black] 
  %%
  %   \coordinate (A) at (0, 0);
   %     \filldraw[draw= black!50!black, fill = black] (A) circle (1mm);
%  node[left, green!50!black] 
 % {$-\frac{1}{2}(N-1)$};

 %\draw[->-] (ML) -- (MR) node[above, midway] {$a+b$};
 
\end{scope}}} \ ,
\end{equation}

\begin{equation} \label{eq:mdtlrel2}
  \NB{\tikz[scale = 0.47, font=\small]{\begin{scope}
  \draw [blue] (0, -1) -- +(0,2) node[above, pos
    = 1] {$Y$} node[below, pos
    = 0] {$Z$};
          \coordinate (B) at (-1, -1);
\draw [] (B) -- (A) coordinate[pos = 0.5] (X);
       \filldraw[draw= black!50!black, fill = black] (X) circle (1mm);
 
\end{scope}}} = 
   \NB{\tikz[scale = 0.47, font=\small]{\begin{scope}
  \draw [blue] (0, -1) -- +(0,2) node[above, pos
    = 1] {$Y$} node[below, pos
    = 0] {$Z$} coordinate[pos = 0.25] (X);
          \coordinate (B) at (-1, -1);
\draw [] (B) -- (A);
       \filldraw[draw= black!50!black, fill = black] (X) circle (1mm);
 
\end{scope}}} =
    \NB{\tikz[scale = 0.47, font=\small]{\begin{scope}
  \draw [blue] (0, -1) -- +(0,2) node[above, pos
    = 1] {$Y$} node[below, pos
    = 0] {$Z$} coordinate[pos = 0.75] (X);
          \coordinate (B) at (-1, -1);
\draw [] (B) -- (A);
       \filldraw[draw= black!50!black, fill = black] (X) circle (1mm);
 
\end{scope}}} \ .
\end{equation}

\begin{lem}
    For any $\alpha_1, \alpha_2 \in \Bbbk$,
    there is an action of $\mathfrak{sl}_2$ on $\MdTL$ determined by
\allowdisplaybreaks
\begin{align*}  
&\de
\left( \NB{\tikz[scale = 0.47, font=\small]{}} \right)
  =0,
  &&
\de
\left( \NB{\tikz[scale = 0.47, font=\small]{}} \right)
=
0 ,\\
&\dh
\left( \NB{\tikz[scale = 0.47, font=\small]{\begin{scope}
  \draw [blue] (0, -1) -- +(0,2) node[above, pos
    = 1] {$R$} node[below, pos
    = 0] {$L$};
     \coordinate (A) at (0, 0);
          \coordinate (B) at (-1, -1);
\draw [] (B) -- (A);
     %   \filldraw[draw= black!50!black, fill = black] (A) circle (1mm);

\end{scope}}} \right)
=
(-\alpha_1-2\alpha_2) \NB{\tikz[scale = 0.47, font=\small]{}} ,
&\quad&
\dh
\left( \NB{\tikz[scale = 0.47, font=\small]{\begin{scope}
  \draw [blue] (0, -1) -- +(0,2) node[above, pos
    = 1] {$L$} node[below, pos
    = 0] {$R$};
     \coordinate (A) at (0, 0);
          \coordinate (B) at (-1, -1);
\draw [] (B) -- (A);
     %   \filldraw[draw= black!50!black, fill = black] (A) circle (1mm);

\end{scope}}} \right)
=
(-2+\alpha_1+2\alpha_2-a_1-2a_2) \NB{\tikz[scale = 0.47, font=\small]{}} ,
\\
&\dh
\left( \NB{\tikz[scale = 0.47, font=\small]{\begin{scope}
  \draw [blue] (0, -1) -- +(0,2) node[above, pos
    = 1] {$R$} node[below, pos
    = 0] {$L$};
     \coordinate (A) at (0, 0);
          \coordinate (B) at (-1, 1);
\draw [] (B) -- (A);
     %   \filldraw[draw= black!50!black, fill = black] (A) circle (1mm);

\end{scope}}} \right)
=
(-\alpha_1-2\alpha_2+a_1+2a_2) \NB{\tikz[scale = 0.47, font=\small]{}}
,&\quad&
\dh
\left( \NB{\tikz[scale = 0.47, font=\small]{\begin{scope}
  \draw [blue] (0, -1) -- +(0,2) node[above, pos
    = 1] {$L$} node[below, pos
    = 0] {$R$};
     \coordinate (A) at (0, 0);
          \coordinate (B) at (-1, 1);
\draw [] (B) -- (A);
     %   \filldraw[draw= black!50!black, fill = black] (A) circle (1mm);

\end{scope}}} \right)
=
(-2+\alpha_1+2\alpha_2) \NB{\tikz[scale = 0.47, font=\small]{}} ,
\\
&\df
\left( \NB{\tikz[scale = 0.47, font=\small]{}} \right)
=
\alpha_1 \NB{\tikz[scale = 0.47, font=\small]{\begin{scope}
  \draw [blue] (0, -1) -- +(0,2) node[above, pos
    = 1] {$R$} node[below, pos
    = 0] {$L$};
          \coordinate (B) at (-1, -1);
\draw [] (B) -- (A) coordinate[pos = 0.5] (X);
       \filldraw[draw= black!50!black, fill = black] (X) circle (1mm);
 
\end{scope}}}
 +
\alpha_{2} E_1 \NB{\tikz[scale = 0.47, font=\small]{}} ,
&\quad&
\df
\left( \NB{\tikz[scale = 0.47, font=\small]{}} \right)
=
(a_1-\alpha_1) \NB{\tikz[scale = 0.47, font=\small]{\begin{scope}
  \draw [blue] (0, -1) -- +(0,2) node[above, pos
    = 1] {$L$} node[below, pos
    = 0] {$R$};
          \coordinate (B) at (-1, -1);
\draw [] (B) -- (A) coordinate[pos = 0.5] (X);
       \filldraw[draw= black!50!black, fill = black] (X) circle (1mm);
 
\end{scope}}}
 +
(a_2+1-\alpha_{2}) E_1 \NB{\tikz[scale = 0.47, font=\small]{}} ,
\\
&\df
\left( \NB{\tikz[scale = 0.47, font=\small]{}} \right)
=
-\alpha_1 \NB{\tikz[scale = 0.47, font=\small]{\begin{scope}
  \draw [blue] (0, -1) -- +(0,2) node[above, pos
    = 1] {$L$} node[below, pos
    = 0] {$R$};
     \coordinate (A) at (0, 0);
          \coordinate (B) at (-1, 1);
\draw [] (B) -- (A) coordinate[pos = 0.5] (X);
       \filldraw[draw= black!50!black, fill = black] (X) circle (1mm);

\end{scope}}}
 +
(1-\alpha_{2}) E_1 \NB{\tikz[scale = 0.47, font=\small]{}} ,
&\quad&
\df
\left( \NB{\tikz[scale = 0.47, font=\small]{}} \right)
=
(\alpha_1-a_1) \NB{\tikz[scale = 0.47, font=\small]{\begin{scope}
  \draw [blue] (0, -1) -- +(0,2) node[above, pos
    = 1] {$R$} node[below, pos
    = 0] {$L$};
     \coordinate (A) at (0, 0);
          \coordinate (B) at (-1, 1);
\draw [] (B) -- (A) coordinate[pos = 0.5] (X);
       \filldraw[draw= black!50!black, fill = black] (X) circle (1mm);

\end{scope}}}
 +
(\alpha_{2}-a_2) E_1 \NB{\tikz[scale = 0.47, font=\small]{}} .
\end{align*}
\allowdisplaybreaks[0]
Furthermore $\dh$ is a negative degree operator if
$a_1+2a_2=0 $ and $\alpha_1+2\alpha_2=1$.
\end{lem}

\begin{proof}
The coefficients are uniquely determined by checking that relations in the category are preserved and that $\mathfrak{sl}_2$-relations hold.
\end{proof}

Recall that the Kirby $k$-color $\omega_k$ was defined by the directed system \eqref{def:kirbycolor}.
\begin{prop} \label{prop:handleslide}
For all non-negative $k$, one has:
\[
\omega_k L \cong \omega_{k+1} R \qquad \textrm{and,} \qquad
\omega_k R \cong \omega_{k+1} L
\]
and these isomorphisms intertwine the the $\mathfrak{sl}_2$-action
if $a_1=\alpha_1=0$ and $2\alpha_2-a_2=1$. 
\end{prop}
\begin{rmk}
The element $\dh$ acts as a negative degree operator if and only if
$\alpha_1=a_1=a_2=0$ and $\alpha_2=\frac{1}{2}$.
\end{rmk}
\begin{proof}
  Both isomorphisms are similar and so we give the details only for the first case.
  The proof is similar to that of \cite[Lemma 4.13]{HRW}.
For $k \in \NN$, consider the map $f_k \colon \omega_k L \rightarrow \omega_{k+1}R$ given by:
\begin{equation}
\NB{\tikz[xscale =3, yscale=2]{
  \node (A) at (-1,0) {$q^{-k} P_k^{-\frac{k}{2}(1-a_2)E_1} L$};
   \node (B) at (1,0) {$q^{-k-2}P^{-\frac{k+2}{2}(1-a_2)E_1}_{k+2} L$};
    \node (C) at (-1,1) {$q^{-k-1}P^{-\frac{k+1}{2}(1-a_2)E_1}_{k+1} R$};
        \node (D) at (1,1) {$q^{-k-3}P^{-\frac{k+3}{2}(1-a_2)E_1}_{k+3} R$};
              \node (E) at (2,0) {$\cdots$};
                            \node (F) at (2,1) {$\cdots$};
\draw[-to] (A) -- (B) node[pos=0.5, above, scale = 0.7] {$U_k$};
         \draw[-to] (C) -- (D) node[pos=0.5, above, scale = 0.7] {$U_{k+1}$};
          \draw[-to] (A) -- (C) node[pos=0.5, left, scale = 0.7] {\NB{\tikz[scale = 0.6]{\begin{scope}
  \draw (0,0) rectangle (4,1) node[pos=.5] {$p_{k+1}$};
 % \fill (A) circle (0.5mm) node[below] {$p_n$};
%   \draw (0,2) rectangle (6,3) coordinate [midway] (A) node[below] {$p_{n+2}$};
    %  \draw [] (5, 2) arc (-180:0:0.5) coordinate[pos = 0.5] (Z) ;
   %\draw (.5, -1) arc (180:0:0.5) coordinate[pos = 0.2] (X) coordinate[pos = 0.8] (Y);
   % \filldraw[draw= black!50!black, fill = black] (Z) circle (1mm);
\draw[-] (.25,1) -- (.25,2) node[above, pos =1] {} node[below, pos =0] {};
\draw[-] (.25,-1) -- (.25,0) node[above, pos =1] {} node[below, pos =0] {};
    \draw[-] (3,1) -- (3,2) node[above, pos =1] {} node[below, pos =0] {};
        \draw[-] (3,-1) -- (3,0) node[above, pos =1] {} node[below, pos =0] {};
             \draw[-] (4.5,-1) -- (4.5,2) node[above, pos =1] {} node[below, pos =0] {};
               \draw[-] (4.5,-.5) -- (3.8,0) node[above, pos =1] {} node[below, pos =0] {};
                \draw[-] (3.8,1) -- (3.8,2) node[above, pos =1] {} node[below, pos =0] {};
    %  \coordinate(D) at (3, .5) node[left] {$\cdots$}; 
\end{scope}}}};
           \draw[-to] (B) -- (D) node[pos=0.5, right, scale = 0.7] {\NB{\tikz[scale = 0.6]{\begin{scope}
  \draw (0,0) rectangle (4,1) node[pos=.5] {$p_{k+3}$};
 % \fill (A) circle (0.5mm) node[below] {$p_n$};
%   \draw (0,2) rectangle (6,3) coordinate [midway] (A) node[below] {$p_{n+2}$};
    %  \draw [] (5, 2) arc (-180:0:0.5) coordinate[pos = 0.5] (Z) ;
   %\draw (.5, -1) arc (180:0:0.5) coordinate[pos = 0.2] (X) coordinate[pos = 0.8] (Y);
   % \filldraw[draw= black!50!black, fill = black] (Z) circle (1mm);
\draw[-] (.25,1) -- (.25,2) node[above, pos =1] {} node[below, pos =0] {};
\draw[-] (.25,-1) -- (.25,0) node[above, pos =1] {} node[below, pos =0] {};
    \draw[-] (3,1) -- (3,2) node[above, pos =1] {} node[below, pos =0] {};
        \draw[-] (3,-1) -- (3,0) node[above, pos =1] {} node[below, pos =0] {};
             \draw[-] (4.5,-1) -- (4.5,2) node[above, pos =1] {} node[below, pos =0] {};
               \draw[-] (4.5,-.5) -- (3.8,0) node[above, pos =1] {} node[below, pos =0] {};
                \draw[-] (3.8,1) -- (3.8,2) node[above, pos =1] {} node[below, pos =0] {};
    %  \coordinate(D) at (3, .5) node[left] {$\cdots$}; 
\end{scope}}}};
           \draw[-to] (B) -- (E) node[pos=0.5, above, scale = 0.7] {};
             \draw[-to] (D) -- (F) node[pos=0.5, above, scale = 0.7] {};
}}
\end{equation}
The composition of two such maps is 
\[
\NB{\tikz[scale = 0.5]{\begin{scope}
  \draw (0,0) rectangle (4,1) node[pos=.5] {$p_{k+1}$};
 % \fill (A) circle (0.5mm) node[below] {$p_n$};
%   \draw (0,2) rectangle (6,3) coordinate [midway] (A) node[below] {$p_{n+2}$};
    %  \draw [] (5, 2) arc (-180:0:0.5) coordinate[pos = 0.5] (Z) ;
   %\draw (.5, -1) arc (180:0:0.5) coordinate[pos = 0.2] (X) coordinate[pos = 0.8] (Y);
   % \filldraw[draw= black!50!black, fill = black] (Z) circle (1mm);
\draw[-] (.25,3) -- (.25,4) node[above, pos =1] {} node[below, pos =0] {};
\draw[-] (.25,1) -- (.25,2) node[above, pos =1] {} node[below, pos =0] {};
\draw[-] (.25,-1) -- (.25,0) node[above, pos =1] {} node[below, pos =0] {};
  \draw[-] (2,3) -- (2,4) node[above, pos =1] {} node[below, pos =0] {};
    \draw[-] (2,1) -- (2,2) node[above, pos =1] {} node[below, pos =0] {};
        \draw[-] (2,-1) -- (2,0) node[above, pos =1] {} node[below, pos =0] {};
             \draw[-] (4.5,-1) -- (4.5,4) node[above, pos =1] {} node[below, pos =0] {};
               \draw[-] (4.5,-.5) -- (3,0) node[above, pos =1] {} node[below, pos =0] {};
                    \draw[-] (4.5,1.5) -- (3.5,2) node[above, pos =1] {} node[below, pos =0] {};
                \draw[-] (3,1) -- (3,2) node[above, pos =1] {} node[below, pos =0] {};
                 \draw[-] (3,3) -- (3,4) node[above, pos =1] {} node[below, pos =0] {};
    %  \coordinate(D) at (3, .5) node[left] {$\cdots$}; 
     \draw (0,2) rectangle (4,3) node[pos=.5] {$p_{k+2}$};
\end{scope}}} = U_k
\]
using the first relation of \eqref{eq:mdtlrel1}.
This is the map in the directed system defining $\omega_k$, so we get the desired isomorphism.
\end{proof}

A formula for the skein lasagna module for 2-handlebodies was given by Manolescu and Neithalath \cite{ManNeith}.
We now recall the reformulation of the skein lasagna module of 2-handlebodies for $N=2$ due to Hogancamp, Rose, and Wedrich \cite{HRW} in the equivariant setting.
A connection between the dotted Temperley--Lieb category and the Bar-Natan category was constructed by Russell \cite{Russell} and studied further by Heyman \cite{Heyman}.

For a framed link $L \subset S^3$, there is a functor $\mathcal{F}_L \colon \dTL \rightarrow \Bbbk[E_1,E_2] \gmod$ (bigraded vector spaces) defined by:
\allowdisplaybreaks
\begin{align*}
c^m & \mapsto \KR_{2;t_1,t_2}^{\mathfrak{sl}_2}(\cab^m(L);\Bbbk[E_1,E_2]) \\
\NB{\tikz[scale = 0.6]{\begin{scope}
  \coordinate (bl) at (-0.5, -1);
  \coordinate (br) at ( 0.5, -1);
  \coordinate (tl) at (-0.5,  1);
  \coordinate (tr) at ( 0.5,  1);
    \coordinate (ml) at (-0.5,  -.8);
        \coordinate (Ml) at (-0.5,  .8);
 \coordinate (mr) at (0.5,  -.6);
\coordinate (Mr) at (0.5,  .6);
  \coordinate (brr) at ( .8, -1);
    \coordinate (trr) at ( .8, 1);

 %\draw[>->] (bl) -- (tl) node[pos = 0, below] {$2$} node[pos = 1,
  %above] {$1$} coordinate[pos = 0.4] (ml);
   \draw[] (bl) -- (tl)
   node[right, pos =.5] {$\cdots$}; 
   % node[pos = 0, below] {$1$} node[pos = 1,  above] {$1$};
  %\draw[>->] (br) -- (tr) node[pos = 0, below] {$1$} node[pos = 1, above] {$2$} coordinate[pos = 0.6] (mr);
  
  %  \draw[->-] (br) -- (tr); %node[pos = 0, below] {$1$} node[pos = 1, above] {$1$};
       \draw[] (brr) -- (trr);% node[pos = 0, below] {$1$} node[pos = 1, above] {$1$};
  
  %\draw[->-] (ml) -- (mr) node [pos= 0.5, above] {$1$};
  %  \draw[->-] (Ml) -- (Mr) node [pos= 0.5, above] {$1$};

\end{scope}}} & \mapsto
\Id \colon \KR_{2;t_1,t_2}^{\mathfrak{sl}_2}(\cab^m(L);\Bbbk[E_1,E_2]) \rightarrow
\KR_{2;t_1,t_2}^{\mathfrak{sl}_2}(\cab^m(L);\Bbbk[E_1,E_2]) \\
\NB{\tikz[scale = 0.6]{}} & \mapsto 
\NB{\tikz[scale = 0.5]{\begin{scope}
       % Parameters for the cylinder
    \def\R{3} % Radius of the cylinder
    \def\H{2} % Height of the cylinder

    % Draw the back-facing ellipse (bottom half) as a dashed line
    \draw[dashed] (0, 0) ellipse (\R cm and 0.5cm);

    % Draw the front and side vertical lines
    \draw (-\R, 0) -- (-\R, \H);
    \draw (\R, 0) -- (\R, \H);

    % Draw the front-facing ellipse (top half) as a solid line
    \draw (0, \H) ellipse (\R cm and 0.5cm);

    % Draw the front-facing ellipse (bottom half) as a solid line
    \draw (-\R, 0) arc (180:360:\R cm and 0.5cm);
    
    % Draw a dot on the front of the cylinder
    \filldraw[draw= black!50!black, fill = black] (-1.5, 0.75) circle (1mm);
\end{scope}}}
\colon
\KR_{2;t_1,t_2}^{\mathfrak{sl}_2}(L;\Bbbk[E_1,E_2])
\rightarrow
\KR_{2;t_1,t_2}^{\mathfrak{sl}_2}(L;\Bbbk[E_1,E_2])
\\
\NB{\tikz[scale = 0.6]{}} & \mapsto 
\NB{\tikz[scale = 0.5]{\begin{scope}[scale=1.5]
  \draw  (2, 0) arc (0:-180:2cm and 0.5cm);
  \draw[densely dashed] (2,0) arc (0:180:2cm and 0.5cm);
%  \draw[densely dashed] (0,0) circle (1.7cm and 0.425cm);
 % \node[font=\small] at (1.3, 0) {$\dots$};
 % \node[font =\small] at (-1.3, 0) {$\dots$};
 % \node[font=\tiny] at (1.95,0) {$1$};
 % \node[font=\tiny] at (1.65,0) {$1$};
 % \node[font=\tiny] at (0.95,0) {$1$};
 % \node[font=\tiny] at (-1.95,0) {$1$};
 % \node[font=\tiny] at (-1.65,0) {$1$};
%  \node[font=\tiny] at (-0.95,0) {$1$};
 % \node[font=\normalsize] at (1.6,1.5) {$k$};
 % \node[font=\normalsize] at (-1.6,1.5) {$k$};
  \draw[densely dashed] (0,0) circle (1cm and 0.25cm);
  \draw (1.5,1) arc (0:-180: 1.5cm and 0.375cm);
  \draw (1.5,1) arc (0:180: 1.5cm and 0.375cm);
 % \draw (0,2) circle (1.5cm and 0.375cm);
  \begin{scope}[very thin]
  \draw (1.5, 1) .. controls +(0.3, 0) and +(0, 0.3) .. (2, 0)
  coordinate[pos=0.6] (a) node[pos=0.35, sloped, font=\small, below] {}; 
%  \draw[densely dashed] (1.7,0) .. controls +(0,0.2) and +(-0.1, -0) .. (a);
  \draw (-1.5, 1) .. controls +(-0.3, 0) and +(0, 0.3) .. (-2, 0)
 coordinate[pos=0.6] (b) node[pos=0.35, sloped, font=\small, below] {}; 
 % \draw[densely dashed] (-1.7,0) .. controls +(0,0.2) and +(0.1, -0) .. (b); 
  \draw[densely dashed] (1.5, 1) .. controls +(-0.3, 0) and +(0, 0.3) .. (1, 0);
  \draw[densely dashed] (-1.5, 1) .. controls +(0.3, 0) and
  +(0, 0.3) .. (-1, 0); 
% \draw (1.5, 1) -- +(0,1);
%  \draw (-1.5, 1) -- +(0,1);
  \end{scope}
\end{scope}}}
\colon
\KR_{2;t_1,t_2}^{\mathfrak{sl}_2}(\cab^2(L_;\Bbbk[E_1,E_2])
\rightarrow
\KR_{2;t_1,t_2}^{\mathfrak{sl}_2}(\cab^0(L);\Bbbk[E_1,E_2])
\\
\NB{\tikz[scale = 0.6]{}} & \mapsto 
\NB{\tikz[scale = 0.5]{\begin{scope}[scale=1.5]
  \draw  (2, 0) arc (0:-180:2cm and 0.5cm);
  \draw (2,0) arc (0:180:2cm and 0.5cm);
%  \draw[densely dashed] (0,0) circle (1.7cm and 0.425cm);
 % \node[font=\small] at (1.3, 0) {$\dots$};
 % \node[font =\small] at (-1.3, 0) {$\dots$};
 % \node[font=\tiny] at (1.95,0) {$1$};
 % \node[font=\tiny] at (1.65,0) {$1$};
 % \node[font=\tiny] at (0.95,0) {$1$};
 % \node[font=\tiny] at (-1.95,0) {$1$};
 % \node[font=\tiny] at (-1.65,0) {$1$};
%  \node[font=\tiny] at (-0.95,0) {$1$};
 % \node[font=\normalsize] at (1.6,1.5) {$k$};
 % \node[font=\normalsize] at (-1.6,1.5) {$k$};
  \draw (0,0) circle (1cm and 0.25cm);
  \draw (1.5,-1) arc (0:-180: 1.5cm and 0.375cm);
  \draw (1.5,-1) arc (0:180: 1.5cm and 0.375cm);
 % \draw (0,2) circle (1.5cm and 0.375cm);
  \begin{scope}[very thin]
  \draw (1.5, -1) .. controls +(0.3, 0) and +(0, 0.3) .. (2, 0)
  coordinate[pos=0.6] (a) node[pos=0.35, sloped, font=\small, below] {}; 
%  \draw[densely dashed] (1.7,0) .. controls +(0,0.2) and +(-0.1, -0) .. (a);
  \draw (-1.5, -1) .. controls +(-0.3, 0) and +(0, 0.3) .. (-2, 0)
 coordinate[pos=0.6] (b) node[pos=0.35, sloped, font=\small, below] {}; 
 % \draw[densely dashed] (-1.7,0) .. controls +(0,0.2) and +(0.1, -0) .. (b); 
  \draw[densely dashed] (1.5, -1) .. controls +(-0.3, 0) and +(0, 0.3) .. (1, 0);
  \draw[densely dashed] (-1.5, -1) .. controls +(0.3, 0) and
  +(0, 0.3) .. (-1, 0); 
% \draw (1.5, 1) -- +(0,1);
%  \draw (-1.5, 1) -- +(0,1);
  \end{scope}
\end{scope}}}
\colon
\KR_{2;t_1,t_2}^{\mathfrak{sl}_2}(\cab^0(L_;\Bbbk[E_1,E_2])
\rightarrow
\KR_{2;t_1,t_2}^{\mathfrak{sl}_2}(\cab^2(L);\Bbbk[E_1,E_2])
\end{align*}
Define $\KR_{2;t_1,t_2}^{\mathfrak{sl}_2}(\Sym^m(L);\Bbbk[E_1,E_2])$ to be $\mathcal{F}_L(P_m)$ and
\begin{align*}
\mathcal{F}_L(\omega_k) = \colim_j  \KR_{2;t_1,t_2}^{\mathfrak{sl}_2}(\Sym^{k+2j}(L);\Bbbk[E_1,E_2]) := 
\KR_{2;t_1,t_2}^{\mathfrak{sl}_2}
(L^{\omega_k};\Bbbk[E_1,E_2])
\end{align*}

\allowdisplaybreaks[0]

Let $W$ be a 2-handlebody where the attaching 2-handles are determined by a framed oriented link $L$.  Let $L_0$ be an oriented link in the boundary $\partial W$.
Let $\omega=\oplus_k \omega_k$.

\begin{prop}
One has
\[
\KR_{2;t_1,t_2}^{\mathfrak{sl}_2}
(L^{\omega} \cup L_0;\Bbbk[E_1,E_2])
\cong
S_0^2(W,L_0;\Bbbk[E_1,E_2])
\] as $\mathfrak{sl}_2$-representations.
\end{prop}

\begin{proof}
    This formulation of the skein lasagna is due to Hogancamp, Rose, and Wedrich's \cite[Theorem 6.8]{HRW} interpretation of Manolescu and Neithalath's formula for 2-handlebodies \cite[Theorem 1.1]{ManNeith}.

    An inspection of the $\mathfrak{sl}_2$-action shows that the module in Theorem \ref{thm:sl2las} is the same as that on
   $ \KR_{2;t_1,t_2}^{\mathfrak{sl}_2}
(L^{\omega} \cup L_0;\Bbbk[E_1,E_2]) $.
\end{proof}

\begin{rmk}
    Proposition \ref{prop:handleslide} provides a proof of the compatibility of sliding a 2-handle over a 2-handle with the $\mathfrak{sl}_2$-action.  Note that this does not give an independent proof of invariance since it does not account for other handle moves.
\end{rmk}

\section{Examples}
\label{example:sec}
\newcommand{\mydiscriminant}{\ensuremath{\delta}}

In this section, we compute the $\mathfrak{sl}_2$-action on the skein lasagna for $\ball^4$ and $\ball^2 \times \sphere^2$ for the case $N=2$ and with ground ring equal to $\CC$.
As we shall see, in these two instances, the iso-classes of the obtained $\sll_2$-modules do not depend on the choice of $t_1$ and $t_2$.

\subsection{Preliminaries}

For $\lambda \in \ZZ$, $L(\lambda)$ denotes the irreducible $\sll_2$-module whose highest
weight is $\lambda$. If $\lambda \geq 0$, the module $L(\lambda)$ is finite-dimensional. Denote by $P(\lambda)$ the projective cover of $L(\lambda)$ and by $M(\lambda)$ the Verma module of highest weight $\lambda$.

The only extensions between Verma modules are given by:
\begin{equation} \label{eq:sesvermas}
 0 \rightarrow   M(\lambda) \rightarrow P(-\lambda-2) \rightarrow M(-\lambda-2) \rightarrow 0 
\end{equation}
for $\lambda \geq 0$.
Note also, that for $\lambda \geq 0$, there is a short exact sequence
\begin{equation} \label{eq:sesvermas2}
 0 \rightarrow   M(-\lambda-2) \rightarrow M(\lambda) \rightarrow L(\lambda) \rightarrow 0.
\end{equation}

Recall that the Zuckerman functor $\Gamma$ extracts from an $\mathfrak{sl}_2$-module $M$ its maximal locally finite submodule.  The dual Zuckerman functor $\mathcal{Z}$ (also known as Bernstein functor) maps an $\sll_2$-module to its maximal locally finite quotient.

Recall that the ring of symmetric polynomials $\mathbb{C}[E_1,E_2]$ has an $\mathfrak{sl}_2$-action determined by
\begin{nalign}
\dh(E_1)&=-2E_1, &&&&& \de(E_1)&=-2, &&&&& \df(E_1)&=E_1^2-2E_2, \\
  \dh(E_2)&=-4E_2, &&&&& \de(E_2)&=-E_1, &&&&& \df(E_2)&=E_1 E_2 \ .
\end{nalign}
This will be further studied in the next subsection.
 Throughout this section, let 
 \begin{equation} \label{eq:defalpha}
 \mydiscriminant=4E_2-E_1^2.
 \end{equation}
From straightforward computations, one obtains the following lemma.
 \begin{lem} \label{lem:alphajformulas}
One has the identities
\begin{nalign}
\dh(\mydiscriminant^j)&=-4j \mydiscriminant^j,&&&&& \de(\mydiscriminant^j)&=0,&&&&&\df(\mydiscriminant^j)&=2j E_1 \mydiscriminant^j \ . \qedhere
\end{nalign} 
\end{lem}

\begin{lem} \label{lem:fralphap}
    For any non-negative integers $r$ and $j$, $\df^r(\mydiscriminant^j) \neq 0$.
\end{lem}

\begin{proof}
    It suffices to track terms which are only powers of $E_1$.  One obtains, for $r>0$,
    \begin{equation}
    \df^r(\mydiscriminant^j) - (-1)^j
    (2j)(2j+1)\cdots (2j+r-1) E_1^{2j+r} \in  E_2 \mathbb{C}[E_1,E_2].    \qedhere \end{equation}
\end{proof}

\subsection{Computation for $\ball^4$}
The 4-ball is a connected 0-handlebody, so that the empty link is a Kirby diagram for $\ball^4$ and its skein lasagna module $S_0^2(\ball^4, \emptyset; \mathbb{C}[E_1,E_2])$ is canonically isomorphic to $\mathbb{C}[E_1,E_2]$.

\begin{lem}
For $j>1$, $\mydiscriminant^j$ generates a Verma module isomorphic to $M(-4j)$.
\end{lem}

\begin{proof}
    By Lemma \ref{lem:fralphap}, $\df^r$ never annihilates $\mydiscriminant^j$, and by Lemma \ref{lem:alphajformulas}, $\mydiscriminant^j$ is a highest weight vector.  Thus $\mydiscriminant^j$ generates a Verma module of highest weight $-4j$.
\end{proof}

\begin{lem}
The elements $1$ and 
$E_1$ generate a dual Verma module $M^*(0)$.
\end{lem}

\begin{proof}
The element $1$ spans a submodule isomorphic to $L(0)$ since it is annihilated by $\df, \de$ and has weight $0$.

Note that $\de(E_1)=-2 $ and $\df^r(E_1) \neq 0$.
Thus, modulo the submodule spanned by $1$, the element $E_1$ generates a Verma module $M(-2)$. Dualizing the short exact sequence \eqref{eq:sesvermas2} yields the lemma.
\end{proof}

\begin{prop}
As an $\mathfrak{sl}_2$-module, the skein lasagna module $S_0^2(\ball^4, \emptyset; \mathbb{C}[E_1,E_2])$  is isomorphic to \begin{equation}
\mathbb{C}[E_1,E_2] \cong 
    M^*(0) \oplus \bigoplus_{j=1}^{\infty} M(-4j) \ .
\end{equation}
\end{prop}

\begin{proof}
    The operator $\dh$ acts diagonally and its character is given by
\[
\left(\frac{1}{1-\lambda^{-2}}\right)\left( \frac{1}{1-\lambda^{-4}} \right) \ .
\]
Thus at every negative multiple of 4, the weight space grows by an extra dimension.  Since $\mydiscriminant^j$ is a highest weight vector, we get the desired decomposition.
\end{proof}

\begin{rmk}
    The dual Zuckerman functor $\mathcal{Z}$ annihilates $S_0^2(\ball^4, \emptyset; \mathbb{C}[E_1,E_2]) \cong \mathbb{C}[E_1,E_2]$ and the Zuckerman functor $\Gamma$ applied to the module yields ${S_0^2(\ball^4, \emptyset; \mathbb{C}[E_1,E_2])} \cong \Gamma \mathbb{C}[E_1,E_2] \cong L(0)$.
\end{rmk}

\subsection{Computation for $\ball^2 \times \sphere^2$}
\subsubsection{Basic facts for $\ball^2 \times \sphere^2$}
The $4$-manifold $\ball^2 \times \sphere^2$ could be realized by attaching a $2$-handle to the $4$-ball $\ball^4$ along the $0$-framed unknot $U$.
The equivariant Khovanov homology of the unknot is spanned over $\mathbb{Z}[E_1,E_2]$ by
\[
A_1=  \NB{\tikz[scale = 1]{\begin{scope}
  \draw (0,0) arc (180 :0: 0.5cm and 0.2cm) node[above, pos =
  0.5] {};
  \draw[very thin] (0,0) arc (180 :0: 0.5cm and -0.6cm) node[pos=0.5,
  above] {};
  \draw (0,0) arc (180 :0: 0.5cm and -0.2cm);
\end{scope}

%%% Local Variables:
%%% mode: latex
%%% TeX-master: t
%%% End:
}} \ , \quad \quad
 A_0= \NB{\tikz[scale = 1]{\begin{scope}
\coordinate (mr) at (.5, -.5);
  \draw (0,0) arc (180 :0: 0.5cm and 0.2cm) node[above, pos =
  0.5] {};
  \draw[very thin] (0,0) arc (180 :0: 0.5cm and -0.6cm);
%  node[pos=0.5, above] \filldraw circle (0.5mm);
  \draw (0,0) arc (180 :0: 0.5cm and -0.2cm);

  %  \draw (0,0) rectangle (1,1) coordinate [midway] (A); %node[pos=0.2, font=\tiny] {$a$};
  %\fill (A) circle (0.5mm) node[below] {$R$};
    \filldraw[draw= black!50!black, fill = black] (mr) circle (.5mm);

\end{scope}}}
\]
Manolescu and Neithelath \cite{ManNeith} computed the skein lasagna module of $\ball^2 \times \sphere^2$  in the integral and non-equivariant setting. They showed that it is isomorphic to the ring $\mathbb{Z}[A_0^{\pm 1}, A_1]$. Generator $A_1$ has quantum grading $-2$ while generators $A_0^{\pm 1}$ have quantum grading $0$.
Polynomials in these generators of degree $\alpha$ are in homology grading $\alpha$.
In order to study the $\mathfrak{sl}_2$-action, we will work over $\mathbb{C}[E_1,E_2]$, in that case the skein lasagna module $\mathcal{M}:= S_0^2(\ball^2\times \sphere^2, \emptyset; \mathbb{C}[E_1,E_2])$ is canonically isomorphic to $\mathcal{A}=\mathbb{C}[E_1,E_2][A_0^{\pm 1}, A_1]$.

For $i \geq 0$, one should think of $A_0^i A_1^j \in \mathcal{A}$ as a symmetric tensor of $i$ $A_0$'s and $j$ $A_1$'s. The action of $\mathfrak{sl}_2$ on the generators is given by
\begin{nalign}
  \label{eq:actononall}
    &\de(1)=0, &&&&& &\df(1)=0, &&&&& &\dh(1)=0, \\ 
    &\de(A_1)=0 , &&&&& &\df(A_1)=-\frac{1}{2} E_1 A_1 ,&&&&& &\dh(A_1) = A_1 , \\
    &\de(A_0) = - A_1 , &&&&& &\df(A_0) = \frac{1}{2}E_1 A_0 - E_2 A_1 , &&&&& &\dh(A_0) = -A_0 , \\
    &\de(A_0^{-1}) = A_0^{-2} A_1 , &&&&& &\df(A_0^{-1}) = E_2 A_1 A_0^{-2} - \frac{1}{2} E_1 A_0^{-1} , &&&&& &\dh(A_0^{-1}) = A_0^{-1},
  \end{nalign}
The action comes from \eqref{eq:h-act-pol},
\eqref{eq:h-act-cup}, \eqref{eq:e-act-pol}, \eqref{eq:e-act-cup}, and the fact that $\de$ acts via 
$-\sum_i \frac{\partial}{\partial x_i} $ on
polynomials and by $0$ on any other basic foam. Note that this action is compatible with the algebra structure of $\mathcal{A}.$

From \eqref{eq:actononall}, one obtains that $\mathcal{M}$ splits as a direct sum of $\sll_2$-modules:  $\mathcal{M}\cong \mathcal{M}^+ \oplus \mathcal{M}^-$ with
\[
  \mathcal{M}^+:= \mathbb{C}[E_1,E_2][A_0,A_1], \qquad\text{and}\qquad
  \mathcal{M}^-:= A_0^{-1}\mathbb{C}[E_1,E_2][A_0^{-1},A_1].
\]

\subsubsection{The $\sll_2$-module $\mathcal{M}^+$} Define 
$v=A_0-\frac{1}{2}E_1 A_1$,   
direct calculation shows that 
$\de(v)=0$ and $\df(v)=\frac{1}{2}E_1 v$.  Since $\de(A_1^m)=0$, one has more generally
$\de(A_1^m v^n)=0$.
Notice that
\begin{equation}
    A_1^m v^n - A_1^m A_0^n \in E_1 \mathbb{C}[E_1,E_2][A_0, A_1] \ .
\end{equation}

\begin{lem} \label{lem:fkilln=1}
    If $m \geq 1$ and $m$ is odd, then $\df^m(A_1^m v)=0$.
\end{lem}

\begin{proof}
    First note that
    \[
    \Delta(\df^m)=\sum_{k=0}^m \binom{m}{k} \df^k \otimes \df^{m-k} \ . 
\]
We will also use the identities
\begin{equation} \label{eq:A1squared}
    \df(A_1^2)=-E_1 A_1^2, \quad \df^2(A_1^2)=2E_2 A_1^2, \quad
    \df^3(A_1^2)=0 \ .
  \end{equation}
\allowdisplaybreaks[0]

We proceed by induction where the base case $m=1$ is a quick calculation.
    By induction, assume $\df^{m-2}(A_1^{m-2} v)=0$.
    Then we have 
    \[\df^{m}(A_1^{m} v) = \df^2 \df^{m-2}(A_1^2 A_1^{m-2} v), \] so that $\df^{m}(A_1^{m} v)=$
    \begin{align*}
&   \df^2 \left[ \binom{m-2}{2} \df^2(A_1^2) \df^{m-4}(A_1^{m-2}v)
        + \binom{m-2}{1} \df(A_1^2) \df^{m-3} (A_1^{m-2}v)\right] \\
        &=(\df^2 \otimes 1 + 2 \df \otimes \df + 1 \otimes \df^2) 
        \left[\binom{m-2}{2} \df^2(A_1^2) \df^{m-4}(A_1^{m-2}v)
        + \binom{m-2}{1} \df(A_1^2) \df^{m-3} (A_1^{m-2}v)\right]  
        =0
    \end{align*}
    by the induction hypothesis and \eqref{eq:A1squared}.
    \end{proof}
    \begin{lem}
        If $m \geq n$ and $m-n$ is even, then $\df^{m-n+1}(A_1^m v^n)=0$.
    \end{lem}
\begin{proof}
    We prove this by induction on $n$.  The base case $n=1$ is Lemma \ref{lem:fkilln=1}.  For the other base case we must show
    $\df^{m+1}(A_1^m)=0$ for $m$ even.  We prove this by induction on $m$ where the base case $m=0$ is $\df(1)=0$.  For the inductive step, one has
    \begin{align*}
        \df^{m+1} A_1^m &= \df^2 \df^{m-1}(A_1^2 A_1^{m-2}) \\
        &=(\df^2 \otimes 1 + 2 \df \otimes \df + 1 \otimes \df^2)\left[
        \binom{m-1}{2} \df^2(A_1^2) \df^{m-3}(A_1^{m-2})
        \right. \\
        & \hspace{3cm} \left. +
        \binom{m-1}{1} \df(A_1^2) \df^{m-2}(A_1^{m-2}) +
        \binom{m-1}{0} (A_1^2) \df^{m-1}(A_1^{m-2}) 
        \right] =0
          \end{align*}
          by induction and \eqref{eq:A1squared}. This finishes the verification of the base cases.

          For the inductive step, 
          \begin{align*}
              \df^{m-n+1}(A_1^m v^n) 
              &= \df^{m-n+1}\big((A_1^{m-1}v^{n-1}) A_1 v\big) \\
              &=(\df^{m-n+1} \otimes 1)\big((A_1^{m-1}v^{n-1}) (A_1 v) \big) =0
          \end{align*}
          where the second equality follows from the fact that the other terms in $\Delta(\df^{m-n+1})$ annihilate $A_1v$, and the last equality follows from the induction hypothesis.
\end{proof}
\begin{lem} \label{lem:genfrcomp}
Let $\gamma=\gamma_{j,m,n}=2j+\frac{n-m}{2}$.  Then for all $r \geq 1$, 
    \begin{align*}
    \df^r(\mydiscriminant^j A_1^m v^n) = 
    \sum_{k+2a=r} (-1)^a \frac{(k+2a) \cdots (k+1)}{a!}
    (\gamma)
\cdots (\gamma+r-a-1)
    E_1^k E_2^a \mydiscriminant^j A_1^m v^n \ ,
    \end{align*}
    where the expression $(k+2a) \cdots (k+1)=1$ if $a=0$.
\end{lem}

\begin{proof}
    This is a lengthy, but straightforward induction on $r$.
\end{proof}

\begin{lem} \label{lem:frkill}
If $m-n-4j \geq 0$ and $m-n-4j$ is even, then 
\[\df^{m-n-4j+1}(\mydiscriminant^j A_1^m v^n)=0 .\]
\end{lem}

\begin{proof}
This follows from Lemma \ref{lem:genfrcomp} upon noticing that the coefficients must be zero.  
\end{proof}

\begin{prop} \label{prop:resultforplus}
    As an $\mathfrak{sl}_2$-module, 
    \[
\mathcal{M}^+ \cong
      \bigoplus_{(j,m,n) \in S} M(m-n-4j)
    \oplus 
     \bigoplus_{(j,m,n) \in S^*} M^*(m-n-4j) \ ,
    \]
    where $S^*=\{(j,m,n) \in \mathbb{Z}_{\geq 0}^3\}$ such that 
    $m-n-4j \geq 0$ and $m-n-4j$ is even, and
    $S=\mathbb{Z}_{\geq 0}^3 - S^*$.
\end{prop}

\begin{proof}
First note that the operator $\dh$ acts diagonally on $\mathbb{C}[E_1,E_2]A_1^m v^n$ and its character is given by
\[
\frac{\lambda^{m-n}}{(1-\lambda^{-2})(1-\lambda^{-4})} \ .
\]

    The elements $\mydiscriminant^j A_1^m v^n$ are highest weight vectors for all $(j,m,n) \in \mathbb{Z}_{\geq 0}^3$.
    For $(j,m,n) \in \sphere^*$, let $v_{j,m,n}=\mydiscriminant^j A_1^m v^n$. By Lemma \ref{lem:frkill}, one has $f^r(v_{j,m,n})=0$ for some $r$ and thus generates an irreducible representation $L(m-n-4j)$.  
    We claim that there exists a non-highest weight vector
    $\omega_{j,m,n}$ of weight $n-m+4j-2$ which generates a Verma module $M(n-m+4j-2)$ modulo $L(m-n-4j)$.
    Consider a vector $P \mydiscriminant^j A_1^m v^n $ where $P$ is a polynomial in $E_1, E_2$. If $P \mydiscriminant^j A_1^m v^n $ is a highest weight vector, then $\de(P \mydiscriminant^j A_1^m v^n)=
\de(P) \mydiscriminant^j A_1^m v^n =0 $.
    But from our analysis of the $\mathfrak{sl}_2$-module $\mathbb{C}[E_1,E_2]$, we know that $P$ must be of the form $\mydiscriminant^j$.
    Thus for character reasons, there exists a vector $\omega_{j,m,n}$, such that $v_{j,m,n}$ and $\omega_{j,m,n}$ generate a dual Verma module
    $M^*(m-n-4j)$.

When $(j,m,n) \in S$, $\mydiscriminant^j A_1^m v^n$ generates a Verma module $M(m-n-4j)$.
    
\end{proof}

\subsubsection{The $\sll_2$-module $\mathcal{M}^-$}

As an $\mathfrak{sl}_2$-representation, $\mathcal{M}^-=A_0^{-1}\mathbb{C}[E_1,E_2][A_0^{-1},A_1]$ is neither highest nor lowest weight.  The Casimir $
\Omega=\dh^2+2\dh + 4 \df \de $ 
does not act by a scalar on indecomposables summands. Indeed, one easily computes:
\begin{gather*}
\Omega(A_0^{-1})=3A_0^{-1}-6E_1A_0^{-2}A_1 + 8 E_2 A_0^{-3} A_1^2 \ , \qquad\qquad
\de(A_0^{-1})=
A_1 A_0^{-2} \ , \\
\Omega(A_1 A_0^{-2})
=
15 A_1 A_0^{-2}
-
20 E_1 A_1^2 A_0^{-3}
+
24 E_2 A_1^3 A_0^{-4} \ .
\end{gather*}
Consider a decomposition of $\mathcal{M}^-$: 
\[
\mathcal{M}^- =
\bigoplus_{\ell \in \mathbb{Z}}
\mathcal{M}^-_{\ell} \ ,
\qquad \textrm{with} \qquad
\mathcal{M}^-_{\ell} =
\bigoplus_{m-n=\ell, m \in \mathbb{Z}_{\geq 0}, n \in \mathbb{Z}_{> 0}} \mathbb{C}[E_1,E_2] \langle A_1^m A_0^{-n} \rangle \ .
\]

\begin{lem} \label{lem:Aminusdecom}
The decomposition 
$\mathcal{M}^- =
\bigoplus_{\ell \in \mathbb{Z}}
\mathcal{M}^-_{\ell}$ is a decomposition of $\mathfrak{sl}_2$-representations.
\end{lem}

\begin{proof}
It suffices to check that $\mathcal{M}^-_{\ell}$ is stable under $\de, \df, \dh$ which follows from the computations:
\begin{align}
    \de(A_1^m A_0^{-n}) &= n A_1^{m+1} A_0^{-n-1} \\
    \df(A_1^m A_0^{-n}) &= -(\frac{m+n}{2}) E_1 A_1^m A_0^{-n} +n E_2 A_1^{m+1} A_0^{-n-1} \\
    \dh(A_1^m A_0^{-n}) &= (m+n) A_1^m A_0^{-n} \ .\qedhere
\end{align}
\end{proof}

\begin{lem}
    The representation $\mathcal{M}^-_{\ell}$ is non-Artinian.
\end{lem}

\begin{proof}
    From the formulas in the proof of Lemma \ref{lem:Aminusdecom}, it is easy to see that there is a descending chain of proper submodules of $\mathcal{S}_\ell$
    \begin{equation} \label{eq:defSj}
    \mathcal{S}_{\ell,r}= \bigoplus_{\substack{ 
    m-n=\ell, m \in \mathbb{Z}_{\geq 0}, n \in \mathbb{Z}_{> 0} \\
    m \geq r} }
    \mathbb{C}[E_1,E_2] \langle A_1^m A_0^{-n} \rangle \ . \qedhere
    \end{equation}
\end{proof}

\begin{lem}
    One has the following isomorphisms of $\sll_2$-modules: \[
    \mathcal{S}_{\ell,r} /  \mathcal{S}_{\ell,r+1} \cong 
    \begin{cases}
        \oplus_{j=0}^{\infty} M^*(2r-\ell-4j) & \text{ if } 2r-\ell-4j \geq 0 \text{ and even} \\
       \oplus_{j=0}^{\infty} M(2r-\ell-4j) & \text{ otherwise} 
    \end{cases} \ .
    \]
\end{lem}

\begin{proof}
    This quotient representation has a basis consisting of the images of $P A_1^r A_0^{-({r-\ell})}$ where $P \in \mathbb{C}[E_1,E_2]$. The formulas from Lemma \ref{lem:Aminusdecom} yield that in the quotient:
    \begin{align*}
    \de(A_1^r A_0^{-(r-\ell)}) = 0, \hspace{.1in}
    \df(A_1^r A_0^{-(r-\ell)}) = \left(\frac{\ell-2r}{2}\right)E_1 A_1^rA_0^{-(r-\ell)}, \hspace{.1in}
    \dh(A_1^r A_0^{-(r-\ell)}) = (2r-\ell) A_1^r A_0^{-(r-\ell)} \ .\end{align*}
    Comparing formulas, one sees that this quotient is isomorphic to the representation generated by $A_1^m v^n$ when $m-n=2r-\ell$.
    The lemma now follows from the earlier analysis of that module in the proof of Proposition \ref{prop:resultforplus}.
\end{proof}

\begin{lem}
The Zuckerman functor annihilates $\mathcal{M}^-_{\ell}$.  That is,
   $\Gamma \mathcal{M}^-_{\ell}=0$. 
\end{lem}

\begin{proof}
It is suffices to show that there are no finite-dimensional submodules of $\mathcal{M}^-_{\ell}$.
Assume that there is a finite-dimensional submodule $M$
with basis $\{v_1, \ldots, v_k \}$ where
\[
v_i=a_{1i} A_1^{m_{1i}} A_0^{-{n_{1i}}}
+ \cdots +
a_{ji} A_1^{m_{ji}} A_0^{-{n_{ji}}}
\]
where each of the $a_{ri} \in \mathbb{C}[E_1,E_2]$ and
$m_{1i} < m_{2i} < \cdots < m_{ji}$.

Without loss of generality, assume that $v_k$ has the term with the largest exponent of $A_1$.  Then $\de(v_k)$ cannot be in the submodule.
\end{proof}

\begin{cor}
The Zuckerman functor annihilates the entire module $\mathcal{M}^-$.  That is,
    $\Gamma \mathcal{M}^- =0$.
\end{cor}

\subsubsection{Summary}
As an $\sll_2$-module, the skein lasagna module $\mathcal{M}:=S_0^2(\ball^2\times \sphere^2, \emptyset; \mathbb{C}[E_1,E_2])$ is isomorphic to \[
\mathbb{C}[E_1,E_2][A_0^{\pm 1},A_1]=\mathcal{M}^+ \oplus \mathcal{M}^-
\] 
where
$\mathcal{M}^+=\mathbb{C}[E_1,E_2][A_0^{},A_1]$ and
$\mathcal{M}^-=A_0^{-1} \mathbb{C}[E_1,E_2][A_0^{-1},A_1]$.

As $\mathfrak{sl}_2$-modules,
\[
\mathcal{M}^+ \cong 
\bigoplus_{(j,m,n) \in S} M(m-n-4j)
    \oplus 
     \bigoplus_{(j,m,n) \in S^*} M^*(m-n-4j) \ ,
    \]
where $S^*=\{(j,m,n) \in \mathbb{Z}_{\geq 0}^3\}$ such that 
    $m-n-4j \geq 0$ and $m-n-4j$ is even, and
    $S=\mathbb{Z}_{\geq 0}^3 - S^*$,
and
    \[
    \mathcal{M}^- =
\bigoplus_{\ell \in \mathbb{Z}}
\mathcal{M}^-_{\ell}
    \]
    where $\mathcal{M}^-_{\ell}$ has a filtration
    \[
    \mathcal{M}^-_{\ell}=\mathcal{S}_{\ell,0} \supset 
    \mathcal{S}_{\ell,1} \supset \cdots 
    \]
and the successive quotients are:
 \[
    \mathcal{S}_{\ell,r} /  \mathcal{S}_{\ell,r+1} \cong 
    \begin{cases}
        \oplus_{j=0}^{\infty} M^*(2r-\ell-4j) & \text{ if } 2r-\ell-4j \geq 0 \text{ and even,} \\
       \oplus_{j=0}^{\infty} M(2r-\ell-4j) & \text{ otherwise.} 
    \end{cases} 
    \]
    Finally,
    \[\Gamma \mathcal{M}= \Gamma \mathcal{M}^+ =
    \bigoplus_{(j,m,n) \in S^*} L(m-n-4j) \ .
    \]

\bibliographystyle{alphaurl}
\bibliography{biblio}

\end{document}